\documentclass[a4paper,12pt]{article}

\usepackage{amsmath,amssymb,amsfonts,amsthm}
\usepackage{multirow,bigdelim}
\usepackage{amscd}
\usepackage{verbatim}
\usepackage{latexsym}
\usepackage{array}
\usepackage{enumerate}
\usepackage[all]{xy}
\usepackage{graphicx}
\usepackage{ulem}
\usepackage{ascmac}
\usepackage{mathtools}
\usepackage{layout}
\usepackage{bbm,enumitem,empheq,nicematrix,relsize,calligra}
\usepackage[initials,alphabetic]{amsrefs}
   \BibSpec{arXiv}{
   +{}{\PrintAuthors}{author}
   +{,}{ \textit}{title}
   +{,}{ }{date}
   +{,}{ arXiv:}{eprint}}
\usepackage{mleftright}
   \mleftright
\usepackage{tikz}
\usetikzlibrary{cd,decorations,decorations.pathreplacing,arrows,calc,intersections,hobby,decorations.pathmorphing} 
\usepackage{adjustbox}

\numberwithin{equation}{section}

\newcommand{\BDC}{{\mathbf{D}}^{\mathrm{b}}}

\newcommand{\Hom}{\mathrm{Hom}}
\newcommand{\shom}{{\mathcal{H}}om}

\newcommand{\sect}{\Gamma}
\newcommand{\rsect}{{\mathrm{R}}\Gamma}
\newcommand{\CC}{\mathbb{C}}

\newcommand{\RR}{\mathbb{R}}

\newcommand{\ZZ}{\mathbb{Z}}

\renewcommand{\(}{\left(}
\renewcommand{\)}{\right)}

\newcommand{\Ker}{{\rm Ker}}

\newcommand{\codim}{{\rm codim}}

\newcommand{\id}{{\rm id}}

\newcommand{\Sol}{{\rm Sol}}

\newcommand{\Db}{{\bf D}^{b}}
\newcommand{\Dbc}{{\bf D}_{c}^{b}}

\newcommand{\tl}[1]{\widetilde{#1}}

\newcommand{\simto}{\overset{\sim}{\longrightarrow}}
\newcommand{\simot}{\overset{\sim}{\longleftarrow}}

\newcommand{\SD}{\mathcal{D}}

\newcommand{\SM}{\mathcal{M}}
\newcommand{\SN}{\mathcal{N}}
\newcommand{\SL}{\mathcal{L}}

\newcommand{\SG}{\mathcal{G}}

\newcommand{\rhom}{{\rm R}{\mathcal{H}}om}

\newcommand{\bfD}{\mathbf{D}}
\newcommand{\rmR}{{\mathrm{R}}}

\renewcommand{\Re}{\operatorname{Re}}

\newcommand{\cS}{\mathcal{S}}

\newcommand{\ST}{\mathcal{T}}

\newcommand{\rmd}{\mathrm{d}}
\newcommand{\bfA}{\mathbf{A}}
\newcommand{\bfB}{\mathbf{B}}

\newcommand{\pt}{\mathrm{pt}}
\renewcommand{\Ker}{\operatorname{Ker}}
\newcommand{\Ima}{\operatorname{Im}}
\newcommand{\Coker}{\operatorname{Coker}}

\newcommand{\C}{\mathrm{C}}
\newcommand{\reg}{\mathrm{reg}}

\newcommand{\sing}{\mathrm{sing}}
\newcommand{\BM}{\mathrm{BM}}
\newcommand{\infsupp}{\mathit{inf}}

\newcommand{\orsh}{\text{\raisebox{0.2mm}{\Large\textcalligra{or}}}\!\;} %% orientation sheaf in KS
\newcommand{\iorsh}{\text{\raisebox{0.2mm}{\Large\textcalligra{or}}}\!\:} %% orientation sheaf in KS (inline)
\newcommand{\ftimes}[1]{\underset{#1}{\times}} %% fiber product
\newcommand{\BDrc}{\bfD_{\RR\text{-}\mathrm{c}}^{\mathrm{b}}} %% R-constructible sheaf
\newcommand{\BDc}{\bfD_{\mathrm{c}}^{\mathrm{b}}} %% C-constructible sheaf
\newcommand{\muhom}{\mu hom}
\newcommand{\fboxtimes}[1]{\underset{#1}{\boxtimes}} %% fiber product

\DeclareMathOperator{\supp}{supp}

\DeclareMathOperator{\CCyc}{CC} %% Characteristic cycle
\DeclareMathOperator{\Eu}{Eu} %% Euler obstruction
\DeclareMathOperator{\LCLS}{LCLS} %% Locally closed subanalytic subsets
\DeclareMathOperator{\msupp}{SS} %% Microsupport

\newcommand{\vbar}{\left.\right|} 
\DeclareRobustCommand{\longtwoheadrightarrow}{\relbar\joinrel\twoheadrightarrow} 
\newcommand{\longhookrightarrow}{\lhook\joinrel\longrightarrow}
\newcommand{\tikzlim}[1]{\lim\limits_{\scalebox{0.7}{$#1$}}}

\DeclarePairedDelimiter{\abs}{\lvert}{\rvert} 
\DeclarePairedDelimiterX{\Set}[2]{\lbrace}{\rbrace}{#1\ \delimsize\vert\ #2}

\newtheorem{theorem}{Theorem}[section]

\newtheorem{corollary}[theorem]{Corollary}
\newtheorem{lemma}[theorem]{Lemma}
\newtheorem{proposition}[theorem]{Proposition}

\theoremstyle{definition}
\newtheorem{definition}[theorem]{Definition}
\theoremstyle{remark}
\newtheorem{remark}[theorem]{\sc Remark}
\newtheorem{example}[theorem]{\sc Example}
\title{Characteristic cycles 
of real and complex \\ constructible sheaves, revisited 
}

\author{Ren FERNANDES 
\footnote{Mathematical Institute, Tohoku University,
Aramaki Aza-Aoba 6-3, Aobaku, Sendai, 980-8578, Japan.
E-mail: fernandes.ren.p7@dc.tohoku.ac.jp}, 
Kazuki KUDOMI 
\footnote{Mathematical Institute, Tohoku University,
Aramaki Aza-Aoba 6-3, Aobaku, Sendai, 980-8578, Japan.
E-mail: kazuki.kudomi.q3@dc.tohoku.ac.jp}
and Kiyoshi TAKEUCHI 
\footnote{Mathematical Institute, Tohoku University,
Aramaki Aza-Aoba 6-3, Aobaku, Sendai, 980-8578, Japan.
E-mail: takemicro@nifty.com} }

\begin{document}

\maketitle
\begin{abstract}
For a smooth morphism i.e. a submersion $f: X \longrightarrow \Sigma$ 
of real analytic manifolds and 
an $\RR$-constructible sheaf $F$ on $X$ satisfying some condition, 
we define a family of Lagrangian cycles 
parameterized by $\Sigma$ that we call the relative 
characteristic cycle of $F$ for $f$. In this way, the theory of 
characteristic cycles introduced by Kashiwara 
(and developed by Kashiwara and Schapira)  
is naturally extended to the relative setting. Based on 
it, we then prove a formula for the characteristic cycles 
of real nearby cycle sheaves. This leads us to obtain also 
formulas for the characteristic cycles of various 
constructible sheaves, such as specialization, 
microlocalization, and complex nearby and vanishing 
cycle sheaves, in a unified manner. In fact, 
our methods allow us to calculate not only their 
characteristic cycles but also their microlocal types 
in many situations. We will illustrate it by various examples. 
\end{abstract}

\section{Introduction}
Characteristic cycles are basic invariants of 
$\SD$-modules and constructible sheaves and play an important role 
not only in $\SD$-module theory but also in algebraic geometry,
singularity theory, topology, representation theory and 
arithmetic geometry.
On one hand, in \cite{KS85} and 
\cite[Chapter I\hspace{-1.2pt}X]{KS90}, for $\RR$-constructible 
sheaves Kashiwara and Schapira defined their characteristic cycles 
and proved many beautiful results on them, such as their functorial 
properties and index theorems.
On the other hand, in \cite{Gin86} Ginsburg studied the 
characteristic cycles of regular holonomic $\SD$-modules 
and obtained remarkable formulas for the characteristic cycles of
localized and nearby cycle $\SD$-modules.
See also Brian\c{c}on-Maisonobe-Merle \cite{BMM94} for the further 
development on them.
As the methods of \cite{Gin86} and \cite{BMM94} are purely algebraic,
until now there has been no way to recover all the 
results in \cite{Gin86} and \cite{BMM94} by the sheaf theoretical 
methods in \cite{KS90}. This prevents us from extending them to more general situations.
Note that however in \cite[Theorem 4.2]{SV96} Schmid and Vilonen proved 
their open embedding theorem for $\RR$-constructible sheaves and 
solved this problem satisfactorily at least for Ginsburg's formula
for localized $\SD$-modules in \cite[Theorem 3.2]{Gin86}.
\par 
In this paper, we will recover the remaining 
results in \cite{Gin86} and \cite{BMM94} by the 
methods in \cite{KS90} in a unified manner and extend them to 
more general situations. First, 
by generalizing the characteristic cycles of Kashiwara and Schapira 
in \cite[Chapter I\hspace{-1.2pt}X]{KS90} to relative situations, we introduce a new 
notion of relative characteristic cycles (see \cite{FMFS21} and 
\cite{SS94} etc. for the related results on $\SD$-modules in 
the relative setting). For this purpose, we start by recalling some basic properties 
of Borel-Moore homology cycles and reformulate some results of \cite{SV96} 
in the language of \cite[Chapter I\hspace{-1.2pt}X]{KS90}. Note that 
some of them will be used for the study of 
irregular holonomic $\SD$-modules in \cite{KT25}. 
Then by Ishimura's functor of relative $\mu hom$ introduced in 
\cite{Ish92}, for a smooth morphism i.e. a submersion $f\colon X\longrightarrow \Sigma$
of real analytic manifolds and an $\RR$-constructible sheaf 
$F\in \Db_{\RR -\mathrm{c}} (X)$ on $X$ adapted to $f$ (see Definition \ref{def-adapted}), in 
Section \ref{sec-relCC} we define the relative characteristic cycle 
$\CCyc_{\Sigma} (F)$ of $F$ for $f$ as a Borel-Moore homology cycle 
in the relative cotangent bundle $T^\ast (X/\Sigma)$
of $f\colon X\longrightarrow \Sigma$.
Since for any point $a\in \Sigma$ of $\Sigma$ we have an isomorphism 
\begin{equation}
    f^{-1}(a)  \times_X T^\ast(X/\Sigma) \simto T^\ast(f^{-1}(a)),
\end{equation}
for $F_a \coloneq F|_{f^{-1}(a)} \in \Db_{\RR -\mathrm{c}}(f^{-1}(a)) \: (a\in \Sigma)$
our relative characteristic cycle $\CCyc_{\Sigma}(F)$ can be 
(at least locally) regarded as the family 
$\{ \CCyc(F_a)\}_{a\in \Sigma}$ of 
Lagrangian cycles (see Theorem \ref{thm-restofrelCC}). 
For the other variants of characteristic cycles, see 
\cite{MT10} and \cite{SS94} etc. 
\par 
Next, we apply our new theory of relative characteristic cycles 
to obtain also formulas for the characteristic cycles of various 
constructible sheaves, such as specialization, 
microlocalization, and complex nearby and vanishing 
cycle sheaves, in a unified manner. 
A key point in our study is the use of the real nearby cycle functor 
defined as follows.
For $\varepsilon >0$ let $f\colon X\longrightarrow (-\varepsilon, \varepsilon)$
be a smooth morphism of real analytic manifolds and set $X_0\coloneq 
f^{-1}(0) \subset X, \: \{ f>0\}\coloneq \{ x\in X\: \mid \: f(x)>0\} \subset X$.
Let $i_X \colon X_0 \longhookrightarrow X$ be the inclusion map.
Then for an $\RR$-constructible sheaf $F\in \Db_{\RR -\mathrm{c}}(X)$
on $X$ we set 
\begin{equation}
    \psi_f^{\RR} (F)\coloneq i_X^{-1}\rsect_{\{ f>0\}}(F) \quad \in \Db_{\RR -\mathrm{c}}(X_0)
\end{equation}
and call it the real nearby cycle sheaf of $F$ along $f$.
Note that in \cite{NS22} Nadler and Shende also studied such sheaves 
and obtained a formula for their micro-supports.
Here, we go one step further and prove the following result on 
their characteristic cycles.
For $a\in I\coloneq (0,\varepsilon) \subset \RR$ set $X_a \coloneq f^{-1}(a) \subset X$
and $F_a \coloneq F|_{X_a} \in \Db_{\RR -\mathrm{c}}(X_a)$ 
and let $\CCyc(F_a)$ be the characteristic cycle of $F_a$ in $T^\ast X_a =
T^\ast (f^{-1}(a)) \subset T^\ast(X/{(-\varepsilon, \varepsilon)})$
(for the definition, see \cite[Section 9.4]{KS90}).
Then we can show that the family $\{ \CCyc(F_a)\}_{a\in I}$ of 
Lagrangian cycles has a limit $\displaystyle \lim_{a\to +0} \CCyc(F_a)$ (locally)
in the sense of \cite{SV96} (see Section \ref{subsec-limit} for the definition) and obtain the 
following theorem.
\begin{theorem}[see {Theorem \ref{thm-limformula}}]\label{thm-limit}
    In the situation as above, we have 
    \begin{equation}
        \CCyc(\psi_f^{\RR}(F)) = \lim_{a\to +0}\CCyc(F_a)
    \end{equation}
    in $T^\ast X_0$, where this equality holds true only over each relatively compact
    open subset of $X_0$.
\end{theorem}
For the proof of Theorem \ref{thm-limit}, we use some arguments similar 
to the ones in the proof of \cite[Theorem 4.2]{SV96}, but our 
proof is more sheaf theoretical in the sense that we study the 
deformations of objects in derived categories by using the 
non-characteristic deformation lemma in \cite[Proposition 2.7.2]{KS90}.
At many points, instead of the stratified Morse theory of Goresky-MacPherson \cite{GM88},
we also make use of the results of the microlocal sheaf theory in \cite{KS90} 
as much as possible. Then eventually, to our surprise,
we find that our proof of Theorem \ref{thm-limit} allows us to 
calculate not only the characteristic cycle but also the microlocal 
types of $\psi_f^{\RR}(F) \in \Db_{\RR-\mathrm{c}}(X_0)$
in many situations (see Example \ref{ex-realnearby} and 
Theorem \ref{thm-perverse} for the details).
As an immediate consequence of Theorem \ref{thm-limit}, 
we obtain also formulas for the characteristic cycles of 
specialization and microlocaization sheaves, which have never been 
stated clearly before.
See Section \ref{sec-realnearby} for the details.
\par 
With Theorem \ref{thm-limit} and its proof at hand, in Section \ref{sec-cc} 
we then proceed to obtain formulas for the characteristic cycles of 
complex nearby and vanishing cycle sheaves.
For this purpose, we first show in Lemma \ref{lem-nearby} that 
the complex nearby cycle sheaves (of complex constructible sheaves)
can be expressed by our ``real'' nearby cycle functor and apply 
Theorem \ref{thm-limit}.
Note that the resulting formula in Theorem \ref{thm-Gins} 
for the characteristic cycles of complex nearby cycle sheaves 
can be seen as a topological analogue of Ginsburg's one 
in \cite[Theorem 5.5]{Gin86}. 
Finally in Section \ref{sec-cc}, by various examples we will show that 
Theorem \ref{thm-limit} and its proof allow us to calculate 
the characteristic cycles of complex nearby and vanishing cycle 
sheaves explicitly (for their applications to 
singularity theory, see the forthcoming paper \cite{FT26}). 
Note that also in \cite{Mas25} Massey studied the same problem 
and obtained various formulas, similar to but different from ours.
It would be an interesting problem to study the relation between these two methods. 

\bigskip
\noindent{\bf Acknowledgement:}
The authors are grateful to Professors Satoshi Koike, Adam Parusi\'{n}ski and Pierre Schapira for
several discussions with them during the preparation of this paper.

\section{Preliminary notions and results}
In this section, we recall the results of Kashiwara-Schapira \cite[Chapter I\hspace{-1.2pt}X]{KS90} 
and Schmid-Vilonen \cite{SV96} and clarify the relations among them. 
In particular, we reformulate the results of \cite{SV96} in terms of 
the theory in \cite[Chapter I\hspace{-1.2pt}X]{KS90} and give some variants of them,  
which will be used in this paper. 
Let $k$ be a commutative field of characteristic zero.
In what follows, we assume that $k=\CC$ for the sake of simplicity.

\subsection{Borel-Moore homology cycles in the subanalytic setting}\label{subsec.5.1}
In this subsection, we briefly recall the theory of Borel-Moore homology groups 
and cycles in the subanalytic setting.
For the details, see Kashiwara-Schapira \cite[Section 9.2]{KS90}.
Let $X$ be a real analytic manifold of dimension $n$. 
For a non-negative integer $p$, let $\LCLS_p(X)$ denote the set of 
locally closed subanalytic subsets of $X$ of dimension $\leq p$.
Then for any $S\in\LCLS_p(X)$ we have 
\begin{equation}
H^j\omega_S\simeq H^j\rsect_S(\orsh_X[n])\vbar_S\simeq0 \quad (j<-p)
\end{equation}
(see e.g. \cite[Proposition 9.2.2 (ii)]{KS90}).
For the inclusion map $j_S\colon S\longhookrightarrow X$ we thus obtain isomorphisms
\begin{equation}
H_S^{n-p}(X;\orsh_X) \simeq H_S^{-p}(X;\omega_X) 
\simeq \Gamma(X;j_{S\ast}H^{-p}\omega_S).
\end{equation}
On the other hand, by the Poincar{\'e}-Verdier duality theorem, for $j\in \ZZ$ there
exist isomorphisms 
\begin{equation}
H_S^{n-j}(X;\orsh_X)\simeq H^{-j}(S;\omega_S)\simeq \bigl[H_{\mathrm{c}}^j(S;\CC_S)\bigr]^\ast
\end{equation}
and we call the dual vector space $\bigl[H_{\mathrm{c}}^j(S;\CC_S)\bigr]^\ast$ of 
$H_{\mathrm{c}}^j(S;\CC_S)$ the Borel-Moore homology group of $S$ of degree $j$ 
and denote it by $H_j^\BM(S;\CC)$.
Note that by \cite[Proposition 9.2.2(i)]{KS90} we have
\begin{equation}
H_j^\BM(S;\CC)=\bigl[H_{\mathrm{c}}^j(S;\CC_S)\bigr]^\ast\simeq0 \quad (j>p).
\end{equation}
For $S\in\LCLS_p(X)$ we define a closed subanalytic subset 
$\partial S\subset\overline{S}$ of the closure $\overline{S}$ of S by
\begin{equation}
\partial S\coloneq \overline{S}\setminus S.
\end{equation}
Then we obtain a distinguished triangle 
\begin{equation}
\omega_{\partial S}\longrightarrow\omega_{\overline{S}}\longrightarrow
\rmR j_{S\ast}\omega_S\overset{+1}{\longrightarrow}
\end{equation}
and hence the long exact sequence
\begin{equation}
0\longrightarrow H^{-p}\omega_{\overline{S}}\longrightarrow j_{S\ast}H^{-p}\omega_S
\longrightarrow H^{1-p}\omega_{\partial S}\longrightarrow\cdots
\end{equation}
associated to it.
In \cite[Section 9.2]{KS90} Kashiwara and Schapira used the morphism 
$j_{S\ast}H^{-p}\omega_S\longrightarrow H^{1-p}\omega_{\partial S}$ to construct the boundary
operator of subanalytic $p$-chains on $X$.
In order to clarify the meaning of their construction, assume that $S$ is closed and 
fix a subanalytic triangulation $\ST$ of $S$.
For $0\leq k\leq p$ let $S_k\subset S$ be the (disjoint) union of the interiors of the
simplices in $\ST$ of dimension $\leq k$ and set $S_k^\circ\coloneq S_k\setminus S_{k-1}$.
Namely $S_k$ is the $k$-skeleton of $S$.
Then $S_k$ is a closed subset of $S$ and we obtain a distinguished triangle
\begin{equation}\label{eq-dtomega}
\omega_{S_{k-1}}\longrightarrow\omega_{S_k}\longrightarrow\rmR(j_{S_k^\circ})_\ast
\omega_{S_k^\circ}\overset{+1}{\longrightarrow} 
\end{equation}
and the long exact sequence
\begin{equation}
0\longrightarrow H^{-k}\omega_{S_k}\longrightarrow
(j_{S_k^\circ})_\ast H^{-k}\omega_{S_k^\circ}\longrightarrow
H^{-(k-1)}\omega_{S_{k-1}}\longrightarrow\cdots
\end{equation}
associated to it, where $j_{S_k^\circ}\colon S_k^\circ\longhookrightarrow X$ is the 
inclusion map.
In particular, for $k=p$ we have the equality $S_p=S$ and by our assumption that $S$
is closed in $X$ we obtain an exact sequence 
\begin{equation}\label{eq-exBM}
0\longrightarrow H_p^\BM(S;\CC)\longrightarrow\Gamma(S_p^\circ;\orsh_{S_p^\circ})
\longrightarrow\Gamma(S_{p-1}^\circ;\orsh_{S_{p-1}^\circ}).
\end{equation}
Here we used the fact that the natural morphism 
\begin{equation}
\Gamma(X;H^{-(p-1)}\omega_{S_{p-1}})\longrightarrow
\Gamma(S_{p-1}^\circ;\orsh_{S_{p-1}^\circ})
\end{equation}
is injective.
For $0\leq k\leq p$ let $C_k^\infsupp(S, \ST ;\CC)$ be the $\CC$-vector space 
generated by the simplicial $k$-chains on $S$ of possibly 
non-compact support 
with respect to the triangulation $\ST$.
Then we obtain a complex 
\begin{equation}
\cdots\longrightarrow C_j^\infsupp(S, \ST ;\CC) \underset{\partial_j}{\longrightarrow}
C_{j-1}^\infsupp(S, \ST ;\CC)\underset{\partial_{j-1}}{\longrightarrow}
C_{j-2}^\infsupp(S, \ST ;\CC)\longrightarrow\cdots
\end{equation}
of $\CC$-vector spaces and set
\begin{equation}
H_k^\infsupp(S, \ST ;\CC)\coloneq \frac{\Ker\partial_k}{\Ima\partial_{k+1}} 
\quad (0\leq k \leq p).
\end{equation}
Thus the exact sequence \eqref{eq-exBM} implies that for the Borel-Moore homology
group 
$H_p^\BM(S;\CC)=\bigl[H_{\mathrm{c}}^p(S;\CC_S)\bigr]^\ast\simeq
H_S^{n-p}(X;\orsh_X)$ of $S$ of the highest degree $p$ there exists an isomorphism
\begin{equation}
H_p^\BM(S;\CC)\simeq H_p^\infsupp(S, \ST ;\CC).
\end{equation}
From this, we can easily show that the notion of subanalytic $p$-cycles in $S$ used 
in Schmid-Vilonen \cite[Section 3]{SV96} is equivalent to the one of 
Kashiwara-Schapira \cite[Section 9.2]{KS90}.
More generally, we have the following well-known result. 
Here we give a short proof to it for the reader's convenience. 

\begin{proposition}
For any $0\leq k\leq p$ there exists an isomorphism
\begin{equation}
H_k^\BM(S;\CC)\simeq H_k^\infsupp(S, \ST ;\CC).
\end{equation}
\end{proposition}

\begin{proof}
First, note that for any $0\leq k\leq p$ we have a concentration 
\begin{equation}
H^j(X;\rmR(j_{S_k^\circ})_\ast\omega_{S_k^\circ})\simeq
H^{j+k}(S_k^\circ;\orsh_{S_k^\circ})\simeq0 \quad (j\neq-k).
\end{equation}
Then by a repeated use of the distinguished triangle \eqref{eq-dtomega} we can easily
show that for any $0\leq k\leq p$ we have an isomorphism
\begin{equation}
H_{k-1}^\BM(S;\CC)\simeq H^{-(k-1)}(X;\omega_S)\simeq H^{-(k-1)}(X;\omega_{S_k}).
\end{equation}
On the other hand, for $0\leq k\leq p$ there exists an exact sequence
\begin{equation}
0\longrightarrow\Gamma(S_k^\circ;\orsh_{S_k^\circ})\longrightarrow
H^{-(k-1)}(X;\omega_{S_{k-1}})\longrightarrow H^{-(k-1)}(X;\omega_{S_k})
\longrightarrow0.
\end{equation}
As the $(k-1)$-skeleton $S_{k-1}$ of $S$ contains only the simplices of 
dimension $\leq k-1$, we have isomorphisms 
\begin{equation}
H^{-(k-1)}(X;\omega_{S_{k-1}})\simeq H_{k-1}^\infsupp 
(S_{k-1}, \ST |_{S_{k-1}} ;\CC)
\simeq\Ker\partial_{k-1}.
\end{equation}
We thus obtain the assertion as follows:
\begin{equation}
H_{k-1}^\BM(S;\CC) \simeq 
\Coker\bigl[\Gamma(S_k^\circ;\orsh_{S_k^\circ})\longrightarrow\Ker\partial_{k-1}\bigr] 
\simeq H_{k-1}^\infsupp(S, \ST ;\CC).
\end{equation}
\end{proof}
If $S\subset X$ is a (purely) $p$-dimensional orientable subanalytic submanifold, then
we have isomorphisms
\begin{equation}
H_p^\BM(S;\CC)\simeq \Gamma(S;\orsh_S)\simeq\Gamma(S;\CC_S)
\end{equation}
and hence the top-dimensional Borel-Moore homology group 
$H_p^\BM(S;\CC)\simeq H_S^{n-p}(X;\orsh_X)$ of $S$ contains an element 
$[S]\in H_p^\BM(S;\CC)$ which corresponds to the one $1\in\Gamma(S;\CC_S)$.
We call it the fundamental class of $S$.
As an element of the dual vector space 
$\bigl[H_{\mathrm{c}}^p(S;\CC_S)\bigr]^\ast$ of $H_{\mathrm{c}}^p(S;\CC_S)$,
it corresponds to the $\CC$-linear morphism
\begin{equation}
\int_S\colon H_{\mathrm{c}}^p(S;\CC_S)\longrightarrow\CC
\end{equation}
defined by the integral of smooth $p$-forms of compact support over $S$ 
(with respect to the orientation of $S$ used to construct the isomorphism $\iorsh_S\simeq \CC_S$).
From now, let us consider the special case where $X$ (resp. $S\subset X$) is a smooth
complex algebraic variety of dimension $m$ 
(resp. a complex algebraic subset of $X$ of dimension $q$).
First, assume that $S$ is irreducible and let $S_\reg\subset S$ be the smooth
part of $S$.
Then $S_\reg$ is a complex manifold of dimension $q$ and hence orientable.
We thus obtain its fundamental class
\begin{equation}
[S_\reg]\in H_{2q}^\BM(S_\reg;\CC)\simeq H_{S_\reg}^{2m-2q}(X;\orsh_X). 
\end{equation}
Moreover, as the singular part $S_\sing\coloneq S\setminus S_\reg$ of $S$ is of (complex)
dimension $\leq q-1$, the restriction morphism
\begin{equation}
H_{2q}^\BM(S;\CC)\longrightarrow H_{2q}^\BM(S_\reg;\CC)
\end{equation}
of the top-dimensional Borel-Moore homology groups is an isomorphism.
This implies that 
there exists a unique element $[S]\in H_{2q}^\BM(S;\CC)$ which corresponds to
the one $[S_\reg]\in H_{2q}^\BM(S_\reg;\CC)$.
We call it the fundamental class of $S$.
In the general case, let $T_1,T_2,\dots,T_l$ be the $q$-dimensional 
irreducible components of $S$ and for each $1\leq i\leq l$ denote by 
$[T_i]\in H_{2q}^\BM(S;\CC)$ the image of the fundamental class of $T_i$ by the morphism
$H_{2q}^\BM(T_i;\CC)\longrightarrow H_{2q}^\BM(S;\CC)$.
Then it is well-known that the top-dimensional Borel-Moore homology group 
$H_{2q}^\BM(S;\CC)\simeq H_S^{2m-2q}(X;\orsh_X)$ of $S$ is a $\CC$-vector 
space of dimension $l$ and the elements $[T_1],[T_2],\dots,[T_l]\in H_{2q}^\BM(S;\CC)$
form a basis of it (see e.g. Chriss-Guisburg \cite[Proposition 2.6.14]{CG97} 
and Fulton \cite[Section B.3, Lemma 4]{Ful97}).
Now we return to the general case where $X$ is a real analytic manifold of dimension $n$
and $S\in \LCLS_p(X)$. 
Let $L$ be a local system on $X$ over the field $\CC$. 
Then the canonical morphism 
\begin{align}
&\rsect_S(\orsh_X)\otimes L\simeq \rhom_{\CC_X}(\CC_S,\orsh_X)\otimes L \\
&\longrightarrow \rhom_{\CC_X}(\CC_S,\orsh_X \otimes L)\simeq \rsect_S(\orsh_X\otimes L)
\end{align} 
is an isomorphism.
We define the twisted Borel-Moore homology group $H_p^\BM(S;L)$ of $S$ with coefficients 
in the local system $L$ by 
\begin{equation}
H_p^\BM(S;L) \coloneq H_S^{n-p}(X;\orsh_X\otimes L) 
\simeq H_S^{-p}(X;\omega_X\otimes L)\simeq\Gamma(S;H^{-p}\omega_S\otimes(L\vbar_S)).
\end{equation}
In particular, for a subanalytic submanifold $S\subset X$ of dimension $p$,
we then obtain an isomorphism
\begin{equation}
H_p^\BM(S;L)\simeq\Gamma(S;\orsh_S \otimes(L\vbar_S)).
\end{equation} 
Even if the submanifold $S\subset X$ is not orientable, at least in the case where
there exists a local system $L$ on $X$ such that $L\vbar_S\simeq \orsh_S$ and hence
\begin{equation}
\Gamma(S;\orsh_S\otimes(L\vbar_S))\simeq\Gamma(S;\CC_S)\simeq\CC
\end{equation}
we can define a ``fundamental class" $[S]$ of $S$ to be the element of $H_p^\BM(S;L)$ 
which corresponds to the one $1\in\CC\simeq\Gamma(S;\iorsh_S\otimes(L\vbar_S))$.
As we see in the next example, it seems that this observation was a key in the 
construction of the theory of characteristic cycles in 
Kashiwara-Schapira \cite[Section 9.3]{KS90}.

\begin{example}
Let $X$ be a real analytic manifold of dimension $n$ and $Y\subset X$ a submanifold
of dimension $m$ $(\leq n)$ and set $S\coloneq T_Y^\ast X\subset T^\ast X$.
Note that the cotangent bundle $T^\ast X$ of $X$ is orientable and hence 
$\orsh_{T^\ast X}\simeq \CC_{T^\ast X}$.
This implies that for the inclusion map $k\colon S=T_Y^\ast X\longhookrightarrow T^\ast X$
we have isomorphisms
\begin{equation}
\omega_S\simeq k^!\CC_{T^\ast X}[2n] 
\simeq\rsect_{T_Y^\ast X}(\CC_{T^\ast X}[2n])|_{T_Y^\ast X}.
\end{equation}
Let $i\colon S=T_Y^\ast X\longhookrightarrow Y\ftimes{X} T^\ast X$ and 
$j\colon Y\ftimes{X} T^\ast X\longhookrightarrow T^\ast X$ be the inclusion maps so
that we have $k=j\circ i$.
Let $\pi_X\colon T^\ast X\longrightarrow X$ be the canonical projection and consider 
the following Cartesian square 
\begin{equation}
\vcenter{
\xymatrix@M=5pt{
Y\ftimes{X} T^\ast X \ar[d]^-{\alpha} \ar@{^{(}->}[r]^-{j} \ar@{}[dr]|\square & 
T^\ast X \ar[d]^-{\pi_X} & \\
Y \ar@{^{(}->}[r] & X. &  
}}
\end{equation}
As the morphism $\pi_X$ is a submersion, we obtain isomorphisms
\begin{align}
j^!\CC_{T^\ast X}[2n] &\simeq 
\rsect_{Y\ftimes{X}T^\ast X}(\CC_{T^\ast X}[2n])|_{Y\ftimes{X}T^\ast X} \\
&\simeq \alpha^{-1}\Bigl(\rsect_Y(\CC_X[2n])\vbar_Y\Bigr) 
\simeq \alpha^{-1}\orsh_{Y/X}[n+m].
\end{align}
Let $\rho\colon Y\ftimes{X}T^\ast X\longtwoheadrightarrow T^\ast Y$ be the morphism
induced by the inclusion map $Y\longhookrightarrow X$ and 
$\iota_Y\colon Y\simeq T_Y^\ast Y\longhookrightarrow T^\ast Y$ the zero section embedding.
Then we obtain a Cartesian square
\begin{equation}
\vcenter{
\xymatrix@M=5pt{
S=T_Y^\ast X \ar[d]^-{\beta} \ar@{^{(}->}[r]^-{i} \ar@{}[dr]|\square & 
Y\ftimes{X} T^\ast X \ar[d]^-{\rho} & \\
Y \ar@{^{(}->}[r]^-{\iota_Y} & T^\ast Y &  
}}
\end{equation}
and can use it similarly to show isomorphisms 
\begin{equation}
i^!\alpha^{-1}\orsh_{Y/X}[n+m] \simeq 
\beta^{-1}\Bigl\{\bigl(\rsect_Y(\CC_{T^\ast Y})\vbar_Y\bigr)\otimes\orsh_{Y/X}[n+m]\Bigr\} 
\simeq \beta^{-1}(\orsh_X^{\otimes{-1}}\vbar_Y)[n].
\end{equation}
For the submanifold $S=T_Y^\ast X\subset T^\ast X$ we thus obtain isomorphisms
\begin{equation}
\orsh_S \simeq \beta^{-1}(\orsh_X^{\otimes{-1}} \vbar_Y) 
\simeq ( \pi_X^{-1}\iorsh_X) \vbar_S
\end{equation}
and 
\begin{equation}
H_n^\BM(S;\CC)\simeq \Gamma\bigl(S; ( \pi_X^{-1}\iorsh_X) \vbar_S \bigr).
\end{equation}
This implies that if $\Gamma(S; ( \pi_X^{-1}\orsh_X) \vbar_S )\simeq0$ we can 
not define a non-trivial class in $H_n^\BM(S;\CC)$.
Nevertheless, for the local system $L=\pi_X^{-1}\iorsh_X\simeq 
\orsh_{T^*X /X}$
on $T^\ast X$ we have 
\begin{equation}
H_n^\BM(S;L) \simeq \Gamma(S;\CC_S)\simeq \CC
\end{equation}
and hence can define a fundamental class $[S]$ of $S$ in 
the twisted Borel-Moore homology group $H_n^\BM(S;L)$.
This would be one of the reasons why Kashiwara and Schapira used $\pi_X^{-1}\omega_X$
instead of $\omega_{T^\ast X}$ to define their sheaf of Lagrangian cycles in 
\cite[Definition 9.3.1]{KS90}.
\end{example}

\subsection{Limits of Borel-Moore homology cycles and their properties}\label{subsec-limit} 
As in Schmid-Vilonen \cite[Section 3]{SV96}, for some $b>0$ 
we set $I\coloneq (0,b)$ and $J\coloneq [0,b)$.
Let $M$ be a smooth manifold and set $M_J\coloneq M\times J$, $M_I\coloneq M\times I$ and
$M_{\{0\}}\coloneq M\times \{0\}$.
Let $A_J\subset M\times J$ be a closed subset and set 
$A_I\coloneq A_J\cap M_I\subset M_I$ and 
$A_{\{0\}}\coloneq A_J\cap M_{\{0\}}\subset M_{\{0\}}\simeq M$.
Moreover, for a local system $L_J$ on $M_J$ over the field $\CC$ we set 
\begin{equation}
L_I\coloneq L_J\vbar_{M_I}, \quad 
L_{\{0\}}\coloneq L_J\vbar_{M_{\{0\}}}. 
\end{equation}
Then for the inclusion map $j_{M_I}\colon M_I\longhookrightarrow M_J$
there exists an isomorphism
\begin{equation}
L_J\simto \rmR(j_{M_I})_\ast L_I
\end{equation}
and hence applying the functor $\rsect_{A_J}(M_J;\cdot)$ to it we obtain 
an isomorphism
\begin{equation}
\rsect_{A_J}(M_J;L_J)\simto \rsect_{A_I}(M_I;L_I).
\end{equation}
Composing its inverse with the natural morphism 
\begin{equation}
\rsect_{A_J}(M_J;L_J)\longrightarrow \rsect_{A_{\{0\}}}(M;L_{\{0\}}),
\end{equation}
we obtain the morphism 
\begin{equation}
\Phi\colon \rsect_{A_I}(M_I;L_I)\longrightarrow 
\rsect_{A_{\{0\}}}(M;L_{\{0\}})
\end{equation}
of Schmid-Vilonen \cite[(3.14)]{SV96} in a slightly different setting. 
In \cite{SV96} Schmid-Vilonen used it to construct the limits $C_{\{ 0 \}}$ of 
families $C_I$ of Borel-Moore homology cycles in $M$ 
parameterized by $t \in I=(0,b)$.
Before explaining their construction, we shall give some variants of the morphism
$\Phi$ and explain relations among them (see Lemmas \ref{lem-pushPhi}, \ref{lem-pullPhi} 
and \ref{new-lemmes} below). 
Let $N$ be a smooth manifold, for which we use 
the same notations $N_I, N_J$ and $N_{\{0\}}$ as for $M$.
First, let $f\colon M\longrightarrow N$ be a proper morphism of smooth manifolds.
Then for a local system $L_J$ on $M_J$ and a closed subset $A_J \subset M_J$, 
there exists a natural morphism 
\begin{equation}
\phi_\ast(f)\colon \rsect_{A_{\{ 0 \}}}(M;L_{\{ 0 \}})\longrightarrow
\rsect_{f(A_{\{ 0 \}} )}(N;\rmR f_\ast L_{\{ 0 \}}).
\end{equation}
Similarly, for the morphism $f_I\coloneq f\times \id_I\colon M_I\longrightarrow N_I$ 
we can define a natural morphism
\begin{equation}
\phi_\ast(f_I)\colon \rsect_{A_I}(M_I;L_I)\longrightarrow 
\rsect_{f_I(A_I)}(N_I;\rmR f_{I\ast} L_I).
\end{equation}

\begin{lemma}\label{lem-pushPhi}
In the situation as above, we set $B_I\coloneq f_I(A_I)\subset N_I$ and 
$B_{\{0\}}\coloneq f(A_{\{0\}})\subset N$.
Then there exists a commutative diagram
\begin{equation}
\begin{tikzcd}
\rsect_{A_I}(M_I;L_I)
\arrow[d,"\phi_\ast(f_I)"] \arrow[r,"\Phi"] & 
\rsect_{A_{\{0\}}}(M;L_{\{0\}}) \arrow[d,"\phi_\ast(f)"] & \\
\rsect_{B_I}(N_I;\rmR f_{I\ast} L_I)
\arrow[r," "] & 
\rsect_{B_{\{0\}}}(N;\rmR f_\ast L_{\{0\}}). &  
\end{tikzcd}
\end{equation}   
\end{lemma}

\begin{proof}
We set $f_J\coloneq f\times \id_J\colon M_J\longrightarrow N_J$ and $B_J\coloneq f_J(A_J)$.
Then by applying the natural transformation 
$\rsect_{A_J}(M_J;\cdot)\longrightarrow\rsect_{B_J}(N_J;\rmR f_{J\ast}(\cdot))$ to
the morphism $L_J \longrightarrow \rsect_{M_I}L_J$,
we obtain a commutative diagram
\begin{equation}\label{eq-pushJI}
\begin{tikzcd}
\rsect_{A_J}(M_J;L_J) 
\arrow[d," "] \arrow[r,"\sim"] & 
\rsect_{A_I}(M_I;L_I)
\arrow[d,"\phi_\ast(f_I)"] & \\
\rsect_{B_J}(N_J;\rmR f_{J\ast} L_J) \arrow[r,"\sim"] & 
\rsect_{B_I}(N_I;\rmR f_{I\ast} L_I). &  
\end{tikzcd}
\end{equation}   
In the same way, we obtain the following commutative diagram by applying 
the natural transformation $\rsect_{A_J}(M_J;\cdot)\longrightarrow
\rsect_{B_J}(N_J;\rmR f_{J\ast}(\cdot))$ to
the morphism $L_J \longrightarrow (L_J)_{M_{\{ 0 \}}}$:
\begin{equation}\label{eq-pushJ0}
\begin{tikzcd}
\rsect_{A_J}(M_J;L_J) 
\arrow[d," "] \arrow[r] & 
\rsect_{A_{\{0\}}}(M;L_{\{0\}})
\arrow[d,"\phi_\ast(f)"] & \\
\rsect_{B_J}(N_J;\rmR f_{J\ast} L_J) \arrow[r] & 
\rsect_{B_{\{0\}}}(N;\rmR f_\ast L_{\{0\}}). &  
\end{tikzcd}
\end{equation}   
By \eqref{eq-pushJI} and \eqref{eq-pushJ0} we obtain the assertion.
\end{proof}

Next, let $g\colon N\longrightarrow M$ be a morphism of smooth manifolds.
Then for a local system $L_J$ on $M_J$ and a 
closed subset $A_J \subset M_J$, there exists a natural
morphism
\begin{equation}
\phi^\ast(g)\colon \rsect_{A_{\{0\}}}(M;L_{\{0\}})
\longrightarrow\rsect_{g^{-1}(A_{\{0\}})}(N;g^{-1} L_{\{0\}}).
\end{equation}
Similarly, for the morphism $g_I\coloneq g\times\id_I\colon N_I\longrightarrow M_I$ 
we can define a natural morphism
\begin{equation}
\phi^\ast(g_I)\colon \rsect_{A_I}(M_I;L_I)\longrightarrow
\rsect_{g_I^{-1}(A_I)}(N_I;g_I^{-1}L_I).
\end{equation} 

\begin{lemma}\label{lem-pullPhi}
In the situation as above, we set $B_I\coloneq g_I^{-1}(A_I)\subset 
N_I$ and $B_{\{0\}}\coloneq g^{-1}(A_{\{0\}})\subset N$.
Then there exists a commutative diagram
\begin{equation}
\begin{tikzcd}
\rsect_{A_I}(M_I;L_I)
\arrow[d,"\phi^\ast(g_I)"] \arrow[r,"\Phi"] & 
\rsect_{A_{\{0\}}}(M;L_{\{0\}}) \arrow[d,"\phi^\ast(g)"] & \\
\rsect_{B_I}(N_I;g_I^{-1} L_I)
\arrow[r," "] & 
\rsect_{B_{\{0\}}}(N;g^{-1} L_{\{0\}}). &  
\end{tikzcd}
\end{equation}   
\end{lemma}

\begin{proof}
We set $g_J\coloneq g\times \id_J\colon N_J\longrightarrow M_J$ and $B_J\coloneq g_J^{-1}(A_J)$.
As in the proof of 
the previous lemma, the assertion follows from the following commutative diagram:
\begin{equation}
\begin{tikzcd}
\rsect_{A_I}(M_I;L_I) \arrow[d,"\phi^\ast(g_I)"] &
\rsect_{A_J}(M_J;L_J) \arrow[d," "] \arrow[l,"\sim"'] \arrow[r] & 
\rsect_{A_{\{0\}}}(M;L_{\{0\}}) \arrow[d,"\phi^\ast(g)"] & \\
\rsect_{B_I}(N_I;g_I^{-1}L_I) &
\rsect_{B_J}(N_J;g_J^{-1} L_J) \arrow[l,"\sim"'] \arrow[r] & 
\rsect_{B_{\{0\}}}(N;g^{-1}L_{\{0\}}). &  
\end{tikzcd}
\end{equation}   
\end{proof}

We set $M_{J^2}\coloneq M\times J^2$, $M_{I^2}\coloneq M\times I^2$, 
$M_{\{0\}\times I}\coloneq M\times \{0\}\times I$, 
$M_{I\times\{0\}}\coloneq M\times I\times \{0\}$ and 
$M_{\{0\}^2}\coloneq M\times\{0\}\times\{0\}$.
Let $A_{J^2}\subset M\times J^2$ be a closed subset and set 
\begin{align}
A_{I^2}\coloneq A_{J^2}\cap M_{I^2} &\quad \subset M_{I^2}, \\
A_{\{0\}\times I}\coloneq A_{J^2}\cap M_{\{0\}\times I} &\quad 
\subset M_{\{0\}\times I}\simeq M_I = M\times I, \\
A_{I\times \{0\}}\coloneq A_{J^2}\cap M_{I\times\{0\}} &\quad 
\subset M_{I\times\{0\}}\simeq M_I = M\times I, \\
A_{\{0\}^2}\coloneq A_{J^2}\cap M_{\{0\}^2}  &\quad \subset M_{\{0\}^2}\simeq M.
\end{align}
Moreover, for a local system $L_{J^2}$ on $M_{J^2}$ let 
$L_{I^2}$, $L_{\{0\}\times I}$, $L_{I\times\{0\}}$ and $L_{\{0\}^2}$ be its 
restrictions to $M_{I^2}$, $M_{\{0\}\times I}$, $M_{I\times\{0\}}$
and $M_{\{0\}^2}$, respectively.
Then as in the construction of the morphism $\Phi$, we obtain morphisms
\begin{empheq}[left=\empheqlbrace]{align*}
\Phi_1&\colon \rsect_{A_{I\times\{0\}}}(M_I; L_{I\times\{0\}})
\longrightarrow \rsect_{A_{\{0\}^2}}(M; L_{\{0\}^2}), \\
\Phi_2&\colon \rsect_{A_{\{0\}\times I}}(M_I; L_{\{0\}\times I})
\longrightarrow \rsect_{A_{\{0\}^2}}(M; L_{\{0\}^2})\,
\end{empheq}
and 
\begin{empheq}[left=\empheqlbrace]{align*}
\Phi_1^\prime&\colon \rsect_{A_{I^2}}(M_{I^2};L_{I^2})\longrightarrow
\rsect_{A_{\{0\}\times I}}(M_I; L_{\{0\}\times I}), \\
\Phi_2^\prime&\colon \rsect_{A_{I^2}}(M_{I^2};L_{I^2})\longrightarrow
\rsect_{A_{I\times\{0\}}}(M_I; L_{I\times\{0\}}).
\end{empheq}
   
\begin{lemma}\label{new-lemmes} 
There exists a commutative diagram
\begin{equation}
\begin{tikzcd}
\rsect_{A_{I^2}}(M_{I^2}; L_{I^2}) \ar[d,"\Phi_1^\prime"] \ar[r,"\Phi_2^\prime"] & 
\rsect_{A_{I\times\{0\}}}(M_I; L_{I\times\{0\}}) \ar[d,"\Phi_1"] & \\
\rsect_{A_{\{0\}\times I}}(M_I; L_{\{0\}\times I}) \ar[r,"\Phi_2"] & 
\rsect_{A_{\{0\}^2}}(M; L_{\{0\}^2}). &  
\end{tikzcd}
\end{equation}
\end{lemma}

\begin{proof}
For a locally closed subset $B\subset M_{J^2}$ of $M_{J^2}$ and its inclusion
map $j_B\colon B\longhookrightarrow M_{J^2}$ we set 
\begin{equation}
\SL[B]\coloneq \rmR(j_B)_\ast (L_{J^2}\vbar_B) \quad \in\BDC(\CC_{M_{J^2}}).
\end{equation}
Then we obtain a commutative diagram
\begin{equation}
\begin{tikzcd}
\SL[M_{I^2}] & & & \\
\SL[M_{J\times I}] \ar[u,"\sim" sloped] \ar[d] & 
\SL[M_{J^2}] \ar[l,"\sim"'] \ar[ul,"\sim" sloped] \ar[d] & & \\
\SL[M_{\{0\}\times I}] &
\SL[M_{\{0\}\times J}] \ar[l,"\sim"'] \ar[r] & \SL[M_{\{0\}^2}]  &   
\end{tikzcd}
\end{equation}
where $M_{J\times I}$ and $M_{\{0\}\times J}$ stand for $M\times J\times I$ 
and $M\times \{0\}\times J$, respectively. 
Applying the functor $\rsect_{A_{J^2}}(M_{J^2};\cdot)$ to it, we see that for 
the isomorphism
\begin{equation}
\Psi\colon \rsect_{A_{J^2}}(M_{J^2} ; L_{J^2}) \simto 
\rsect_{A_{I^2}}(M_{I^2}; L_{I^2})
\end{equation}
induced by the one $\SL[M_{J^2}]\simto\SL[M_{I^2}]$ and the natural 
morphism
\begin{equation}
\Theta\colon \rsect_{A_{J^2}}(M_{J^2}; L_{J^2})
\longrightarrow \rsect_{A_{\{0\}^2}}(M; L_{\{0\}^2})
\end{equation}
we have $\Phi_2\circ\Phi_1^\prime\circ\Psi=\Theta$.
Similarly, we can show that $\Phi_1\circ\Phi_2^\prime\circ\Psi=\Theta$.
As $\Psi$ is an isomorphism, we thus obtain 
$\Phi_1\circ\Phi_2^\prime=\Phi_2 \circ \Phi_1^\prime$ as desired.
\end{proof}

Now assume that $M$ is a real analytic manifold of dimension $n$ and 
$A_J\subset M_J=M\times J$ is a closed subanalytic subset of dimension $\leq p+1$ such that
$\overline{A_I}=A_J$ and 
for any $t\in J=[0,b)$ the dimension of the subset 
\begin{equation}
A_{\{t\}}\coloneq A_J\cap (M\times \{t\}) \quad \subset M\times\{t\}\simeq M
\end{equation}  
of $M$ is $\leq p$.
Then we obtain a morphism 
\begin{equation}
H^{n-p}(\Phi)\colon H_{A_I}^{n-p}(M_I;L_I)\longrightarrow H_{A_{\{0\}}}^{n-p}(M;L_{\{0\}})
\end{equation}
induced by $\Phi$.
For $L_J=(j_{M_I})_\ast \orsh_{M_I}$ it gives rise to a morphism 
\begin{equation}
\Xi_p\colon H_{p+1}^\BM(A_I;\CC)\longrightarrow H_p^\BM(A_{\{0\}};\CC)
\end{equation}
of Borel-Moore homology groups.
As an element $C_I$ of $H_{p+1}^\BM(A_I;\CC)$ can be considered as a family of 
subanalytic $p$-cycles in $M$ parameterized by $t\in I=(0,b)$, we call 
$\Xi_p(C_I)\in H_p^\BM(A_{\{0\}};\CC)$ the limit of $C_I$ and denote it by
$\displaystyle\lim_{t\to+0}C_{\{t\}}$ or $\displaystyle\lim_{t\to+0} C_I$
(cf. Schmid-Vilonen \cite[Section 3]{SV96}).
On the other hand, for the inclusion map $j_{A_I}\colon A_I\longhookrightarrow M_J$
and $\partial A_I=\overline{A_I}\setminus A_I=A_J\setminus A_I=A_{\{0\}}$ there exist
a distinguished triangle 
\begin{equation}
\omega_{A_0} \longrightarrow \omega_{A_J} 
\longrightarrow \rmR(j_{A_I})_\ast \omega_{A_I} \overset{+1}{\longrightarrow}
\end{equation}  
and a morphism 
\begin{equation}\label{eq-jAI}
(j_{A_I})_\ast H^{-(p+1)}\omega_{A_I}\longrightarrow H^{-p}\omega_{A_{\{0\}}}   
\end{equation}
associated to it.
Applying the functor $\rsect(M_J;\cdot)$ to it, we obtain a morphism
\begin{equation}
\Delta_p\colon H_{p+1}^\BM(A_I;\CC)\longrightarrow H_p^\BM(A_{\{0\}};\CC)
\end{equation}
of Borel-Moore homology groups.
Recall that for an element $C_I$ of $H_{p+1}^\BM(A_I;\CC)$ one can naturally identify
$\Delta_p(C_I)\in H_p^\BM(A_{\{0\}};\CC)$ with the boundary of $C_I$.
Here we regard $C_I$ as a subanalytic $(p+1)$-``chain" in $N\coloneq M\times (-b,b)$.
The following result was obtained in \cite[Proposition 3.17]{SV96}.

\begin{proposition}[{Schmid-Vilonen \cite[Proposition 3.17]{SV96}}]\label{prop-boundary}
We have $\Xi_p=\Delta_p$.
Namely for a family $C_I \in H_{p+1}^\BM(A_I;\CC)$ of subanalytic $p$-cycles in $M$ 
parameterized by $t\in I=(0,b)$ there exists an equality
\begin{equation}
\lim_{t\to+0}C_{\{t\}} =\Delta_p(C_I).
\end{equation}
\end{proposition}

\begin{proof}
In the $(n+1)$-dimensional real analytic manifold $N=M\times (-b,b)$ we set 
$N_\pm\coloneq \Set*{(x,t)\in N}{\pm t>0}\subset N$ and identify 
$M_{\{0\}} = \Set*{(x,t)\in N}{t=0} $ with $M$ naturally.
Then we have $N_+=M_I$ and there exists a local system $K$ on $N$ such that
$K\vbar_{M_J}\simeq L_J$, $K\vbar_{M_I}\simeq L_I$ and $K\vbar_{M_{\{0\}}}\simeq L_{\{0\}}$.
Note that such $K$ is unique up to isomorphisms.
Then there exist a distinguished triangle 
\begin{equation}
\rsect_M(K)\longrightarrow K\longrightarrow
\rsect_{N_+}(K)\oplus\rsect_{N_-}(K) \overset{+1}{\longrightarrow}
\end{equation} 
and a morphism
\begin{equation}
\rsect_{N_+}(K)\oplus\rsect_{N_-}(K)
\overset{\delta}{\longrightarrow}\rsect_{M}(K)[1]
\end{equation}
associated to it.
Moreover, for the inclusion maps $j_{M_J}\colon M_J\longhookrightarrow N$ and 
$j_M\colon M\simeq M_{\{0\}} \longhookrightarrow N$ we obtain a natural morphism
\begin{equation}
\rsect_{N_+}(K)\simeq (j_{M_J})_\ast L_J\overset{\sigma}{\longrightarrow}
(j_M)_\ast L_{\{0\}}
\end{equation}
and a commutative diagram
\begin{equation}\label{eq-sigma}
\vcenter{
\xymatrix@M=5pt{
\rsect_{N_+}(K)\simeq (j_{M_J})_\ast L_J \ar[d] \ar[r]^-{\sigma} & 
(j_M)_\ast L_{\{0\}} \ar[d]^-[@!-90]{\sim} & \\
\rsect_{N_+}(K)\oplus\rsect_{N_-}(K) \ar[r]^-{\delta} & 
\rsect_M(K)[1] &  
}}
\end{equation}
containing it.
Applying the functor $\rsect_{A_J}(\cdot)$ to the bottom horizontal
arrow of it, we obtain the morphism \eqref{eq-jAI}.
Then the assertion immediately follows by applying the functor 
$\rsect_{A_J}(N;\cdot)$ to the commutative diagram \eqref{eq-sigma}. 
\end{proof}

\begin{remark}
In \cite[Proposition 3.17]{SV96} Schmid and Vilonen consider the geometric boundary 
$\partial C_I$ of $C_I\in H_{p+1}^\BM(A_I;\CC)$.
So their result in it coincides with ours in Proposition \ref{prop-boundary}.
\end{remark}

From now, we shall give a reformulation of Schmid-Vilonen \cite[Proposition 3.27]{SV96}
in the theory of Lagrangian cycles of Kashiwara-Schapira \cite[Section 9.3]{KS90}.
Let $X$ be a real analytic manifold of dimension $n$ 
and $\pi_X\colon T^\ast X\longrightarrow X$ the canonical projection.
Then in \cite[Definition 9.3.1]{KS90} Kashiwara and 
Schapira defined the sheaf $\SL_X$ of
``conic" Lagrangian cycles on $T^\ast X$ by 
\begin{equation}
\SL_X\coloneq\varinjlim_{\Lambda}H_\Lambda^0(\pi_X^{-1}\omega_X),
\end{equation}
where $\Lambda\subset T^\ast X$ ranges through the family of all closed conic subanalytic
isotropic subsets of $T^\ast X$. 
Dropping the conicness of $\Lambda\subset T^\ast X$ we obtain the sheaf $\tl{\SL}_X$ of
(not necessarily conic) Lagrangian cycles on $T^\ast X$.
Note that $\pi_X^{-1}\omega_X\simeq \omega_{T^\ast X}\otimes 
\orsh_{T^\ast X/X}[-n]$ and a global section of $\tl{\SL}_X$ 
is an element of the twisted Borel-Moore
homology group
\begin{equation}
H_\Lambda^0(T^\ast X;\pi_X^{-1}\omega_X)\simeq
H_n^\BM (\Lambda;\orsh_{T^\ast X/X} )
\end{equation}
for some closed subanalytic isotropic subset $\Lambda\subset T^\ast X$.
Let $\Lambda_I\subset T^\ast X\times I$ and $\Lambda_J\subset T^\ast X\times J$
be closed subanalytic subsets of dimension $\leq n+1$ such that 
$\overline{\Lambda_I} = \Lambda_J$ and for any $t\in J=[0,b)$ the closed subset
\begin{equation}
\Lambda_{\{t\}}\coloneq \Lambda_J\cap (T^\ast X\times \{t\}) 
\quad \subset T^\ast X\times \{t\}\simeq T^\ast X 
\end{equation}
of $T^\ast X$ is isotropic and hence of dimension $\leq n$.
Then we obtain a limit morphism 
\begin{equation}
\lim_{t\to+0}\colon H_{n+1}^\BM (\Lambda_I; \orsh_{T^\ast X/X}\boxtimes\CC_I 
)\longrightarrow
H_n^\BM (\Lambda_{\{0\}};\orsh_{T^\ast X/X} )
\end{equation} 
of twisted Borel-Moore homology groups.
We call the pair $(\Lambda_I,\Lambda_J)$ a family of Lagrangian subsets in $T^\ast X$ 
for short.
Let $f\colon Y\longrightarrow X$ be a morphism of real analytic manifolds and
\begin{equation}
T^\ast Y \overset{\rho_f}{\longleftarrow} Y\ftimes{X} T^\ast X
\overset{\varpi_f}{\longrightarrow} T^\ast X
\end{equation}
the morphisms induced by $f$.
Let $n$ and $m$ be the dimensions of $X$ and $Y$, respectively.
Then by Goresky-MacPherson \cite[page 43]{GM88} and the proof of 
Kashiwara \cite[Proposition A.54]{Kas03} we can easily show that if $\Lambda\subset T^\ast X$
(resp. $\Lambda\subset T^\ast Y$) is a closed subanalytic isotropic 
subset and $\rho_f$ (resp. $\varpi_f$) is proper on 
$\varpi_f^{-1}(\Lambda)$ (resp. $\rho_f^{-1}(\Lambda)$)
then $\rho_f\varpi_f^{-1}(\Lambda)\subset T^\ast Y$
(resp. $\varpi_f\rho_f^{-1}(\Lambda)\subset T^\ast X$) is also a 
closed subanalytic isotropic subset 
(see e.g. the proof of \cite[Lemma 5.4]{Tak22}).
First, let $\Lambda\subset T^\ast Y$ be a closed subanalytic isotropic subset of 
$T^\ast Y$ and assume that $\varpi_f$ is proper on $\rho_f^{-1}(\Lambda)$.
Then as in the proof of \cite[Proposition 9.3.2 (i)]{KS90} for 
$\Lambda^\prime\coloneq \varpi_f\rho_f^{-1}(\Lambda)\subset T^\ast X$
we can construct a morphism 
\begin{align}
\mu_\ast(f)\colon H_\Lambda^0(T^\ast Y;\pi_Y^{-1}\omega_Y)
&\simeq H_m^\BM (\Lambda;\orsh_{T^\ast Y/Y} ) \\
&\longrightarrow 
H_{\Lambda^\prime}^0(T^\ast X;\pi_X^{-1}\omega_X)
\simeq H_n^\BM (\Lambda^\prime;\orsh_{T^\ast X/X} )
\end{align}
of the direct image of Lagrangian cycles by $f$.
Similarly, for a family
$(\Lambda_I,\Lambda_J)$ of Lagrangian subsets in $T^\ast Y$ satisfying the 
properness condition for any $t\in J$,
defining a family $(\Lambda_I^\prime,\Lambda_J^\prime)$ of Lagrangian subsets 
in $T^\ast X$ from $(\Lambda_I,\Lambda_J)$ by $f$ we obtain a morphism
\begin{equation}
\mu_\ast(f,I)\colon 
H_{m+1}^\BM (\Lambda_I; \orsh_{T^\ast Y/Y}\boxtimes\CC_I )
\longrightarrow
H_{n+1}^\BM (\Lambda_I^\prime; \orsh_{T^\ast X/X}\boxtimes\CC_I )
\end{equation}
Then by Lemmas \ref{lem-pushPhi} and \ref{lem-pullPhi} and the proof of 
\cite[Proposition 9.3.2 (i)]{KS90} we obtain the following result of 
Schmid-Vilonen \cite[Proposition 3.27]{SV96} in a slightly modified form.

\begin{proposition}[{Schmid-Vilonen \cite[Proposition 3.27]{SV96}}]\label{prop-Lag_push}
In the situation as above, the diagram
\begin{equation}
\begin{tikzcd}
H_{m+1}^\BM (\Lambda_I; \orsh_{T^\ast Y/Y}\boxtimes\CC_I )
\arrow[d,"\mu_\ast{(f,I)}"] \arrow[r,"\tikzlim{t\to+0}"] & 
H_m^\BM (\Lambda_{\{0\}};\orsh_{T^\ast Y/Y} ) \arrow[d,"\mu_\ast{(f)}"] & \\
H_{n+1}^\BM (\Lambda_I^\prime;\orsh_{T^\ast X/X}\boxtimes\CC_I 
) \arrow[r,"\tikzlim{t\to+0}"] & 
H_n^\BM (\Lambda_{\{0\}}^\prime;\orsh_{T^\ast X/X} ) &  
\end{tikzcd}
\end{equation}
is commutative.
\end{proposition}

Next, let $\Lambda\subset T^\ast X$ be a closed subanalytic isotropic subset of 
$T^\ast X$ and assume that $\rho_f$ is proper on $\varpi_f^{-1}(\Lambda)$.
Then as in the proof of \cite[Proposition 9.3.2 (ii)]{KS90} for 
$\Lambda^\prime\coloneq \rho_f\varpi_f^{-1}(\Lambda)\subset T^\ast Y$ we can construct
a morphism
\begin{equation}
\mu^\ast(f)\colon H_n^\BM (\Lambda;\orsh_{T^\ast X/X} )
\longrightarrow H_m^\BM (\Lambda^\prime;\orsh_{T^\ast Y/Y} )
\end{equation}
of the inverse image of Lagrangian cycles by $f$.
Similarly, for a family $(\Lambda_I,\Lambda_J)$ of Lagrangian subsets in $T^\ast X$
satisfying the properness condition for any $t\in J$, defining a family 
$(\Lambda_I^\prime,\Lambda_J^\prime)$ of Lagrangian subsets in $T^\ast Y$ 
from $(\Lambda_I,\Lambda_J)$ by $f$ we 
obtain a morphism 
\begin{equation}
\mu^\ast(f,I)\colon H_{n+1}^\BM (\Lambda_I; \orsh_{T^\ast X/X}\boxtimes\CC_I )
\longrightarrow H_{m+1}^\BM (\Lambda_I^\prime; \orsh_{T^\ast Y/Y}
\boxtimes\CC_I )
\end{equation}
and the following result.

\begin{proposition}
In the situation as above, the diagram
\begin{equation}
\begin{tikzcd}
H_{n+1}^\BM (\Lambda_I; \orsh_{T^\ast X/X}\boxtimes\CC_I )
\arrow[d,"\mu^\ast{(f,I)}"] \arrow[r,"\tikzlim{t\to+0}"] & 
H_n^\BM (\Lambda_{\{0\}};\orsh_{T^\ast X/X} ) \arrow[d,"\mu^\ast{(f)}"] & \\
H_{m+1}^\BM (\Lambda_I^\prime; \orsh_{T^\ast Y/Y}\boxtimes\CC_I
) \arrow[r,"\tikzlim{t\to+0}"] & 
H_m^\BM (\Lambda_{\{0\}}^\prime; \orsh_{T^\ast Y/Y} 
) &  
\end{tikzcd}
\end{equation}
is commutative.
\end{proposition}
   
\begin{example}
Let $X$ be a real analytic manifold of dimension $n$ 
and $F\in\BDrc(\CC_X)$ an $\RR$-constructible sheaf on it.
Then by Kashiwara-Schapira \cite[Section 9.3]{KS90} we obtain its characteristic cycle
\begin{equation}
\CCyc(F) \quad \in H_{\msupp(F)}^n(T^\ast X;\orsh_{T^\ast X/ X})
\simeq H_n^\BM(\msupp(F);\orsh_{T^\ast X/X}).
\end{equation}
From it, in \cite{SV96} for a real analytic function $g\colon X\longrightarrow \RR$
on $X$ Schmid and Vilonen constructed a family of Lagrangian cycles 
\begin{equation}
C_{\{t\}}\coloneq \CCyc(F) + tdg \quad (t\in I=(0,b))
\end{equation}
in $T^\ast X$ by hands and used it to prove their main theorems.
We can construct it sheaf-theoretically as follows.
First, for a morphism $f\colon M\longrightarrow N$ of real analytic manifolds, 
a closed subanalytic subset $S\subset N$ and a local system $L$ on $N$,
there exists a morphism 
\begin{equation}
\rsect_S(N;L)\longrightarrow \rsect_S(N;\rmR f_\ast f^{-1} L) 
\simeq \rsect_{f^{-1}(S)}(M;f^{-1}L).
\end{equation}
We apply this construction to the following situation:
\begin{empheq}[left=\empheqlbrace]{align*}
&M= T^\ast X\times I\ni (p,t)\overset{f}{\longmapsto} (p-tdg(\pi_X(p))) \in N=T^\ast X, \\
&S= \Lambda\coloneq \msupp(F) \subset N=T^\ast X, \quad L=\orsh_{T^\ast X/X}.
\end{empheq}
Then we obtain a morphism 
\begin{equation}
H_\Lambda^n(T^\ast X; \orsh_{T^\ast X/X})\longrightarrow
H_{A_I}^n(T^\ast X\times I; \orsh_{T^\ast X/X}\boxtimes\CC_I),
\end{equation}
where we set 
\begin{equation}
A_I\coloneq f^{-1}(S)=
\Set*{(p+tdg(\pi_X(p)),t)}{p\in\Lambda, t\in I}.
\end{equation}
We thus can use the image $C_I\in H_{A_I}^n(T^\ast X\times 
I;\orsh_{T^\ast X/X}\boxtimes \CC_I)$ of $\CCyc(F)$ by it.
\end{example}

\subsection{Intersection numbers of Borel-Moore homology cycles and limits}
In this subsection, we define intersection numbers of Borel-Moore homology cycles in a parameterized setting as in Section \ref{subsec-limit} and investigate their relation with limits.
Fix $b>0$ and set $I\coloneq(0,b)$, $J\coloneq[0,b)$ and $\tl{I}\coloneq (-b,b)$.
For a real analytic manifold $M$ of dimension $n$ we set 
$M_J\coloneq M\times J$ and $M_{\tl{I}}\coloneq M\times\tl{I}$.
Let $t\colon M_{\tl{I}}\longrightarrow \tl{I}$ be the projection.
For $a\in \tl{I}=(-b,b)$ we denote by $i_a\colon M\simeq M\times\{a\}\longhookrightarrow M_{\tl{I}}$ the closed embedding.
\begin{remark}
    Let $A_J\subset M_J$ be a closed subset and $L$ a local system on $M$ over the field $\CC$.
    We regard $A_J$ as a closed subset of $M_{\tl{I}}$.
    We also denote by 
    $\CC_J\in \BDC(\tl{I})$ the extension by zero of
    the constant sheaf $\CC_J\in\BDC(J)$.
    Then we have 
    \begin{equation}
        \rsect_{A_J}(M_J;(\omega_M\otimes L)\boxtimes \CC_J)
        \simeq 
        \rsect_{A_J}(M_{\tl{I}};(\omega_M\otimes L)\boxtimes \CC_J).
    \end{equation}
    In particular, if $L\simeq \CC_M$ then we obtain
    \begin{equation}
        \rsect_{A_J}(M_J;\omega_M\boxtimes \CC_J) \simeq \rsect_{A_J}(M_{\tl{I}};\omega_M\boxtimes \CC_J) \simeq 
        \rsect_{A_J}(M_{\tl{I}};t^!\CC_J).
    \end{equation}
    In what follows, for any closed subset $A_J\subset M_J$ we implicitly regard $A_J$ as a closed subset of $M_{\tl{I}}$ and freely use these isomorphisms. 
\end{remark}

First, let us define the ``integration morphism'' in the parameterized setting.
Let $A_J\subset M_J$ be a closed subanalytic subset.
We assume that the projection $t\colon M_{\tl{I}}\longrightarrow \tl{I}$ is proper on $A_{J}$.
Then we have a chain of morphisms
\begin{align}
    \rsect_{A_J}(M_{\tl{I}}; \omega_M\boxtimes\CC_J)
    &\simot \rsect(\tl{I}; \rmR t_!\rsect_{A_J} t^!\CC_J) \\
    &\simto \rsect(\tl{I}; (\rmR t_!\rsect_{A_J} t^!\CC_J)_{t(A_{J})} ) \\
    &\longrightarrow \rsect(\tl{I}; (\rmR t_! t^!\CC_{J})_{t(A_{J})} ) \\
    &\longrightarrow \rsect\bigl( \tl{I}; \CC_{t(A_{J})} \bigr).
\end{align}
Taking the 0-th cohomology groups we obtain a morphism
\begin{equation}\label{eq-relint}
    \int_t\colon H_{A_J}^0(M_{\tl{I}};\omega_M\boxtimes \CC_J)\longrightarrow
    H^0(\tl{I};\CC_{t(A_{J})}).
\end{equation}
Note that the cohomology group 
$H^0(\tl{I};\CC_{t(A_{J})})\simeq H^0(J;\CC_{t(A_{J})})$ 
is the $\CC$-vector space consisting of $\CC$-valued locally constant functions on the closed subset $t(A_{J})\subset J$ ($\subset \tl{I}$).

Next, we define the intersection of two cycles.
Let $A_{J}\subset M_{J}$ and $A_{J}^\prime\subset M_{J}$
be closed subanalytic subsets of dimension $\leq p+1$ and $\leq q+1$, respectively. 
We assume that for each $a\in J=[0,b)$ the dimensions of the closed subanalytic subsets 
$A_a\coloneq i_a^{-1}(A_J)\subset M$ and $A_a^\prime\coloneq i_a^{-1}(A_J^\prime)\subset M$ are $\leq p$ and $\leq q$, respectively.
Note that we can regard $A_J$ and $A_J^\prime$ as (closed) subanalytic subsets of $M_{\tl{I}}$.
Let $L$ and $L^\prime$ be local systems on $M$ over the field $\CC$.
Then we have a chain of morphisms
\begin{align}
    &\rsect_{A_J}(M_{\tl{I}}; (\omega_M\otimes L)\boxtimes\CC_{J})\otimes
    \rsect_{A_J^\prime}(M_{\tl{I}}; (\omega_M\otimes L^\prime)\boxtimes \CC_{J}) \\
    &\longrightarrow
    \rsect_{A_J\cap A_J^\prime}\bigl( M_{\tl{I}}; (\omega_M\otimes\omega_M\otimes L\otimes L^\prime)\boxtimes \CC_{J} \bigr) \\
    &\simeq \rsect_{A_J\cap A_J^\prime}\bigl( M_{\tl{I}}; (\omega_M\otimes\orsh_M\otimes L\otimes L^\prime)\boxtimes \CC_{J} \bigr)[n]
\end{align}
which induces a morphism
\begin{align}
    \cap\colon H_{A_J}^{-p}(M_{\tl{I}};(\omega_M\otimes L)\boxtimes\CC_{J})
    \times H_{A_J^\prime}^{-q}(M_{\tl{I}};(\omega_M\otimes L^\prime)\boxtimes\CC_{J}) & \\
    \longrightarrow H_{A_J\cap A_J^\prime}^{n-p-q}\bigl( M_{\tl{I}}; (\omega_M\otimes L\otimes L^\prime \otimes \orsh_{M})\boxtimes \CC_{J} \bigr).&
\end{align}
From now on, we assume that $p+q=n$, $L\otimes L^\prime\simeq\iorsh_M$ and the projection $t\colon M_{\tl{I}}\longrightarrow \tl{I}$ is proper on $A_J\cap A_J^\prime\subset M_{\tl{I}}$.
For $C_J\in H_{A_J}^{-p}(M_{\tl{I}};(\omega_M\otimes L)\boxtimes\CC_{J})$ and $C_J^\prime \in H_{A_J^\prime}^{-q}(M_{\tl{I}};(\omega_M\otimes L^\prime)\boxtimes\CC_{J})$ let us set
\begin{equation}
    \#(C_J\cap C_J^\prime) \coloneq \int_t(C_J\cap C_J^\prime) \quad \in  H^0(\tl{I};\CC_{t(A_J\cap A_J^\prime)}).
\end{equation}
For $a\in t(A_J\cap A_J^\prime)$, there is a natural morphism
\begin{align}
    \phi^\ast(i_a) &\colon H_{A_J}^{-p}(M_{\tl{I}};(\omega_M\otimes L)\boxtimes\CC_{J})
    \longrightarrow H_{A_a}^{-p}(M; \omega_M\otimes L)
\end{align}
and we set $C_a\coloneq\phi^\ast(i_a)(C_J)\in H_{A_a}^{-p}(M; \omega_M\otimes L)$.
Similarly, we set $C_a^\prime\coloneq\phi^\ast(i_a)(C_J^\prime)\in H_{A_a^\prime}^{-q}(M; \omega_M\otimes L^\prime)$.
Since the projection $t\colon M_{\tl{I}}\longrightarrow \tl{I}$ is proper on $A_J\cap A_J^\prime\subset M_{\tl{I}}$, the closed subset $A_a\cap A_a^\prime$ of $M$ is compact.
Therefore, the (usual) intersection number $\#(C_a\cap C_a^\prime) \in\CC$ of  $C_a$ and $C_a^\prime$ is well-defined (see \cite[Definition 9.1.2]{KS90}).
For each $a\in t(A_J\cap A_J^\prime)$ we denote by $\#(C_J\cap C_J^\prime)(a)$ the value of the function $\#(C_J\cap C_J^\prime)$ at $a$.

\begin{proposition}\label{prop-intersection}
    In the above situation, one has
    \begin{equation}
        \#(C_J\cap C_J^\prime)(a) =
        \#(C_a\cap C_a^\prime) \quad \in \CC.
    \end{equation}
\end{proposition}

\begin{proof}
    Let us consider the Cartesian diagram below.
    \begin{equation}
    \begin{tikzcd}
        M \rar["i_a",hook] \dar["t_a"'] \drar["\square",phantom,pos=0.5] & M_{\tl{I}} \dar["t"] \\
        \{a\} \rar["\iota_a"',hook] & \tl{I}. 
    \end{tikzcd}
    \end{equation}
    Since $A_a\cap A_a^\prime$ is compact, by the same construction as (\ref{eq-relint}) we have a chain of morphisms 
    \begin{align}
        \rsect_{A_a\cap A_a^\prime}(M; \omega_M)
        &\simot \rsect(\{a\}; \rmR t_{a!}\rsect_{A_a\cap A_a^\prime}t_a^!\CC) \\
        &\simto \rsect(\{a\}; (\rmR t_{a!}\rsect_{A_a\cap A_a^\prime}t_a^!\CC)_{t_a(A_a\cap A_a^\prime)} ) \\
        &\longrightarrow \rsect(\{a\}; (\rmR t_{a!}t_a^!\CC)_{t_a(A_a\cap A_a^\prime)} ) \\
        &\longrightarrow \rsect(\{a\}; \CC_{t_a(A_a\cap A_a^\prime)}) 
    \end{align}
    which induces a morphism
    \begin{equation}
        \int_{t_a}\colon H_{A_a\cap A_a^\prime}^0(M;\omega_M) 
        \longrightarrow H^0(\{a\}; \CC_{t_a(A_a\cap A_a^\prime)}).
    \end{equation}
    Let $\int_M\colon H_{A_a\cap A_a^\prime}^0(M;\omega_M) 
    \longrightarrow \CC$ be the usual integration morphism 
     (see \cite[Section 3.3]{KS90}).
    By construction and $t_a(A_a\cap A_a^\prime)=\{a\}$, it is straightforward to check that the composition of morphisms 
    \begin{equation}
        H_{A_a\cap A_a^\prime}^0(M;\omega_M) \overset{\int_{M}}{\longrightarrow}\CC\simeq H^0(\{a\};\CC_{\{a\}})\simto H^0(\{a\}; \CC_{t_a(A_a\cap A_a^\prime)})
    \end{equation}
    coincides with $\int_{t_a}\colon H_{A_a\cap A_a^\prime}^0(M;\omega_M)
        \longrightarrow H^0(\{a\}; \CC_{t_a(A_a\cap A_a^\prime)})$.
    Therefore, it is sufficient to check that the following diagram commutes:
    \begin{equation*}
    \begin{tikzcd}[nodes={font=\footnotesize},column sep=2cm,row sep=1cm]
        H_{A_J}^{-p}(M_{\tl{I}};(\omega_M\otimes L)\boxtimes\CC_{J}) \otimes
        H_{A_J^\prime}^{-q}(M_{\tl{I}};(\omega_M\otimes L^\prime)\boxtimes\CC_{J})
        \dar["\cap"] \rar["\phi^\ast(i_a)\otimes \phi^\ast(i_a)"]
        \ar[dr,"\bfA",phantom,pos=0.42] &
        H_{A_a}^{-p}(M; \omega_M\otimes L) \otimes
        H_{A_a^\prime}^{-q}(M; \omega_M\otimes L^\prime) \dar["\cap"] 
        \\
        H_{A_J\cap A_J^\prime}^{0}\bigl( M_{\tl{I}}; \omega_M\boxtimes \CC_{J} \bigr) 
        \dar["\int_t"] \rar["\phi^\ast(i_a)"] \ar[dr,"\bfB",phantom,pos=0.48] &
        H_{A_a\cap A_a^\prime}^0 (M; \omega_M)
        \dar["\int_{t_a}"] 
        \\
        H^0(\tl{I};\CC_{t(A_J\cap A_J^\prime)}) \rar["\phi^\ast(\iota_a)"'] & 
        H^0(\{a\}; \CC_{t_a(A_a\cap A_a^\prime)}),
    \end{tikzcd}
    \end{equation*}
    where $\phi^\ast(\iota_a)$ is a natural morphism induced by the unit $\id_{\tl{I}}\longrightarrow \iota_{a\ast}\iota_a^{-1}$.
    We shall omit to write ``$\rmR$'' for derived functors, for the sake of simplicity.
    The commutativity of the square $\bfA$ follows from Diagram 2.a below.
    The square $\bfB$ decomposes as in Diagram 2.b below.
    It follows from Lemma \ref{lem-intunit} that the square $(\#)$ in Diagram 2.b commutes.
    This completes the proof.
\end{proof}

\begin{center}
\begin{adjustbox}{rotate=90,center}
\begin{minipage}{1\textheight}
\hspace*{-1.5cm}
\begin{tikzcd}[nodes={font=\scriptsize},row sep=1cm,column sep=tiny]
        \sect\Bigl( M_{\tl{I}}; \sect_{A_J}\bigl( (\omega_M\otimes L)\boxtimes \CC_J \bigr) \Bigr) 
        \otimes 
        \sect\Bigl( M_{\tl{I}}; \sect_{A_J^\prime}\bigl( (\omega_M\otimes L^\prime)\boxtimes \CC_J  \bigr) \Bigr)
        \ar[r] \ar[d] &
        \sect\Bigl( M_{\tl{I}}; \sect_{A_J} i_{a\ast} i_a^{-1}\bigl( (\omega_M\otimes L)\boxtimes \CC_{J} \bigr) \Bigr) 
        \otimes 
        \sect\Bigl( M_{\tl{I}}; \sect_{A_J^\prime} i_{a\ast}i_a^{-1}\bigl( (\omega_M\otimes L^\prime)\boxtimes \CC_{J} \bigr) \Bigr) 
        \ar[r,dash,"\sim"] \ar[d] &
        \sect\bigl( M; \sect_{A_a}(\omega_M\otimes L) \bigr)
        \otimes
        \sect\bigl( M; \sect_{A_a^\prime}(\omega_M\otimes L^\prime) \bigr)
        \ar[d] \\
        \sect\Bigl( M_{\tl{I}}; \sect_{A_J} \bigl((\omega_M\otimes L)\boxtimes \CC_{J} \bigr) 
        \otimes \sect_{A_J^\prime} \bigl((\omega_M\otimes L^\prime)\boxtimes \CC_{J}\bigr)  \Bigr) 
        \ar[r] \ar[dd] &
        \sect\Bigl( M_{\tl{I}}; \sect_{A_J} i_{a\ast}i_a^{-1} \bigl((\omega_M\otimes L)\boxtimes \CC_{J}\bigr) 
        \otimes \sect_{A_J^\prime} i_{a\ast}i_a^{-1} \bigl((\omega_M\otimes L^\prime )\boxtimes \CC_{J} \bigr) \Bigr) 
        \ar[r,dash,"\sim"] \ar[d] &
        \sect\bigl(M; \sect_{A_a}(\omega_M\otimes L)\otimes \sect_{A_a^\prime}(\omega_M\otimes L^\prime) \bigr) 
        \ar[dd] \\
        & 
        \sect\Bigl( M_{\tl{I}}; \sect_{A_J\cap A_J^\prime} \bigl( i_{a\ast}i_a^{-1}( (\omega_M\otimes L)\boxtimes \CC_{J} )\otimes i_{a\ast}i_a^{-1}( (\omega_M\otimes L^\prime)\boxtimes \CC_{J} ) \bigr) \Bigr) 
        \ar[d,dash,"\sim" sloped] & \\
        \sect\Bigl( M_{\tl{I}}; \sect_{A_J\cap A_J^\prime}\bigl( (\omega_M\otimes \omega_M \otimes L\otimes L^\prime)\boxtimes \CC_{J} \bigr) \Bigr) 
        \ar[r] \ar[d,dash,"\sim" sloped] &
        \sect\Bigl( M_{\tl{I}}; \sect_{A_J\cap A_J^\prime} i_{a\ast}i_a^{-1} \bigl( (\omega_M\otimes \omega_M \otimes L\otimes L^\prime)\boxtimes \CC_{J} \bigr) \Bigr) 
        \ar[r,dash,"\sim"] \ar[d,dash,"\sim" sloped] &
        \sect\bigl(M; \sect_{A_a\cap A_a^\prime}(\omega_M\otimes \omega_M \otimes L\otimes L^\prime) \bigr) 
        \ar[d,dash,"\sim" sloped] \\
        \sect_{A_J\cap A_J^\prime}( M_{\tl{I}}; \omega_M\boxtimes \CC_{J} ) [n] 
        \ar[r] &
        \sect_{A_J\cap A_J^\prime}\bigl( M_{\tl{I}}; i_{a\ast}i_a^{-1}(\omega_M\boxtimes \CC_{J} ) \bigr)[n] 
        \ar[r,dash,"\sim"] &
        \sect_{A_a\cap A_a^\prime}(M; \omega_M ) [n] 
\end{tikzcd}
\hspace*{11cm}
\\
\hspace*{11cm} 
\textrm{Diagram 2.a}
\end{minipage}
\end{adjustbox}
\end{center}

\begin{center}
\begin{adjustbox}{center}
\begin{tikzcd}[nodes={font=\footnotesize},row sep=0.6cm]
    \sect\bigl( M_{\tl{I}}; \sect_{A_J\cap A_J^\prime} (\omega_M\boxtimes \CC_{J}) \bigr)
    \ar[r] &
    \sect\bigl( M_{\tl{I}}; \sect_{A_J\cap A_J^\prime} i_{a\ast}i_a^{-1}(\omega_M\boxtimes \CC_{J}) \bigr) 
    \ar[r,"\sim",dash] &
    \sect(M;\sect_{A_a\cap A_a^\prime} \omega_M ) \\
    \sect(\tl{I};  t_!\sect_{A_J\cap A_J^\prime} t^!\CC_J) 
    \ar[r] \ar[d,"\sim" sloped] \ar[u,"\sim" sloped] &
    \sect(\tl{I}; t_!\sect_{A_J\cap A_J^\prime} i_{a\ast}i_a^{-1} t^!\CC_{J}) 
    \ar[r,"\sim",dash] \ar[d,"\sim" sloped] \ar[u,"\sim" sloped] &
    \sect( \{a\}; t_{a!}\sect_{A_a\cap A_a^\prime}t_a^!\CC) 
    \ar[d,"\sim" sloped] \ar[u,"\sim" sloped] \\
    \sect\bigl( \tl{I}; (t_!\sect_{A_J\cap A_J^\prime}t^!\CC_{J})_{t(A_J\cap A_J^\prime)} \bigr) 
    \ar[r] \ar[d] &
    \sect\bigl( \tl{I};  (t_!\sect_{A_J\cap A_J^\prime}i_{a\ast}i_a^{-1} t^!\CC_{J})_{t(A_J\cap A_J^\prime)} \bigr) 
    \ar[r,"\sim",dash] \ar[d] &
    \sect(\{a\}; (t_{a!} \sect_{A_a\cap A_a^\prime} t_a^!\CC)_{t_a(A_a\cap A_a^\prime)} )
    \ar[d] \\
    \sect\bigl( \tl{I}; (t_!t^!\CC_{J})_{t(A_J\cap A_J^\prime)} \bigr) 
    \ar[r] \ar[d] \ar[rrd,"(\#)",phantom] &
    \sect\bigl( \tl{I}; (t_! i_{a\ast}i_a^{-1} t^!\CC_{J})_{t(A_J\cap A_J^\prime)} \bigr) 
    \ar[r,dash,"\sim" sloped] & 
    \sect(\{a\}; (t_{a!}t_a^!\CC)_{t_a(A_a\cap A_a^\prime)} )
    \ar[d] \\
    \sect(\tl{I}; \CC_{t(A_J\cap A_J^\prime)}) 
    \ar[rr] &
    &
    \sect(\{a\}; \CC_{t_a(A_a\cap A_a^\prime)} ). 
\end{tikzcd}
\end{adjustbox}
\\
\textrm{Diagram 2.b}
\end{center}

\begin{example}\label{ex-Lagintersect}
    Let us apply Proposition \ref{prop-intersection} to the case of Lagrangian cycles.
    Let $X$ be a real analytic manifold of dimension $n$.
    For $a\in \tl{I}=(-b,b)$ we denote by $i_a\colon T^\ast X\simeq T^\ast X\times \{a\}\longhookrightarrow T^\ast X\times \tl{I}$ the closed embedding.
    Let $\Lambda_I,\Lambda_I^\prime\subset T^\ast X\times I$ and $\Lambda_J,\Lambda_J^\prime\subset T^\ast X\times J$ be closed subanalytic subsets of dimension $\leq n+1$ such that $\overline{\Lambda_I}=\Lambda_J$ and $\overline{\Lambda_I^\prime}=\Lambda_J^\prime$. 
    We assume that for each $a\in J$ the closed subsets 
    \begin{equation}
    \Lambda_{a} \coloneq i_a^{-1}(\Lambda_J),
    \quad 
    \Lambda_{a}^\prime \coloneq i_a^{-1}(\Lambda_J^\prime)
    \end{equation}
    of $T^\ast X$ are isotropic. 
    Let us consider families of Lagrangian cycles 
    $C_I\in H_{\Lambda_I}^{-n}\bigl( T^\ast X\times I; (\omega_{T^\ast X}\otimes \orsh_{T^\ast X/X})\boxtimes\CC_{I} \bigr)$ 
    ($\simeq H_{n+1}^\BM(\Lambda_I; \orsh_{T^\ast X/X}\boxtimes\CC_{I})$)
    and 
    $C_I^\prime\in H_{\Lambda_I^\prime}^{-n}\bigl( T^\ast X\times I; (\omega_{T^\ast X}\otimes \orsh_{T^\ast X/X})\boxtimes\CC_{I} \bigr)$. 
    As mentioned in Section \ref{subsec-limit}, we can define a Lagrangian cycle
    \begin{equation}
        \lim_{t\to+0}C_t\coloneq \lim_{t\to+0}(C_I)\in 
        H_{\Lambda_0}^{-n}( T^\ast X; \omega_{T^\ast X}\otimes \orsh_{T^\ast X/X} )
        \quad ( \simeq H_n^\BM(\Lambda_{0};\orsh_{T^\ast X/X}) )
    \end{equation}
    and there exists a subanalytic cycle 
    \begin{align}
        C_J &\in H_{\Lambda_J}^{-n}(T^\ast X\times J; (\omega_{T^\ast X}\otimes \orsh_{T^\ast X/X})\boxtimes\CC_{J}) \\
        & \quad \simeq H_{\Lambda_J}^{-n}(T^\ast X\times\tl{I}; (\omega_{T^\ast X}\otimes \orsh_{T^\ast X/X})\boxtimes\CC_{J})
    \end{align}
    such that
    $\phi^\ast(i_0) (C_J)=\displaystyle\lim_{t\to+0}C_t$. 
    Similarly, for $C_I^\prime$ we define 
    \begin{align}
    \displaystyle\lim_{t\to+0}C_t^\prime &\in H_{\Lambda_0^\prime}^{-n}( T^\ast X; \omega_{T^\ast X}\otimes \orsh_{T^\ast X/X} ), \\
    C_J^\prime  &\in H_{\Lambda_J^\prime}^{-n}(T^\ast X\times\tl{I}; (\omega_{T^\ast X}\otimes \orsh_{T^\ast X/X})\boxtimes\CC_{J}).
    \end{align}
    We assume that the projection $t\colon T^\ast X\times \tl{I}\longrightarrow \tl{I}$ is proper on $\Lambda_J\cap \Lambda_J^\prime$ and $ t(\Lambda_J\cap\Lambda_J^\prime)=J$.
    Then we obtain a locally constant function $\#(C_J\cap C_J^\prime)$ on $J$.
    By Proposition \ref{prop-intersection}, for any $a\in J$ we have
    \begin{align}
        \#((\lim_{t\to+0}C_t)\cap (\lim_{t\to+0}C_t^\prime))
        &=\#(C_J\cap C_J^\prime)(0) \\
        &=\#(C_J\cap C_J^\prime)(a) \\
        &=\#(C_{a}\cap C_a^\prime),
    \end{align}
    where $C_{a}\coloneq \phi^\ast(i_a) (C_J) \in 
    H_{\Lambda_a}^{-n}( T^\ast X; \omega_{T^\ast X}\otimes \orsh_{T^\ast X/X} )$ 
    and $C_{a}^\prime\coloneq \phi^\ast(i_a) (C_J^\prime)\in 
    H_{\Lambda_a^\prime}^{-n}( T^\ast X;$ $\omega_{T^\ast X}\otimes \orsh_{T^\ast X/X} )$.
\end{example}

\subsection{Specialization and microlocalization functors}
In this subsection, we recall the definitions of the specialization and 
microlocalization functors and prove some elementary properties on them.
Let $X$ be a real analytic manifold and $M\subset X$ its closed submanifold.
Then we can define a new manifold $\widetilde{X_M}$ called the normal deformation 
of $X$ along $M$ and two morphisms:

\begin{equation}
\begin{cases}
   p\colon \widetilde{X_M} \longrightarrow X, \\
   t\colon \widetilde{X_M} \longrightarrow \RR,
\end{cases} 
\end{equation}
such that 

\begin{equation}
   \begin{cases}
      p^{-1} \( X\setminus M\) \simto (X\setminus M)\times (\RR \setminus \{ 0\}), \\
      t^{-1} \( \RR \setminus \{ 0\}\) \simto X\times (\RR \setminus \{ 0\}), \\
      t^{-1} (0) \simto T_M X .
   \end{cases}
\end{equation}
For the details, see Kashiwara-Schapira \cite[Section 4.1]{KS90}.
In this situation, we have the following commutative diagram:

\begin{equation}
   \xymatrix{
      T_M X \ar@{^{(}->}[r]^s \ar[d]^{\tau} & \widetilde{X_M} \ar[d]^p & \Omega \ar@{_{(}->}[l]_j \ar[ld]^{\widetilde{p}}\\
      M \ar@{^{(}->}[r]^i & X, & {}
   }
\end{equation}
where we set $\Omega \coloneq t^{-1}(\RR_{>0})\subset \widetilde{X_M}$ . Then for
$F\in \mathbf{D}^{\mathrm{b}} (X)$, we define the specialization (resp. microlocalization)
of $F$ along $M$ by:
$$
\nu_M (F) \coloneq s^{-1} \mathrm{R} j_\ast \widetilde{p}^{-1} F \quad 
\text{(resp.} \medspace 
\mu_M (F) \coloneq (\nu_M (F))^{\wedge}
\text{)},
$$
where the functor 
$(\cdot)^{\wedge} \colon \mathbf{D}^{\mathrm{b}}(T_M X) \longrightarrow \mathbf{D}^{\mathrm{b}}(T_M^\ast X)$
is the Fourier-Sato transform.
Let $f\colon Y \longrightarrow X$ be a morphism of manifolds and $M$ (resp. $N$)
a closed submanifold of $X$ (resp. $Y$).
Assume that $f(N) \subset M$, and consider the morphism 
$$
T_N f \colon T_N Y \longrightarrow N\times_M T_M X \longrightarrow T_M X 
$$
induced by $f$. Then we have 
canonical morphisms (see \cite[Proposition 4.2.5, 4.3.5]{KS90}):
\begin{align}
   \alpha_f & \colon \( T_N f \)^{-1} \nu_M (F) \longrightarrow \nu_N (f^{-1} F) \\
   \beta_f & \colon \mathrm{R}{\rho_f}_! \left( \omega_{N/M} \otimes \varpi_f^{-1} \mu_M (F)\right) \longrightarrow
   \mu_N (\omega_{Y/X} \otimes f^{-1}F),
\end{align}
where the morphisms $\rho_f$ and $\varpi_f$ are induced by $f\colon Y \longrightarrow X$, 
as illustrated in the diagram below:
\begin{equation}
   \xymatrix{
      T_N^\ast Y & N\times_M T_M^\ast X \ar[l]_-{\rho_f} \ar[r]^-{\varpi_f} & T_M^\ast X .
   }
\end{equation}

\begin{proposition}\label{prop-special}
Let $f\colon Y \longrightarrow X$ and $g\colon Z \longrightarrow Y$ be morphisms
of manifolds and M (resp. $N, L$) a closed submanifold of $X$ (resp. $Y, Z$).
Assume $f(N) \subset M$ and $g(L)\subset N$. Set $h \coloneq f\circ g$.
Then the following diagram commutes:
\begin{equation}\label{eq-sp}
   \xymatrix{
      (T_L g)^{-1} (T_N f)^{-1} \nu_M(F) \ar[r]^{\alpha_f} \ar[d]_{\wr} & (T_L g)^{-1} \nu_N (f^{-1}F) \ar[d]^{\alpha_g} \\
      (T_L h)^{-1}\nu_M(F) \ar[r]^{\alpha_h}& \nu_L(h^{-1}F).
   }
\end{equation}
\end{proposition}
\begin{proof}
By the morphism $t_X\colon \widetilde{X_M} \longrightarrow \RR$ (resp. $t_Y\colon \widetilde{Y_N} \longrightarrow \RR$, $t_Z\colon \widetilde{Z_L} \longrightarrow \RR$)
we set $\Omega_X \coloneq t_X^{-1}(\RR_{>0})\subset \widetilde{X_M}$ (resp. $\Omega_Y \coloneq t_Y^{-1}(\RR_{>0})\subset \widetilde{Y_N}, \: \Omega_Z \coloneq t_Z^{-1} (\RR_{>0}) \subset \widetilde{Z_L}$) and denote the inclusion map $\Omega_X \longhookrightarrow \widetilde{X_M}$ (resp. $\Omega_Y \longhookrightarrow \widetilde{Y_N}$, $\Omega_Z \longhookrightarrow \widetilde{Z_L}$) by $j_X$ (resp. $j_Y$, $j_Z$).
Then we obtain the commutative diagram
\begin{equation}
    \xymatrix{
    \widetilde{Z_L} \ar[d]^{\widetilde{g^\prime}} \ar@{}[rd]|\Box & \Omega_Z \ar@{_{(}->}[l]_-{j_Z} \ar[d]_-{\widetilde{g}} \\
    \widetilde{Y_N} \ar[d]^{\widetilde{f^\prime}} \ar@{}[rd]|\Box & \Omega_Y \ar@{_{(}->}[l]_-{j_Y} \ar[d]_-{\widetilde{f}} \\
    \widetilde{X_M}  & \Omega_X, \ar@{_{(}->}[l]_-{j_X} 
    }
\end{equation}
where $\widetilde{f^\prime}$ (resp. $\widetilde{g^\prime}$) is induced by $f$ (resp. $g$) and 
$\widetilde{f}$ (resp. $\widetilde{g}$) is the restriction of $\widetilde{f^\prime}$ (resp. $\widetilde{g^\prime}$ ) to $\Omega_Y$ (resp. $\Omega_Z$).
By the proof of \cite[Proposition 4.2.5]{KS90}, we have only to show that the following diagram commutes:
\begin{equation}
    \xymatrix{
    \widetilde{g^\prime}^{-1} \widetilde{f^\prime}^{-1} \mathrm{R} {j_X} _\ast \ar[r] \ar@{=}[d] &
    \widetilde{g^\prime}^{-1} \mathrm{R} {j_Y}_\ast \widetilde{f}^{-1} \ar[d] \\
    \widetilde{g^\prime}^{-1} \widetilde{f^\prime}^{-1} \mathrm{R} {j_X} _\ast \ar[r] &
    \mathrm{R} {j_Z}_\ast \widetilde{g}^{-1} \widetilde{f}^{-1}.
    }
\end{equation}
But this follows from Lemma \ref{lem-bs}.
\end{proof}

By applying the Fourier-Sato transform to the diagram (\ref{eq-sp}), 
we also obtain the following result.

\begin{proposition}\label{prop-micro}
In the situation of Proposition \textup{\ref{prop-special}},
the following diagram commutes:

\begin{equation}
   \xymatrix{
      \mathrm{R}{\rho_g}_! \( \omega_{L/N} \otimes \varpi_g^{-1} 
      \(\mathrm{R}{\rho_f}_!(\omega_{N/M}\otimes \varpi_f^{-1}\mu_M(F))\)\) \ar[r]^-{\beta_f} \ar@{}[d]|{\rotatebox{90}{$\simeq$}} &
      \mathrm{R}{\rho_g}_! \( \omega_{L/N} \otimes \varpi_g^{-1} \mu_N (\omega_{Y/X} \otimes f^{-1} F)\) \ar[d]^{\beta_g} \\
      \mathrm{R}{\rho_h}_! \( \omega_{L/M} \otimes \varpi_h^{-1} \mu_M(F)\) \ar[r]^{\beta_h} & 
      \mu_L \( \omega_{Z/X} \otimes h^{-1} F\).
   }
\end{equation}
\end{proposition}
Let $p_1$ and $p_2$ be the first and the second projections of $X \times X$ to $X$ respectively.
Recall also that for $F,G\in \mathbf{D}^{\mathrm{b}} (X)$,  $\mu hom (G,F) \in \mathbf{D}^{\mathrm{b}} (T^\ast X)$  
is defined by
\begin{equation}
   \mu hom(G, F) \coloneq \mu_{\Delta_X} \( \rhom (p_2^{-1}G, p_1^!F)\),
\end{equation}
where $\Delta_X \subset X\times X$ is the diagonal of $X\times X$ and we used the natural
identification $T_{\Delta_X}^\ast (X\times X) \simeq T^\ast X$.
Let $\pi_X\colon T_{\Delta_X} ^\ast (X\times X) \simeq T^\ast X \longrightarrow \Delta_X 
(\simeq X)$ be the projection.
Then by \cite[Proposition 4.4.2 (i)]{KS90} there exists an isomorphism
\begin{equation}
   \mathrm{R} {\pi_X}_\ast \mu hom (G, F) \simeq \rhom (G, F).
\end{equation}

\section{Definition of relative characteristic cycles}\label{sec-relCC}
In this section, we define relative characteristic cycles of $\RR$-constructible
sheaves by using Ishimura's functor of relative $\muhom$ introduced in \cite{Ish92}.
Let $f\colon X\longtwoheadrightarrow\Sigma$ be a smooth morphism of $C^\infty$-manifolds 
and define the relative cotangent bundle $T^\ast(X/\Sigma)$ of $f$ by the exact sequence
\begin{equation}
0\longrightarrow X\ftimes{\Sigma} T^\ast\Sigma \underset{\alpha}{\longrightarrow}
T^\ast X \underset{\beta}{\longrightarrow} T^\ast(X/\Sigma)\longrightarrow 0
\end{equation}
of vector bundles on $X$.
Let 
\begin{equation}
\Delta_X(\simeq X) \underset{\gamma_f}{\longhookrightarrow} X\ftimes{\Sigma}X
\underset{\delta_f}{\longhookrightarrow} X\times X
\end{equation}
be the inclusion maps and $\delta_X\colon \Delta_X(\simeq X)
\longhookrightarrow X\times X$ (resp. $\delta_\Sigma\colon
\Delta_\Sigma(\simeq\Sigma)\longhookrightarrow\Sigma\times\Sigma$)
the diagonal embedding of $X$ (resp. $\Sigma$).
Then we have $\delta_X=\delta_f\circ\gamma_f$ and for the submanifold 
$\Delta_f\coloneq\gamma_f(\Delta_X)(\simeq X)$ of $X\ftimes{\Sigma}X$ there exists an isomorphism
\begin{equation}
T^\ast(X/\Sigma)\simeq T_{\Delta_f}^\ast(X\ftimes{\Sigma}X).
\end{equation}
Moreover the surjective morphism $\beta\colon T^\ast X
\longtwoheadrightarrow T^\ast (X/\Sigma)$ is naturally identified with the one 
\begin{equation}
T_{\Delta_X}^\ast(X\times X)\longtwoheadrightarrow T_{\Delta_f}^\ast (X\ftimes{\Sigma} X)
\end{equation}
induced from $\delta_f\colon X\ftimes{\Sigma} X
\longhookrightarrow X\times X$ (see Ishimura \cite[Section 1.2]{Ish92}).
Note also that for any point $a\in \Sigma$ of $\Sigma$ and the submanifold 
$f^{-1}(a)\subset X$ of $X$ associated to it we have a natural isomorphism
\begin{equation}
f^{-1}(a)\ftimes{X} T^\ast(X/\Sigma)\simto T^\ast (f^{-1}(a)).
\end{equation}
Let $k$ be a commutative field of characteristic $0$ and denote by $\BDC(X)$ the 
derived category of bounded complexes of sheaves of $k_X$-modules on $X$.
Then by using the conicness of the microsupport $\msupp(F)
\subset T^\ast X$ of $F\in \BDC(X)$ we can easily show the following lemma.

\begin{lemma}\label{lem-adopted}
By the morphism $\alpha\colon X\ftimes{\Sigma}T^\ast\Sigma
\longhookrightarrow T^\ast X$ we identify $X\ftimes{\Sigma} 
T^\ast\Sigma$ with a subbundle of $T^\ast X$.
Then for $F\in\BDC(X)$ the following three conditions are equivalent:
\begin{enumerate}
    \item [\rm (i)]
    The restriction $\beta\vbar_{\msupp(F)}\colon\msupp(F)\longrightarrow T^\ast(X/\Sigma)$ of 
    $\beta\colon T^\ast X\longtwoheadrightarrow T^\ast(X/\Sigma)$ 
to $\msupp(F)\subset T^\ast X$ is proper,
    \item [\rm (ii)]
    $\msupp(F)\cap (X\ftimes{\Sigma}T^\ast\Sigma)\subset T_X^\ast X$,
    \item [\rm (iii)]
    For any point $a\in \Sigma$ of $\Sigma$ the inclusion map 
$f^{-1}(a)\longhookrightarrow X$ of the submanifold $f^{-1}(a)\subset X$ 
is non-characteristic for $F\in \BDC(X)$. 
\end{enumerate}
\end{lemma}

\begin{definition}\label{def-adapted}
We say that $F\in\BDC(X)$ is adapted to $f\colon X\longrightarrow\Sigma$ if it 
satisfies the conditions of Lemma \ref{lem-adopted}.
\end{definition}

\begin{remark}
Let $f\colon X\longrightarrow\Sigma$ be a smooth morphism of
real analytic manifolds and $F\in\BDrc(X)$ an
$\RR$-constructible sheaf on $X$. Assume that $f$ is
proper on the support of $F$. Then by the results in
\cite[page 43]{GM88} there exists a subanalytic open
dense subset $\Omega \subset \Sigma$ of $\Sigma$
such that $F|_{f^{-1}( \Omega )} \in\BDrc( f^{-1}( \Omega ) )$ is
adapted to the restriction $f|_{f^{-1}( \Omega )}:
f^{-1}( \Omega ) \longrightarrow \Omega$ of $f$.
\end{remark}

\begin{definition}[c.f. {\cite[D\'efinition 1.1]{Ish92}}]
If $F\in \BDC(X)$ is adapted to $f\colon X\longrightarrow\Sigma$, we set
\begin{equation}
\msupp_\Sigma(F)\coloneq \beta(\msupp(F)) \quad \subset T^\ast(X/\Sigma)
\end{equation}
and call it the relative microsupport of $F$.
\end{definition}

In what follows, we assume that $X$, $\Sigma$ and $f\colon X\longrightarrow\Sigma$ are
``real analytic'' and consider only $\RR$-constructible sheaves $F\in\BDrc(X)$ on $X$. 
Then by Lemma \ref{lem-adopted} we can easily show the following result.

\begin{lemma}
Assume that $F\in\BDrc(X)$ is adapted to $f\colon X\longrightarrow\Sigma$.
Then for any point $a\in\Sigma$ of $\Sigma$
\begin{equation}
\msupp_\Sigma(F)\cap \bigl\{f^{-1}(a)\ftimes{X}T^\ast (X/\Sigma)\bigr\}
\end{equation}
is a closed conic subanalytic isotropic subset of
\begin{equation}\label{fib}
f^{-1}(a)\ftimes{X}T^\ast (X/\Sigma)\simto T^\ast(f^{-1}(a)).
\end{equation}
\end{lemma}

Let $p_1$ and $p_2$ (resp. $q_1$ and $q_2$) be the first and the 
second projections of $X\times X$ (resp. $X\ftimes{\Sigma}X$) to $X$.
Then as in Ishimura \cite[D\'efinition 1.2]{Ish92} for $F,G\in \BDrc(X)$ we set 
\begin{equation}
\muhom_\Sigma(G,F)\coloneq \mu_{\Delta_f}\rhom(q_2^{-1}G,q_1^!F) 
\quad \in\BDrc(T^\ast(X/\Sigma))
\end{equation}
and call it the relative $\muhom$ of $F$ and $G$.
Note that if $F,G\in \BDrc(X)$ are adapted to $f\colon X\longrightarrow\Sigma$ 
the inclusion map $\delta_f\colon X\ftimes{\Sigma}X\longhookrightarrow X\times X$
is non-characteristic for 
\begin{equation}
F\boxtimes D_X(G)\simeq \rhom(p_2^{-1}G,p_1^!F)
\end{equation}
and hence by Kashiwara-Schapira \cite[Proposition 4.3.5]{KS90} we obtain a morphism
\begin{align}
\rmR\beta_\ast\muhom(G,F) &\simeq \rmR\beta_\ast\mu_{\Delta_X}(F\boxtimes D_X(G)) \\
& \longrightarrow 
\mu_{\Delta_f}\bigl(\delta_f^{-1}(F\boxtimes D_X(G))\otimes\delta_f^! k_{X\times X}\bigr).
\end{align}
Moreover, by \cite[Proposition 5.4.13(ii)]{KS90} there exist isomorphisms
\begin{equation}
\delta_f^{-1}(F\boxtimes D_X(G))\otimes\delta_f^! k_{X\times X}
\simto \delta_f^!(F\boxtimes D_X(G)) \simeq \rhom(q_2^{-1}G,q_1^!F).
\end{equation}
For $F,G\in\BDrc(X)$ adapted to $f\colon X\longrightarrow\Sigma$ we thus obtain a morphism
\begin{equation}
\rmR\beta_\ast \muhom(G,F)\longrightarrow\muhom_\Sigma(G,F).
\end{equation}
Note also that by Ishimura \cite[Proposition 1.4]{Ish92} for such $F,G\in\BDrc(X)$
we have 
\begin{equation}
\supp\muhom_\Sigma(G,F)\subset \msupp_\Sigma(F)\cap \msupp_\Sigma(G).
\end{equation}
Let $h\colon X\ftimes{\Sigma}X\longtwoheadrightarrow\Sigma$ be 
the natural morphism induced by $f\colon X\longrightarrow\Sigma$.

\begin{lemma}\label{lem-reldual}
There exists an isomorphism 
\begin{equation}
\delta_f^! k_{X\times X} \simeq h^{-1}\omega_\Sigma^{\otimes-1}.
\end{equation}
\end{lemma}

\begin{proof}
For the smooth morphism $g\coloneq f\times f\colon X\times X
\longtwoheadrightarrow\Sigma\times\Sigma$ let us consider the Cartesian diagram
\begin{equation}
\begin{tikzcd}
X\ftimes{\Sigma} X \rar["\delta_f",hook] \dar["h"'] \drar[phantom, 
"\square",pos=.4]  & X\times X \dar["g"] \\
\Sigma \rar["\delta_\Sigma"',hook] & \Sigma\times\Sigma.
\end{tikzcd}
\end{equation}
Then by the smoothness of $g$ we obtain isomorphisms 
\begin{equation}
\delta_f^! k_{X\times X}\simeq\delta_f^! g^{-1}k_{\Sigma\times\Sigma}\simeq
h^{-1}\delta_\Sigma^!k_{\Sigma\times\Sigma}\simeq h^{-1}\omega_\Sigma^{\otimes-1}.
\end{equation}
\end{proof}

Let $\pi_X\colon T^\ast X\longrightarrow X$ and $\pi_{X/\Sigma}
\colon T^\ast(X/\Sigma)\longrightarrow X$ be the projections.
Then by Lemma \ref{lem-adopted} for $F\in \BDrc(X)$ adapted to 
$f\colon X\longrightarrow\Sigma$ we obtain a chain of morphisms
\begin{align}\label{def-relcc}
\begin{split}
\rhom(F,F) &\simeq \rmR\pi_{X\ast}\muhom(F,F) \\
&\simeq \rmR\pi_{X/\Sigma\ast}\rmR\beta_\ast \mu_{\Delta_X}(F\boxtimes D_X(F)) \\
& \longrightarrow \rmR\pi_{X/\Sigma\ast} \mu_{\Delta_f} \Bigl( \delta_f^{-1}
(F\boxtimes D_X(F))\otimes\delta_f^!k_{X\times X} \Bigr) \\
&\simeq \rmR\pi_{X/\Sigma\ast}\rsect_{\msupp_\Sigma(F)} \mu_{\Delta_f} 
\Bigl( (F\fboxtimes{\Sigma} D_X(F))\otimes h^{-1}\omega_\Sigma^{\otimes-1} \Bigr) \\
& \longrightarrow \rmR\pi_{X/\Sigma\ast}\rsect_{\msupp_\Sigma(F)} 
\mu_{\Delta_f} \Bigl(
\gamma_{f\ast}\bigl(F\otimes\rhom(F,\omega_{X/\Sigma})\bigr)\Bigr) \\
& \longrightarrow \rmR\pi_{X/\Sigma\ast}\rsect_{\msupp_\Sigma(F)} \mu_{\Delta_f} 
(\gamma_{f\ast}\omega_{X/\Sigma}) \\
&\simeq  \rmR\pi_{X/\Sigma\ast}\rsect_{\msupp_\Sigma(F)} 
( \pi_{X/\Sigma}^{-1}\omega_{X/\Sigma} ),
\end{split}
\end{align}
where to construct the third morphism we used \cite[Proposition 4.3.5]{KS90} and the fifth morphism is induced by the one 
\begin{equation}
K\longrightarrow \gamma_{f\ast}\gamma_f^{-1}K
\simeq \gamma_{f\ast} (F\otimes\rhom(F,\omega_{X/\Sigma}))
\end{equation}
for $K\coloneq\delta_f^{-1}(F\boxtimes D_X(F))\otimes h^{-1}\omega_\Sigma^{\otimes-1}
\in\BDrc(X\ftimes{\Sigma}X)$.
Applying the global section functor $\rsect(X;\, \cdot\, )$ to it and 
taking the 0-th cohomology groups we obtain a morphism 
\begin{equation}
\Hom(F,F)\longrightarrow H_{\msupp_\Sigma(F)}^0\bigl(T^\ast(X/\Sigma);
\pi_{X/\Sigma}^{-1}\omega_{X/\Sigma}\bigr).
\end{equation}

\begin{definition}
We call the image of $\id_F\in \Hom(F,F)$ in $H_{\msupp_\Sigma(F)}^0\bigl(T^\ast(X/\Sigma);
\pi_{X/\Sigma}^{-1}\omega_{X/\Sigma}\bigr)$ the relative characteristic 
cycle of $F\in \BDrc(X)$ for $f$ and denote it by 
$\CCyc_\Sigma(F)$.
\end{definition}

For a point $a\in \Sigma$ of $\Sigma$ set $X_a\coloneq f^{-1}(a)\subset X$ and let 
$i_a\colon X_a\longhookrightarrow X$ be its inclusion map.
Then as in the proof of Lemma \ref{lem-reldual} we can easily obtain an isomorphism 
$i_a^{-1}\omega_{X/\Sigma}\simeq\omega_{X_a}$.
This implies that there exists a morphism 
\begin{equation}
\Phi_a\colon H_{\msupp_\Sigma(F)}^0\bigl(T^\ast(X/\Sigma)
;\pi_{X/\Sigma}^{-1}\omega_{X/\Sigma}\bigr)
\longrightarrow 
H_{\msupp_\Sigma(F)\cap T^\ast X_a}^0 (T^\ast X_a;\pi_{X_a}^{-1}\omega_{X_a}),
\end{equation}
where $\pi_{X_a}\colon T^\ast X_a\longrightarrow X_a$ is the projection.
Recall that $\msupp_\Sigma(F)\cap T^\ast X_a\subset 
T^\ast X_a$ is a closed conic subanalytic isotropic subset of $T^\ast X_a$.
We thus can consider our relative characteristic cycle 
$\CCyc_\Sigma(F)\in H_{\msupp_\Sigma(F)}^0
\bigl(T^\ast(X/\Sigma);\pi_{X/\Sigma}^{-1}\omega_{X/\Sigma}\bigr)$ 
as a family of Lagrangian cycles 
$\Phi_a(\CCyc_\Sigma(F)) \in \sect(T^\ast X_a;\SL_{X_a})$ 
($a \in \Sigma$) parameterized by $\Sigma$.
On the other hand, for the restriction $F_a \coloneq i_a^{-1}F\simeq F\vbar_{X_a}$
of $F\in\BDrc(X)$ to $X_a=f^{-1}(a) \subset X$ we have the characteristic cycle
\begin{equation}
\CCyc(F_a) \quad \in H_{\msupp(F_a)}^0(T^\ast X_a; \pi_{X_a}^{-1}\omega_{X_a})
\end{equation}
of Kashiwara and Schapira (see \cite[Section 9.4]{KS90}) and 
$\msupp(F_a)\subset \msupp_\Sigma(F)\cap T^\ast X_a$.
Then it would be very natural to expect that we have the following result. 

\begin{theorem}\label{thm-restofrelCC}
Assume that $F\in \BDrc(X)$ is adapted to $f\colon X\longrightarrow\Sigma$.
Then for any point $a\in \Sigma$ of $\Sigma$ we have
$\Phi_a(\CCyc_\Sigma(F))=\CCyc(F_a)$ in $\sect(T^\ast X_a; \SL_{X_a})$.
\end{theorem}

\begin{proof}
Our proof is similar to that of \cite[Proposition 9.4.3]{KS90}. 
Now let $\delta_{X_a} \colon X_a\longhookrightarrow X_a\times X_a$ be the diagonal embedding 
and set $\Lambda \coloneq \mathrm{SS}_{\Sigma} (F) \cap T^\ast X_a $.
Then it suffices to show the commutativity of Diagram 3.a below, where we denote
$\mathrm{R} \Gamma, \rhom, \ldots$ simply by $\Gamma, \mathcal{H}om, \ldots$ respectively.

\begin{center}
$
     \xymatrix{
        \Hom (F,F) \ar[r]  &\Hom (F_a, F_a) \\
        \sect_{\Delta_X} \( X^2; F\boxtimes \mathrm{D}_X F\) 
        \ar[u]^{\wr} \ar[r]^{\textcircled{\scriptsize 1}} 
        \ar[d]_{\textcircled{\scriptsize 2}} \ar@{}[rd]|{\text{\scalebox{1.2}{$\mathbf{A}$}}} &
        \sect_{\Delta_{X_a}} \( {X_a}^2; F_a \boxtimes \mathrm{D}_{X_a} F_a\) \ar[u]^{\wr} \ar@{=}[d] \\
        \sect_{\Delta_f} \( X\times_{\Sigma} X; 
        \delta_f^{-1}(F\boxtimes D_X(F))\otimes\delta_f^!k_{X\times X}\) 
        \ar[r]_-{\textcircled{\scriptsize 3}} \ar[d]_{\wr} &
        \sect_{\Delta_{X_a}} \( {X_a}^2; F_a \boxtimes \mathrm{D}_{X_a} F_a\)  \ar[d]_{\wr} \\
        \sect_{\mathrm{SS}_{\Sigma} (F)} 
        \( T^\ast (X/\Sigma); \mu_{\Delta_f} 
        \( \delta_f^{-1}(F\boxtimes D_X(F))\otimes\delta_f^!k_{X\times X}\) \) 
        \ar[d] \ar[r] & 
        \sect_{\Lambda} \( T^\ast X_a; \mu_{\Delta_{X_a}}\( F_a\boxtimes \mathrm{D}_{X_a} F_a\) \) \ar[d] \\
        \sect_{\mathrm{SS}_{\Sigma} (F)} 
        \( T^\ast \( X/\Sigma\) ; \mu_{\Delta_f}\( {\gamma_f}_\ast \( F\otimes  \shom ( F,\omega_{X/\Sigma} ) \) \) \) 
        \ar[d] \ar[r] & 
        \sect_{\Lambda} (T^\ast X_a; \mu_{\Delta_{X_a}}({\delta_{X_a}}_\ast  (F_a\otimes \mathrm{D}_{X_a} F_a))) \ar[d] \\
        \sect_{\mathrm{SS}_{\Sigma} (F)} \( T^\ast \( X/\Sigma\) ; \mu_{\Delta_f}\( {\gamma_f}_\ast \omega_{X/\Sigma} \) \) \ar[r] \ar[d]_{\wr} & 
        \sect_{\Lambda} (T^\ast X_a; \mu_{\Delta_{X_a}}({\delta_{X_a}}_\ast \omega_{X_a}) ) \ar[d]_{\wr} \\
        \sect_{\mathrm{SS}_{\Sigma} (F)} ( T^\ast \( X/\Sigma\); \pi_{X/\Sigma}^{-1} \omega_{X/\Sigma}) \ar[r] &
        \sect_{\Lambda} \( T^\ast X_a; \pi_{X_a}^{-1} \omega_{X_a}\)  \\
         }
$
\bigskip 
Diagram 3.a 
\end{center}
Here, the vertical arrows on the left are constructed by 
the chain of morphisms of (\ref{def-relcc}). 
Moreover, the vertical ones on the right 
are constructed as in \cite[Proposition 9.4.3]{KS90}.
Let us explain the construction of the horizontal arrows in Diagram 3.a.
For this purpose, 
let $i_a \colon X_a \longhookrightarrow X$ be the inclusion map, 
$f_1 \colon X_a^2 \longrightarrow X_a\times X$ the map $(\mathrm{id}_{X_a}, i_a)$,
$f_2 \colon X_a\times X \longrightarrow X^2$ the one $(i_a,\mathrm{id}_X)$ and
$\Gamma_{i_a} \subset X_a\times X $ the graph of $i_a$.
By the diagram
\begin{equation}
   \xymatrix{
      X_a^2 \ar[r]^-{f_1} & X_a \times X \ar[r]^-{f_2} & X^2 \\
      \Delta_{X_a} \ar[r] \ar@{^{(}->}[u] & \Gamma_{i_a} \ar[r] \ar@{^{(}->}[u] & \Delta_X \ar@{^{(}->}[u]
   }
\end{equation}
and \cite[Proposition 4.3.5 (i)]{KS90}
we obtain the following morphisms
\begin{align}
   &\varpi_{i_a}^{-1} \circ \mu_{\Delta_X} \to \mu_{\Gamma_{i_a}} \circ f_2^{-1}, \label{upper-left-mor}\\
   &{\rho_{i_a}}_! \circ \mu_{\Gamma_{i_a}} \to \mu_{\Delta_{X_a}} \circ (\omega_{X_a/X} \otimes f_1^{-1}). \label{upper-right-mor}
\end{align}
Then the morphism \textcircled{\scriptsize 1} in Diagram 3.a is constructed 
by (\ref{upper-left-mor}) and (\ref{upper-right-mor}).
On the other hand,
we define a map 
$\gamma_a: X_a^2 \longrightarrow X\times_{\Sigma} X$ by
$(x,y)\longmapsto (x,y) \in X\times_{\Sigma} X$ for $(x,y)\in X_a^2$.
We denote the inclusion map 
$X_a\times_X T^\ast(X/\Sigma) \simeq T^\ast X_a \longhookrightarrow T^\ast(X/\Sigma)$ by 
$\iota$.
Then by the diagram
\begin{equation}
   \xymatrix{
      X_a^2 \ar[r]^-{\gamma_a} & X\times_{\Sigma} X \ar[r]^-{\delta_f} & X^2 \\
      \Delta_{X_a} \ar@{^{(}->}[u] \ar[r] & \Delta_f \ar@{^{(}->}[u] \ar[r] & \Delta_X \ar@{^{(}->}[u]  
   }
\end{equation}
and \cite[Proposition 4.3.5 (i)]{KS90}
we obtain the natural morphisms
\begin{align}
   &\beta_! \circ \mu_{\Delta_X} \to \mu_{\Delta_f} \circ (\delta_f^{-1}\otimes \delta_f^!k_{X\times X}), \label{vert-mor}\\
   &\iota^{-1} \circ \mu_{\Delta_f} \to \mu_{\Delta_{X_a}} \circ \gamma_a^{-1}. \label{lower-mor}
\end{align}
It follows that the morphisms \textcircled{\scriptsize 2} and \textcircled{\scriptsize 3} are 
constructed by (\ref{vert-mor}) and (\ref{lower-mor}), respectively.
The other vertical arrows in Diagram 3.a are constructed similarly.
Since the two morphisms $\delta_f \circ \gamma_a$ and $f_2\circ f_1$ are equal,
the commutativity of $\mathbf{A}$ in Diagram 3.a follows from Proposition \ref{prop-micro}.
We can also check the commutativity of the remaining squares in Diagram 3.a 
as in the proof of \cite[Proposition 9.4.3]{KS90}.
This completes the proof.
\end{proof}

\section{Real nearby cycle sheaves and
their characteristic cycles}\label{sec-realnearby}

In this section, given a real analytic function and an 
$\RR$-constructible sheaf on a real analytic manifold, 
we define a real nearby cycle sheaf associated to them 
and study its basic properties.
Then we obtain a formula for its characteristic cycle.
Let $X$ be a real analytic manifold and 
$f \colon X \longrightarrow \RR$ a non-constant real analytic function on it.
Set $X_0 \coloneq f^{-1}(0) \subset X, \{f>0 \} \coloneq 
\{ x\in X \mid f(x)>0 \} \subset X$ and 
let $i_X \colon X_0 \longhookrightarrow X $ be the inclusion map.

\begin{definition}
    For an $\RR$-constructible sheaf $F \in 
    \BDrc (X)$ on $X$
    we set 
    \begin{equation}
        \psi_f^{\RR}(F) \coloneq i_X^{-1} \rsect_{\{f>0 \}} (F) \quad 
        \in \BDrc(X_0)
    \end{equation}
    and call it the real nearby cycle sheaf of $F$ along $f$.
\end{definition}
We thus obtain a functor 

\begin{equation}
    \psi_f^{\RR} (\cdot) \colon \BDrc (X)
     \longrightarrow \BDrc(X_0).
\end{equation}
This functor is related to the right Milnor fibers of 
$f\colon X \longrightarrow \RR$ 
defined as follows.
For a point $x_0 \in X_0$ of $X_0 =f^{-1} (0)$
and a sufficiently small open ball $B (x_0) \subset X$
centered at it 
we can show that there exists $0< \eta_0 \ll1$ such that 
for any $0< \eta <\eta_0$ the restriction

\begin{equation}
    B(x_0) \cap f^{-1}((0,\eta)) \longrightarrow
    (0,\eta)
\end{equation}
of $f\colon X \longrightarrow \RR$
is a fiber bundle over the open interval $(0,\eta) \subset \RR.$
We call its fiber $M_{f,x_0}^\RR \subset \{f>0 \}$
the right Milnor fiber of $f$ at the point $x_0 \in X_0.$
Then, as in the case of Deligne's nearby cycle functor defined for 
holomorphic functions in \cite{Del73}, 
we obtain the following result (see e.g. \cite[Proposition 4.2.2]{Dim04}, 
\cite{Le77} and the proof of \cite[Theorem 2.6]{Tak25}).

\begin{lemma}
    Let $F\in \BDrc (X)$ be an 
    $\RR$-constructible sheaf on $X$.
    Then for any point $x_0 \in X_0$ of $X_0=f^{-1}(0)$
    and any $j\in \mathbb{Z}$ there exists an isomorphism
    \begin{equation}
        H^j \psi_f^{\RR}(F)_{x_0} \simto H^j (M_{f,x_0}^\RR;F).
    \end{equation}
\end{lemma}
As an analogue of \cite[Proposition 4.2.11]{Dim04} and 
\cite[Exercise VIII.15]{KS90}, 
we have also the following results.

\begin{proposition}
    Let $\nu \colon Y \longrightarrow X$ be a proper morphism
    of real analytic manifolds and 
    $f\colon X\longrightarrow \RR$ a non-constant real analytic
    function on $X$.
    For the non-constant real analytic functions $f\colon X \longrightarrow
    \RR$ and $g\coloneq f\circ \nu \colon Y \longrightarrow \RR$
    we set $X_0 \coloneq f^{-1}(0) \subset X$ and 
    $Y_0\coloneq g^{-1}(0)=\nu^{-1}(X_0) \subset Y$.
    Then for any $\RR$-constructible sheaf $G\in \BDrc(Y)$
    on $Y$ there exists an isomorphism 
    \begin{equation}
        \psi_f^{\RR}(\rmR \nu_\ast G) \simto \rmR {\nu_0}_\ast \psi_g^{\RR}(G),
    \end{equation}
    where $\nu_0 \colon Y_0 \longrightarrow X_0$ is the restriction
    of $\nu \colon Y \longrightarrow X$ to $Y_0\subset Y$.
\end{proposition}

\begin{proof}
    Let us consider the Cartesian diagram
    \begin{equation}
        \xymatrix{
            Y_0 \ar@{^{(}-_>}[r]^{i_Y} \ar[d]_{\nu_0} \ar@{}[rd]|{\Box} & 
            Y \ar[d]^{\nu} \\
            X_0 \ar@{^{(}-_>}[r]_{i_X} & X.
        }
    \end{equation}
    Then we obtain isomorphisms
    \begin{align}
        \psi_f^{\RR} (\rmR \nu_\ast G) & =i_X^{-1} \rsect_{\{ f>0 \}} (\rmR \nu_\ast G) \\
        & \simeq i_X^{-1} \rmR \nu_\ast \rsect_{\{ g>0 \}}(G) \\
        & \simeq \rmR {\nu_0}_\ast i_Y^{-1} \rsect_{\{ g>0 \}}(G) \\
        & \simeq \rmR {\nu_0}_\ast \psi_g^{\RR} (G).
    \end{align}
\end{proof}

\begin{corollary}
    Let $f\colon X \longrightarrow \RR$ be a non-constant real 
    analytic function and $i_f \colon X \longhookrightarrow X\times \RR_t$ 
    $(x\longmapsto (x,f(x)))$
    the graph embedding by it.
    Then for $F\in \BDrc(X)$,
    the real analytic function $t\colon X\times \RR \longrightarrow \RR$ $((x,t)\longmapsto t)$
    on $X\times \RR$ and the inclusion map 
    $\iota \colon X_0=f^{-1}(0) \longhookrightarrow X$
    we have an isomorphism
    \begin{equation}
        \psi_t^{\RR}({i_f}_\ast F) \simto \iota_\ast \psi_f^{\RR}(F).
    \end{equation}
\end{corollary}

By this corollary, for the study of the real nearby cycle sheaves, 
it suffices to consider only the case where 
$\mathrm{d} f(x_0)\neq 0$ for any point $x_0\in X_0 =f^{-1}(0) \subset X$
and hence $X_0=f^{-1}(0)$ is a smooth hypersurface in $X$.
Until the end of this section, we shall always assume that this 
condition on the non-constant real analytic function $f\colon X \longrightarrow \RR$
is satisfied and give a formula for the characteristic cycle 
$\C \C(\psi_f^{\RR}(F))$ of $\psi_f^{\RR}(F) \in \BDrc(X_0)$.
Shrinking $X$ if necessary, for an open interval
$(-\varepsilon , \varepsilon) \subset \RR$ $(\varepsilon >0)$
in $\RR$ we may assume also that the restriction
$f^{-1}((-\varepsilon, \varepsilon)) \longrightarrow (-\varepsilon, \varepsilon)$
of $f$ is smooth.
Therefore we replace $X$ by $f^{-1}((-\varepsilon, \varepsilon)) \subset X$. 
As the constructions of $\psi_f^{\RR}(F)$ and its characteristic
cycle being local, we may assume also that $f$ is 
proper on the support $\supp F$ of 
$F\in \BDrc(X)$. 
Indeed, for a point $x_0\in X_0$ of the smooth hypersurface 
$X_0=f^{-1}(0)\subset X$ of $X$, 
to describe the characteristic cycle $\C \C (\psi_f^{\RR}(F))$ 
on its neighborhood in $X_0$ for a relatively compact subanalytic
open neighborhood $U\subset X$ of $x_0$ in $X$ 
we may replace $F$ by $F_U = {k}_U \otimes_{k_X} F
\in \BDrc(X)$. For this reason, in what follows we assume that $f$ 
is proper on the support $\supp F$ of $F\in \BDrc(X)$. 
Then, by the microlocal Bertini-Sard theorem \cite[Proposition 8.3.12]{KS90} of Kashiwara and 
Schapira we can take $\varepsilon >0$ sufficiently small so that
for the open interval $I\coloneq (0, \varepsilon)\subset \RR$
in $\RR$ and the open subset $X_I\coloneq f^{-1}(I) \subset X$
the $\RR$-constructible  sheaf $F|_{X_I} \in \BDrc(X_I)$
on $X_I$ is adapted to the smooth map $f_I \colon X_I \longrightarrow I$
obtained by restricting $f$ to $X_I \subset X$.
Applying the constructions in Section \ref{sec-relCC} 
to $F|_{X_I} \in \BDrc(X_I) $
we can then associate its relative characteristic cycle 
\begin{equation}
    \CCyc_I (F|_{X_I}) \quad \in H^0_{{\mathrm{SS}}_I(F|_{X_I})} 
(T^\ast(X_I/I); \pi^{-1}_{X_I/I} \omega_{X_I/I}).
\end{equation}
For a point $a\in I$ in $I=(0, \varepsilon)\subset \RR$
set $X_a \coloneq f^{-1}(a) \subset X$
and $F_a \coloneq F|_{X_a} \in \BDrc(X_a)$.
Then by Theorem \ref{thm-restofrelCC}
we can consider $\C \C_I(F|_{X_I})$
as a family of Lagrangian cycles $\C \C (F_a) \quad (a\in I)$
parameterized by $I=(0,\varepsilon)$.

\begin{lemma}\label{ex-lem}
In the situation as above, 
the limit $\lim_{a\to +0} \C \C(F_a) $
of Lagrangian cycles $\C \C (F_a) \quad (a\in I=(0, \varepsilon))$
(in the sense of Section \ref{subsec-limit}) exists. 
\end{lemma}

\begin{proof}
The problem being local, we may assume that 
$X=X_0 \times (-\varepsilon, \varepsilon), 
X_I =X_0\times I$
and hence $T^\ast(X/(-\varepsilon, \varepsilon))=T^\ast X_0 \times (-\varepsilon, \varepsilon),
T^\ast(X_I/I) =T^\ast X_0 \times I$ and $\omega_{X_I/I} =\omega_{X_0} \boxtimes k_I$.
We thus can regard $\C \C_I (F|_{X_I})$
as a family of Lagrangian cycles $\C \C(F_a)$
in $T^\ast X_a \simeq T^\ast X_0 \quad (a\in I)$ 
parameterized by $I=(0, \varepsilon)$. 
Let us set $J\coloneq [0, \varepsilon) \subset \RR$
and construct a family $\{ \Xi_a \}_{a\in J}$
of closed conic subanalytic Lagrangian subsets $\Xi_a 
\subset T^\ast X_a \simeq T^\ast X_0 \quad (a\in J)$
parameterized by $J=[0, \varepsilon)$ 
such that the support of $\C \C_I(F|_{X_I})$ in 
$T^\ast (X_I/I) =T^\ast X_0 \times I$
is contained in its total space $\Xi \coloneq \bigsqcup_{a\in J} \Xi_a \times \{ a\} 
\subset T^\ast X_0 \times J$.
First, set $\{ f \geq 0 \} \coloneq X_0 \times J \subset X$ and 
let $\mathcal{S} =\{ S\subset X \}$ be a subanalytic 
Whitney stratification of $\supp F_{\{ f \geq 0 \}} \subset 
X=X_0\times (-\varepsilon, \varepsilon)$ 
satisfying the Thom condition $a_f$ (see. Kurdyka-Parusi\'{n}ski \cite{KP94}) 
such that
\begin{equation}
    \mathrm{SS} (F_{\{ f \geq 0 \}}) \subset \bigsqcup_{S\in \mathcal{S}} T^\ast_S X  
\end{equation} 
and for a subset $\mathcal{S}_0 \subset \mathcal{S}$
we have $\supp F_{\{ f \geq 0 \}} \cap 
(X_0\times \{ 0\}) =\bigsqcup_{S\in \mathcal{S}_0} S$.
Replacing $\varepsilon >0$ if necessary, 
we may assume that for any $a \in I=(0, \varepsilon)$
the smooth hypersurface $X_a=f^{-1}(a)\subset X$ 
intersects all the strata $S\in \mathcal{S}$ in $\mathcal{S}$
such that $S \subset \{ f>0 \}$
transversally (see e.g. Goresky-MacPherson \cite[p.43]{GM88}).
For $a \in J=[0, \varepsilon)$
we set 
\begin{equation}
    \Xi_a \coloneq
    \begin{cases}
        \displaystyle \bigsqcup_{S\in \mathcal{S}} T^\ast _{S\cap X_a} 
(X_a) \quad (a\in I=(0, \varepsilon)), \\[20pt]
        \displaystyle \bigsqcup_{S\in \mathcal{S}_0} T^\ast _S X_0 \quad (a=0).
    \end{cases}
\end{equation}
Note that for the natural morphism $\beta : T^*X_I \longrightarrow T^\ast (X_I/ I)$ 
we have 
\begin{equation}
\beta \Bigl( \bigsqcup_{S \in \mathcal{S} \setminus \mathcal{S}_0} T^*_SX \Bigr) 
=
\bigsqcup_{a\in I} \Xi_a \times \{ a \} \quad \subset T^\ast (X_I/I). 
\end{equation}
Then we see that $\mathrm{SS}_I (F|_{X_I}) \subset T^\ast (X_I/ I)$
is contained in $\bigsqcup_{a\in I} \Xi_a \times \{ a \} \subset T^\ast (X_I/I)$. 
By the $a_f$-condition of $\mathcal{S}$,
the closure of $\bigsqcup_{a\in I} \Xi_a \times \{ a \}$ 
in $T^\ast (X/(-\varepsilon, \varepsilon))$ intersects 
$T^*X_0 \times \{ 0 \} \simeq T^*X_0$ only in the subset 
$\Xi_0 \subset T^*X_0$. 
In other words, 
$\Xi \coloneq \bigsqcup_{a\in J} \Xi_a \times \{ a \} \subset T^\ast X_0 \times J$
is a closed subset of the relative cotangent bundle 
$T^\ast (X/(-\varepsilon, \varepsilon))\simto T^\ast X_0 \times (-\varepsilon, \varepsilon)$. 
As a consequence, $\Xi_a \subset T^\ast X_a \simeq T^\ast X_0$ 
approaches to $\Xi_0 \subset T^\ast X_0$ locally uniformly 
(in the projective sense) as $a\longrightarrow +0$. 
To express this situation we write $\Xi_a \rightsquigarrow \Xi_0 \ (a \longrightarrow +0)$.
Moreover we have following result, whose proof 
will be given in Appendix \ref{app-proof}. 
\begin{lemma}\label{ext-lem} 
The closed subset 
$\Xi \subset T^\ast (X/(-\varepsilon, \varepsilon))$ 
of $T^\ast (X/(-\varepsilon, \varepsilon))\simto 
T^\ast X_0 \times (-\varepsilon, \varepsilon)$
is subanalytic. 
\end{lemma} 
We thus have constructed a family $\{ \Xi_a \}_{a\in J} $
of closed conic subanalytic Lagrangian subsets $\Xi_a \subset T^\ast X_a 
\simto T^\ast X_0$ satisfying the required conditions and 
hence can consider the limit $\lim_{a\to +0} \C \C(F_a) $ 
in the sense of Section \ref{subsec-limit}.
\end{proof}

\begin{theorem}\label{thm-limformula}
    In the situation as above, we have 
    \begin{equation}\label{eq-limformula}
        \C \C (\psi_f^{\RR} (F)) = \lim_{a\to +0} \C \C (F_a)
    \end{equation}
    in $T^\ast X_0$.
\end{theorem}

\begin{proof}
Our proof is similar to that of Schmid-Vilonen \cite[Theorem 4.2]{SV96}. 
As in the proof of Lemma \ref{ex-lem}, we may assume that 
$X=X_0\times(-\varepsilon,\varepsilon)$. Then we 
take $\delta>0$ such that $0<\delta<\varepsilon$ and set 
 $G\coloneq( \rsect_{\{f>0 \}}(F) )_{X_0\times[-\delta,\delta]}\in \BDrc(X)$ 
so that we have $\psi_f^{\RR} (F) \simeq G|_{X_0}$. 
Let $\cS$ be a subanalytic Whitney stratification of 
$\supp G \subset X=X_0\times(-\varepsilon,\varepsilon)$ satisfying the $a_f$-condition 
such that
\begin{equation}
    \msupp(G)\subset \Lambda\coloneq \bigsqcup_{S\in\cS}T_S^\ast X
\end{equation}
and $\supp G \cap (X_0\times \{0\}) \simeq \supp G \cap X_0$ 
is a union of some strata in $\cS$. 
Then we set $\cS_0\coloneq\Set*{S\in\cS}{S\subset X_0\times\{0\}}\subset \cS$ and define
$\Lambda_a\subset T^\ast X_a$ ($0 \leq a \ll\delta$) by 
\begin{equation}
    \Lambda_a \coloneq
    \begin{cases}
        \displaystyle \bigsqcup_{S\in \mathcal{S}} T^\ast _{S\cap X_a} 
(X_a) \quad ( 0<a\ll\delta ), \\[20pt]
        \displaystyle \bigsqcup_{S\in \mathcal{S}_0} T^\ast _S X_0 \quad (a=0)
    \end{cases}
\end{equation}
so that we have $\msupp(F_a) \subset \Lambda_a$ for $0<a\ll\delta$. 
By taking a suitable subdivision of the Whitney stratification $\cS$, 
we may assume that for the subset $\cS_0\subset\cS$ we have 
\begin{equation}
    \msupp(\psi_f^\RR(F)) =  \msupp ( G|_{X_0} )
\subset \Lambda_0 =\bigsqcup_{S\in\cS_0}T_S^\ast X_0.
\end{equation}
Let us fix  $0<\delta_1\ll\delta$. Then, 
as in the proof of Lemma \ref{ext-lem}, by the $a_f$-condition we 
can show that $\bigsqcup_{0\leq a\leq\delta_1} 
\Lambda_a\times{\{a\}} \subset T^\ast(X/(-\varepsilon, \varepsilon))$ 
is a closed subanalytic subset of $T^\ast(X/(-\varepsilon, \varepsilon))$. 
Let $\SL$ be its subanalytic Whitney stratification 
consisting of connected strata,  
such that for any $S\in\cS_0$ the conormal bundle 
$T_S^\ast X_0 \subset \Lambda_0$ is a union of some strata in it 
and decompose $\Lambda_0$ into 
some of its strata as follows: 
\begin{equation}
\Lambda_0 =\bigsqcup_{\alpha\in A} \Theta_\alpha. 
\end{equation}
Then for any $\alpha\in A$ there exists a (unique) stratum $S\in\cS_0$ in $\cS_0$
such that $\Theta_\alpha\subset T_S^\ast X_0$ and we denote it by $S_\alpha$. 
For $\alpha\in A$ such that $\dim \Theta_\alpha=\dim X_0=\dim T_{S_\alpha}^\ast X_0$, 
let $[\Theta_\alpha]$ be the Lagrangian cycle obtained by restricting 
the one $[T_{S_\alpha}^\ast X_0]$ 
(for the definition, see \cite[Example 9.3.4]{KS90}) 
to an open neighborhood of $\Theta_\alpha$ in $T^\ast X_0$. 
Then it suffices to show that the multiplicities of $[\Theta_\alpha]$ in the both
sides of (\ref{eq-limformula}) coincide each other for such $\alpha\in A$. 
Fix such $\alpha\in A$ and choose a reference point 
$p_\alpha\in\Theta_\alpha\subset T_{S_\alpha}^\ast X_0$ of $\Theta_\alpha$. 
The case $\dim S_\alpha=\dim X_0$ being trivial, here we consider 
only the case $\dim S_\alpha<\dim X_0$ and take the point $p_\alpha$ from 
$T_{S_\alpha}^\ast X_0\setminus T_{X_0}^\ast X_0$.
Let $\pi_{X_0}\colon T^\ast X_0\longrightarrow X_0$ be the projection and set 
$q_\alpha\coloneq \pi_{X_0}(p_\alpha)\in S_\alpha\subset X_0$.
Next take a real analytic function $\varphi$ defined on a neighborhood $W$ of 
$q_\alpha$ in $X_0$ satisfying the following conditions:
\begin{equation}\label{eq-conditions}
\left\{
\begin{aligned}
    \textrm{(i)}\ &\varphi(q_\alpha)=0, \quad  
    \rmd\varphi(q_\alpha)=p_\alpha \in \Theta_\alpha \subset T_{S_\alpha}^\ast X_0, \\
    \textrm{(ii)}\ &\varphi\vbar_{S_\alpha}\colon S_\alpha\longrightarrow \RR
    \textrm{ has a Morse (non-degenerate) critical point at } q_\alpha\in S_\alpha \\ 
    & \textrm{ with Morse index 0 i.e. the Hesse matrix of }
    \varphi\vbar_{S_\alpha} \textrm{ at } q_\alpha\in S_\alpha 
    \textrm{ is positive definite.}
\end{aligned}
\right.
\end{equation}
Then by \cite[Lemma 7.2.2]{KS85} we see that the (not necessarily conic)
Lagrangian submanifold
\begin{equation}
    \Lambda^\varphi\coloneq \rmd\varphi(W)
    =\Set*{\rmd\varphi(x)}{x\in W} \quad \subset T^\ast X_0
\end{equation}
of $T^\ast X_0$ intersects $\Theta_\alpha\subset T^\ast X_0$ transversally.
Since $\Lambda_a \rightsquigarrow \Lambda_0$ in $T^\ast X_a$ 
as $a\longrightarrow+0$ (see the proof of Lemma \ref{ext-lem}) 
and any statum $L$ of 
the Whitney stratification $\SL$ 
such that $L \subset T^*(X_I/I) \subset 
T^\ast(X/(-\varepsilon, \varepsilon))$ and $\Theta_{\alpha} \subset 
\overline{L}$ is of dimension $\dim X_0 +1$, we see that 
$\Lambda^{\varphi}$ intersects $\Lambda_a\subset T^\ast X_a 
\simeq T^*X_0$ transversally for
all sufficiently small $0<a\ll 1$. 
Let $\pt$ be the one point space.
Then by \cite[Proposition 6.6.1(ii)]{KS90} there exists 
$M^\alpha\in \BDrc(\pt)$ such that for the object 
$M_{S_\alpha}^\alpha\coloneq k_{S_\alpha}\otimes_{k_{X_0}} 
M_{X_0}^\alpha\in \BDrc(X_0)$ of $\BDrc(X_0)$ we have an isomorphism 
$\psi_f^\RR(G)\simeq M_{S_\alpha}^\alpha$ in the localized category
$\BDC(X_0;p_\alpha)$ at the point $p_\alpha\in T^\ast X_0$ 
(for the definition, see \cite[Definition 6.1.1]{KS90}).
This implies that the multiplicity $m_\alpha$ of $[\Theta_\alpha]$ in 
$\CCyc(\psi_f^\RR(G))$ is equal to 
\begin{equation}
    \chi(M^\alpha) \coloneq \sum_{j\in\ZZ}(-1)^j \dim_k H^j(M^\alpha) \quad \in\ZZ 
\end{equation}
(see \cite[(9.4.8) and Proposition 9.4.5]{KS90}). 
It follows also from the isomorphism $\psi_f^\RR(F)\simeq M_{S_\alpha}^\alpha$ 
in $\BDC(X_0;p_\alpha)$ that we have isomorphisms 
\begin{equation}
    \rsect_{\{\varphi\geq0\}} (\psi_f^\RR(F))_{q_\alpha}
    \simeq \rsect_{\{\varphi\geq0\}} (M_{S_\alpha}^\alpha)_{q_\alpha}
    \simeq M^\alpha,
\end{equation}
where we used the fact that the stratum $S_\alpha\subset X_0$ is contained in 
$\{\varphi\geq0\} \subset X_0$ on a neighborhood of $q_\alpha\in S_\alpha$ in $X_0$.
Indeed, we can take a real analytic coordinate $x=(x_1,\dots,x_n)$ $(n
\coloneq \dim X_0)$ 
of $X_0$ on a neighborhood of $q_\alpha\in S_\alpha\subset X_0$ such that 
$q_\alpha=\{x=0\}$, $S_\alpha=\{x_1=\dots=x_d=0\}$ $(d\coloneq\codim_{X_0} S_\alpha)$,
$p_\alpha=(0,dx_1)\in T_{S_\alpha}^\ast X_0$ and
\begin{equation}
    \varphi(x)=x_1+(x_{d+1}^2+\cdots+x_n^2).
\end{equation}
By this coordinate, we define a real analytic function $\rho$ on 
a neighborhood of $q_\alpha\in S_{\alpha} \subset X_0$ by
\begin{equation}
    \rho(x)\coloneq x_1^2+x_2^2+\cdots+ x_n^2
\end{equation}
and for $r>0$ set $B(r)\coloneq \{\rho<r\}\subset X_0$.
Then, by applying the microlocal Bertini-Sard theorem \cite[Proposition 8.3.12]{KS90}
to the microsupport of the $\RR$-constructible sheaf 
$\rsect_{\{\varphi\geq0\}}(\psi_f^\RR(F))$ and $\rho$ and using the non-characteristic
deformation lemma \cite[Proposition 2.7.2]{KS90}, we see that there exists $0<r_0\ll1$
such that we have an isomorphism 
\begin{equation}
    \rsect\bigl(B(r); \rsect_{\{\varphi\geq0\}}(\psi_f^\RR(F)) \bigr)
    \simto\rsect_{\{\varphi\geq0\}}(\psi_f^\RR(F))_{q_\alpha} 
\end{equation}
for any $0< r\leq r_0$.
We shall modify the left hand side as that of 
\eqref{ex-eqan} below. 
For this purpose, 
we define a real analytic function $\tl{\rho}$ on 
$X=X_0\times (-\varepsilon,\varepsilon)$ by 
\begin{equation}
    \tl{\rho}(x,t)\coloneq \rho (x) \quad 
    \bigl((x,t)\in X= X_0\times(-\varepsilon,\varepsilon) \bigr).
\end{equation}
Then, by applying the microlocal Bertini-Sard theorem \cite[Proposition 8.3.12]{KS90}
to the closed conic subanalytic isotropic subset 
\begin{equation}
    \Lambda = \bigsqcup_{S\in\cS}T_S^\ast X \quad \subset T^\ast X
\end{equation}
and $\tl{\rho}\colon X\longrightarrow\RR$, we see that there exists 
$0<r_1\ll1$ such that for any $0<r\leq r_1$ the smooth hypersurface 
$\tl{\rho}^{-1}(r)\subset X$ intersects all the strata $S\in\cS$ in $\cS$ transversally.
Here we used the fact that $\tl{\rho}\colon X = X_0\times(-\varepsilon,\varepsilon)
\longrightarrow\RR$ is 
proper on $\pi_X(\Lambda) \subset X_0\times [-\delta,\delta]
\subset X=X_0\times (-\varepsilon,\varepsilon)$.
We thus obtain a Whitney stratification 
\begin{equation}
    \cS^\prime \coloneq
    \Set[\big]{S\cap \{\tl{\rho}=r_1\}}{S\in\cS} \cup 
    \Set[\big]{S\cap \{\tl{\rho}<r_1\}}{S\in\cS}
\end{equation}
of $\supp G \cap ( \overline{B(r_1)}\times[-\delta,\delta])$ and set 
$\cS_0^\prime\coloneq\Set*{S\in\cS^\prime}{S\subset \overline{B(r_1)}\times\{0\}}$.
Then we can easily check that $\mathcal{S}^\prime$ satisfies the $a_f$-condition and 
the $\RR$-constructible sheaf 
$G^\prime\coloneq G_{\{ \tl{\rho} \leq r_1 \}} 
\simeq 
G_{\overline{B(r_1)}\times (-\varepsilon,\varepsilon)}
\in \BDrc(X)$
on $X=X_0\times (-\varepsilon,\varepsilon)$ satisfies the condition
\begin{equation}
    \msupp(G^\prime) \subset \Lambda^\prime
    \coloneq\bigsqcup_{S\in\cS^\prime}T_S^\ast X.
\end{equation}
For $0\leq a\ll\delta$ we set 
\begin{equation}
    \Lambda_a^\prime\coloneq
    \left\{
    \begin{aligned}
        \displaystyle& \bigsqcup_{S\in\cS^\prime} T_{S\cap X_a}^\ast (X_a) 
        \quad (a>0), \\ 
        \\
        \displaystyle& \bigsqcup_{S\in\cS_0^\prime} T_S^\ast (X_0) 
        \quad (a=0)
    \end{aligned}
    \right.
\end{equation}
so that we have $\Lambda_a^\prime \rightsquigarrow \Lambda_0^\prime$ in 
$T^\ast X_a\simeq T^\ast X_0$ as $a\longrightarrow+0$.
Moreover, we can check also the condition 
\begin{equation}
\msupp(\psi_f^\RR(F)_{\overline{B(r_1)}})
= \msupp ( G^{\prime} |_{X_0}) 
\quad \subset \Lambda_0^\prime=\bigsqcup_{S\in\cS_0^\prime} T_S^\ast X_0.
\end{equation}
Then by applying the microlocal Bertini-Sard theorem \cite[Proposition 8.3.12]{KS90}
to the closed conic subanalytic isotropic subset 
\begin{equation}
    \Lambda_0^\prime=\bigsqcup_{S\in\cS_0^\prime} T_S^\ast X_0 
    \quad \subset T^\ast X_0
\end{equation}
and the real analytic function $\varphi\colon W\longrightarrow\RR$ 
(here we assume that $\overline{B(r_1)}\subset W$), we see that there exists 
$0<\varepsilon_0\ll1$ such that for any $s\in\RR$ satisfying the condition 
$0<\abs{s}\leq\varepsilon_0$ the smooth hypersurface $\varphi^{-1}(s)\subset W$
intersects all the strata $S\in\cS_0^\prime$ in $\cS_0^\prime$ transversally.
Moreover, the smooth hypersurface $\varphi^{-1}(0)\subset W$ intersects any stratum
$S\in \cS_0^\prime$ in $\cS_0^\prime$ other than the one 
$S_\alpha^\prime\coloneq S_\alpha\cap \{\rho <r_1\}\in \cS_0^\prime$
transversally (as in Schmid-Vilonen \cite[(2.6)]{SV96}, here  
we used \cite[Lemma 3.5.1]{GM88}). 
This implies that we have
\begin{equation}
    \rmd\varphi\bigl(\{\abs{\varphi}\leq\varepsilon_0\}\bigr) \cap
    \msupp(\psi_f^\RR(F)_{\overline{B(r_1)}})
    \subset
    \rmd\varphi\bigl(\{\abs{\varphi}\leq\varepsilon_0\}\bigr) \cap
    \Lambda_0^\prime \subset \{p_\alpha\}.
\end{equation}
Then by applying the microlocal Morse theory of Kashiwara-Schapira in \cite{KS90}
to the Morse function $\varphi\colon W\longrightarrow\RR$ and 
the $\RR$-constructible sheaf 
$\psi_f^\RR(F)_{\overline{B(r_1)}}\in\BDrc(X_0)$ we obtain isomorphisms
\begin{equation}\label{ex-eqan} 
    \rsect\bigl(X_0; \psi_f^\RR(F)_{\overline{B(r_1)}\cap \{
-\varepsilon_0<\varphi\leq\varepsilon_0\}} \bigr)
    \simeq \rsect_{\{\varphi\geq0\}}\bigl( \psi_f^\RR(F) \bigr)_{q_\alpha}
    \simeq M^\alpha
\end{equation}
(see Figure 1.). 
\begin{figure}[t]
\centering
\begin{tikzpicture}
\begin{scope}[xshift=-7cm,scale=0.9]
    \coordinate(O) at (0,0);
    \coordinate(A) at (-2,3);
    \coordinate(B) at (-2,-3);
    \coordinate(X1) at (-1,2);
    \coordinate(X2) at (-0.5,-1);
    \coordinate(C) at (0.2,2.8);
    \coordinate(D) at (0.5,-2.8);
    \coordinate(Y1) at (1.5,1.5);
    \coordinate(Y2) at (1.5,-1);
    \node[below] at (B) {$\varphi^{-1}(-\varepsilon_0)$};
    \node[below] at (D) {$\varphi^{-1}(\varepsilon_0)$};
    
    \coordinate (L) at (3,2);
    \coordinate (P) at (1.6,1.25);
    \draw[->] (L) to[bend right=20] (P);
    \node[fill=white, inner sep=3pt] at (L) {$\overline{B(r_1)}$};
    
    \node[below,font=\footnotesize,align=left] at (1,-3.9) {Figure 1.  The 
intersection $\overline{B(r_1)}\cap\{-\varepsilon_0<\varphi\leq\varepsilon_0\}$ \\
    is shown in gray.};
    
    \clip (3,3)--(3,-3)--(-3,-3)--(-3,3)--(3,3);
    \begin{scope}
    \clip (O) circle[radius=2];
    \clip (A) .. controls (X1) and (X2) .. (B) -- (3,-3) -- (3,3) -- cycle;
    \clip (C) .. controls (Y1) and (Y2) .. (D) -- (-3,-3) -- (-3,3) -- cycle;
    \fill[black!15] (-3,-3) rectangle (3,3);
    \end{scope}

    \draw (O) circle[radius=2];
    \draw[densely dashed] (A) .. controls (X1) and (X2) .. (B);
    \draw (C) .. controls (Y1) and (Y2) .. (D);
    \fill[black] (O) circle[radius=0.06];
    \node[below] at (O) {$0$};
\end{scope}

\begin{scope}[xshift=1cm,yshift=-0.5cm,scale=0.75]
    \node[above] at (-3,3) {$T^\ast X_0$};
    \node[left] at (-3,1.7) {$\Lambda_0$};
    \node[left] at (-3,0.3) {$\Lambda_0$};
    \node[below] at (1.8,-3) {$\Lambda_0$};
    \node[right] at (3,-1.8) {$\Lambda_0$};

    \node[below,font=\footnotesize,align=left] at (1.5,-4) {Figure 2.  The subsets 
$\mathrm{d}\varphi(\{\abs{\varphi}\leq\varepsilon_0\}\cap \overline{B(r_1)})$ 
\\ and $\Theta_\alpha$ only intersect at $p_\alpha$ in a neighborhood \\ of $p_\alpha$.};
    
    \begin{scope}
    \clip (3,3)--(3,-3)--(-3,-3)--(-3,3)--(3,3);
    \coordinate(O) at (0,0);
    \coordinate(A) at (-1,1);
    \coordinate(B) at (1,-1);
    \coordinate(Fu) at (-2,5);
    \coordinate(Fl) at (-5,2.2);
    \coordinate(Fld) at (-5,-0.5);
    \coordinate(Fd) at (2,-5);
    \coordinate(Fr) at (5,-2.2);
    
    \coordinate(C) at (-1.5,-4);
    \coordinate(D) at (3.5,2);
    \coordinate(Cp1) at (-1.3,-0.4);
    \coordinate(Cp2) at (1.2,1.7);
    
    \draw[thick] (A) -- (B);
    \draw[semithick] (A) -- (Fu);
    \draw[semithick] (A) -- (Fl);
    \draw[semithick] (A) -- (Fld);
    \draw[semithick] (B) -- (Fr);
    \draw[semithick] (B) -- (Fd);
    \draw[semithick] (C) .. controls (Cp1) and (Cp2) .. (D);

    \draw[decorate,thick,decoration={name=zigzag, amplitude=0.35mm, segment length=1.2mm}] (A) -- (B);
    \fill[white] (C) circle[radius=1.7];
    \fill[white] (D) circle[radius=1.5];
    \fill[white] (O) circle[radius=0.12];
    \fill[black] (O) circle[radius=0.08];
    \fill[white] (A) circle[radius=0.1];
    \fill[white] (B) circle[radius=0.1];
    \draw (O) circle[radius=0.12];
    \draw (A) circle[radius=0.1];
    \draw (B) circle[radius=0.1];

    \coordinate(Ap1) at (-5,-1);
    \coordinate(Xp1) at (-1.25,0.75);
    \coordinate(Ap2) at (0,-0.45);
    \coordinate(Xp2) at (0.5,-1);
    \coordinate(Xp3) at (0.8,-1);
    \coordinate(Ap3) at (1.5,-5);
    \draw[densely dashed] (Ap1).. controls (Xp1) .. (Ap2);
    \draw[densely dashed,dash phase=3pt] (Ap2) .. controls (Xp2) and (Xp3) .. (Ap3);

    \coordinate(Bp1) at (-1,5);
    \coordinate(Yp1) at (-0.9,2);
    \coordinate(Yp2) at (-0.75,1);
    \coordinate(Bp2) at (0.5,-0.05);
    \coordinate(Yp3) at (2,-1.2);
    \coordinate(Yp4) at (3,-1);
    \coordinate(Bp3) at (5,-1.6);
    \draw[densely dashed] (Bp1).. controls (Yp1) and (Yp2) .. (Bp2);
    \draw[densely dashed,dash phase=2pt] (Bp2).. controls (Yp3) and (Yp4) .. (Bp3);

    \coordinate(Cp1) at (-0.5,5);
    \coordinate(Zp1) at (-0.3,2);
    \coordinate(Zp2) at (-0.2,1);
    \coordinate(Cp2) at (1,0.15);
    \coordinate(Zp3) at (2,-0.6);
    \coordinate(Zp4) at (3,-0.5);
    \coordinate(Cp3) at (5,-1);
    \draw[densely dashed]  (Cp1).. controls (Zp1) and (Zp2) .. (Cp2);
    \draw[densely dashed,dash phase=3pt]  (Cp2).. controls (Zp3) and (Zp4) .. (Cp3);
    \end{scope}
    
    \coordinate(S1) at (3.5,3);
    \coordinate(S2) at (2.1,1.7);
    \draw[->] (S1) to [bend right=30] (S2);
    \node[fill=white, inner sep=3pt] at (S1) {$\mathrm{d}
\varphi(\{\abs{\varphi}\leq\varepsilon_0\}\cap \overline{B(r_1)})$};

    \coordinate(S3) at (-2,-1);
    \coordinate(S4) at (-0.55,0.4);
    \draw[->] (S3) to[bend left=20](S4);
    \node[fill=white, inner sep=3pt] at (S3) {$\Theta_\alpha$};

    \coordinate(S5) at (-0.3,-2);
    \coordinate(S6) at (0,-0.15);
    \draw[->] (S5) to[bend right=10] (S6);
    \node[fill=white, inner sep=3pt] at (S5) {$p_\alpha$};

    \coordinate(T1) at (2.4,-0.5);
    \coordinate(T2) at (1.9,-0.9);
    \coordinate(T3) at (0.9,-2);
    \coordinate(T4) at (3,-2.7);
    \coordinate(T5) at (5,-3);
    \draw[->] (T4) to[bend left=30] (T1);
    \draw[->] (T4) to[bend left=25] (T2);
    \draw[->] (T4) to[bend right=20] (T3);
    \node[fill=white, inner sep=3pt] at (T5) {$\Lambda_a$ ($0<a\ll\delta_0$)};
\end{scope}
\end{tikzpicture}
\end{figure}
We define a real analytic function $\tl{\varphi}$ on 
$X=X_0\times (-\varepsilon,\varepsilon)$ by 
\begin{equation}
    \tl{\varphi}(x,t)\coloneq \varphi (x) \quad 
    \bigl((x,t)\in X= X_0\times(-\varepsilon,\varepsilon) \bigr).
\end{equation}
and set 
\begin{equation}
    \SG\coloneq (G^\prime)_{\{-\varepsilon_0<\tl{\varphi}\leq\varepsilon_0\}} 
    \quad \in \BDrc(X)
\end{equation}
so that for the inclusion map 
$i=i_{X}: X_0\times \{0\} ( \simeq X_0) \longhookrightarrow 
X=X_0\times (-\varepsilon,\varepsilon)$ 
we have an isomorphism
\begin{equation}
    \rsect\bigl( X_0; \psi_f^\RR(F)_{\overline{B(r_1)}
\cap \{-\varepsilon_0<\varphi\leq\varepsilon_0\}} \bigr)
    \simeq \rsect(X_0; i^{-1} \SG).
\end{equation}
Then, by applying the microlocal Bertini-Sard theorem \cite[Proposition 8.3.12]{KS90}
to $\msupp(\SG)\sqcup\msupp(\SG)^a\subset T^\ast X$
and the real analytic function 
$t\colon X\times(-\varepsilon,\varepsilon)\longrightarrow(-\varepsilon,\varepsilon)$
and using the non-characteristic deformation lemma \cite[Proposition 2.7.2]{KS90},
we see that there exists $0<\delta_0\ll\delta$ such that for any $0<a\leq\delta_0$
we have isomorphisms
\begin{equation}
    \rsect(X_0; i^{-1} \SG) \simot 
    \rsect(X_0\times [0,a] ; \SG) \simto 
    \rsect(X_a; \SG\vbar_{X_a}).
\end{equation}
For any $0<a\leq\delta_0$ we thus obtain an isomorphism
\begin{equation}
    \rsect(X_a; (F_a)_{\overline{B(r_1)}\cap \{-\varepsilon_0<\varphi\leq\varepsilon_0\}})
    \simeq M^\alpha.
\end{equation}
Now recall that for $0<a\ll\delta_0$ we have 
\begin{equation}
    \msupp((F_a)_{\overline{B(r_1)}})=
    \msupp(G^\prime\vbar_{X_a}) \quad \subset
    \Lambda_a^\prime=\bigsqcup_{S\in\cS^\prime} T_{S\cap X_a}^\ast (X_a)
\end{equation}
and $\Lambda_a^\prime \rightsquigarrow\Lambda_0^\prime$ as $a\longrightarrow+0$.
Then for $0<a\ll\delta_0$ the compact subset 
$\rmd\varphi\bigl( \{\abs{\varphi}\leq\varepsilon_0\} \cap \overline{B(r_1)} \bigr) \subset T^\ast X_0$
intersects $\Lambda_a^\prime$ only on a neighborhood of the point 
$p_\alpha\in \Theta_\alpha\subset T^\ast X_0$ in $T^\ast X_0$ and 
the intersection is transversal (see Figure 2.). 
Then by Proposition \ref{prop-intersection} and Example \ref{ex-Lagintersect} 
as well as \cite[Theorems 9.5.3 and 9.5.6]{KS90},
we conclude that the 
multiplicity of $[\Theta_\alpha]$ in the limit 
$\displaystyle\lim_{a\to+0}\CCyc(F_a)$ is equal to
\begin{equation}
    \chi\bigl(\rsect(X_a; (F_a)_{\overline{B(r_1)}\cap 
\{-\varepsilon_0<\varphi\leq\varepsilon_0\}}) \bigr) =
    \chi\bigl( \rsect_{\{\varphi\geq0\}}(\psi_f^\RR (F) )_{q_\alpha} \bigr)
    =\chi(M^\alpha)=m_\alpha
\end{equation}
for $0<a\ll\delta_0$. This completes the proof.
\end{proof}

\begin{remark}
    In the proof of Theorem \ref{thm-limformula} we used subanalytic Whitney stratifications adapted to $f^{-1}(0)$ and satisfying the $a_f$-condition of Thom which is essentially stronger than the usual one $a$ of Whitney.
    In \cite{Bek93} and \cite{KP94}, Bekka, Kurdyka and Parusi\'{n}ski showed that for any subanalytic Whitney stratification adapted to $f^{-1}(0)$ we can stratify $f^{-1}(0)$ ti satisfy the $a_f$-condition.
    In this real subanalytic case, there are also some examples of Whitney stratifications adapted to $f^{-1}(0)$ but not satisfying the $a_f$-condition (see e.g. \cite[page 100]{Par93}).
    But for complex analytic stratifications and holomorphic functions, we do not have such a problem (see \cite[Th\'{e}or\`{e}me 4.2.1]{BMM94}).
    This shows the difficulty in using real subanalytic stratifications.
    For example, in \cite[page 469]{SV96} Schmid and Vilonen should have used a subanalytic Whitney stratification satisfying also the Thom condition.
    Their argument using only the condition $a$ of Whitney is not correct.
    However, for real constructible sheaves adapted to a subanalytic stratification, by taking a suitable subdivision of the Whitney stratification in \cite[page 469]{SV96} to have the Thom condition we can prove their claim there.
    For more general real stratifications adapted to an o-minimal structure, we have also some partial remedies to this problem. 
    See for example \cite{Loi98}.
\end{remark}

By the above arguments and 
Theorem \ref{thm-limformula}, we thus obtain Theorem \ref{thm-limit}. 
As is clear from the proof of Theorem \ref{thm-limformula}, 
via the microlocal Morse theory of Kashiwara-Schapira \cite{KS90} for 
the Morse function $\varphi$, our formula 
\begin{equation}
    \rsect(X_a; (F_a)_{\overline{B(r_1)}\cap \{-\varepsilon_0<\varphi\leq\varepsilon_0\}})
    \simeq M^\alpha \qquad ( 0<a\ll\delta_0 )
\end{equation}
allows us to calculate not only the characteristic cycle but also the 
``microlocal types'' $M^\alpha\in\BDrc(\pt)$ ($\alpha \in A$) 
of $\psi_f^\RR(F)\in \BDrc(X_0)$
in many situations.
Let us give some examples.

\begin{example}\label{ex-realnearby}
Let $X=\RR^n$ be the $n$-dimensional real vector space with  linear coordinates $(x_1,\dots,x_n)$.
For a real analytic function $f\colon X=\RR^n\longrightarrow\RR$
set $Z\coloneq f^{-1}(0)\subset X$ and let $\iota\colon Z\longhookrightarrow X=\RR^n$ be the inclusion map.
Let $i_f\colon X\longhookrightarrow X\times\RR$ be the graph embedding of $f$ 
and $t\colon X\times\RR\longrightarrow \RR$ the projection.
We set $F\coloneq i_{f\ast}\CC_X\in\BDrc(X\times\RR)$.
Then we have $\iota_\ast \psi^\RR_f(\CC_X)\simeq \psi_t^\RR(F)\in \BDrc(X)$.
\begin{enumerate}
\item [\rm (i)]
Let us consider the case of $f(x_1,\dots,x_n)\coloneq x_1^2+\cdots+x_n^2$.
In this situation, we have  $\msupp(\psi_t^\RR(F))\subset T_{\{0\}}^\ast X$.
Set $p\coloneq (0; \rmd x_1)\in T_{\{0\}}^\ast X$.
We shall calculate the ``microlocal type'' of the $\RR$-constructible sheaf $\iota_\ast 
\psi^\RR_f(\CC_X)\simeq \psi_t^\RR(F)\in\BDrc(X)$ at $p$ and its characteristic cycle.
In what follows, we use the same notation as in the proof of Theorem \ref{thm-limformula}.
Let us define a real analytic function 
$\varphi\colon X\longrightarrow\RR$ by $\varphi(x_1,\dots,x_n)\coloneq x_1$.
It is straightforward to check that $\varphi$ satisfies the conditions (\ref{eq-conditions}).
Let $r_1>0$ and $\varepsilon_0>0$ be sufficiently small real numbers as in the proof of Theorem \ref{thm-limformula}.
Then for a sufficiently small $a>0$ one has $f^{-1}(a)\subset \overline{B(r_1)}\cap 
\{-\varepsilon_0<\varphi\leq\varepsilon_0\}$ and hence
\begin{align}
    H^j\rsect(X_a; (F_a)_{\overline{B(r_1)}\cap \{-\varepsilon_0<\varphi\leq\varepsilon_0\}})
    &\simeq 
    H^j\rsect(X; \CC_{f^{-1}(a)}) \\
    &\simeq
    \begin{cases}
        \CC & (j=0,n-1), \\
        0 & (\mathit{otherwise}).
    \end{cases}
\end{align}
Therefore, we obtain $\psi_t^\RR(F)\simeq \CC_{\{0\}}\oplus(\CC_{\{0\}}[1-n])$ in the localized category $\BDC(X; p)$ and 
$\CCyc(\psi_t^\RR(F))=(1+(-1)^{n-1})\cdot [T_{\{0\}}^\ast X]$.
Note that if we are only interested in the characteristic cycle, we can compute it 
more directly from Theorem \ref{thm-limformula} as follows.
We use the same notation as above.
For $a>0$ set $Z_a\coloneq f^{-1}(a)\subset X$.
Then the restriction $\varphi|_{Z_a}\colon Z_a\longrightarrow \RR$
has two Morse critical points $q_\pm\coloneq(\pm\sqrt a,0,\dots,0)\in Z_a$.
We can check that the Morse indices of $\varphi|_{Z_a}$ at $q_+$ and $q_-$ are $n-1$ and $0$, respectively.
It follows from $\CCyc(F_a)=\CCyc(\CC_{Z_a})=[T_{Z_a}^\ast X]$ and \cite[Example 9.5.7]{KS90} that
\begin{equation}
     \#([\Lambda^\varphi]\cap \CCyc(F_a) ) 
    = (-1)^0 +(-1)^{n-1} =1+(-1)^{n-1},
\end{equation}
where $[\Lambda^\varphi]$ is the Lagrangian cycle associated to $\varphi$.
By Proposition \ref{prop-intersection} and Theorem \ref{thm-limformula}, one has
\begin{align}
    \#([\Lambda^\varphi]\cap \CCyc(\psi_t^\RR(F)))
    &=  \#\bigl( [\Lambda^\varphi]\cap (\lim_{a\to+0}\CCyc(F_a)) \bigr)\\
    &= \#([\Lambda^\varphi]\cap \CCyc(F_a) ) \\
    &= 1+(-1)^{n-1}.
\end{align}
This implies that the multiplicity of $[T_{\{0\}}^\ast X]$ in $\CCyc(\psi_t^\RR(F))$ is equal to $1+(-1)^{n-1}$.

\item [\rm (ii)]
Let us consider the case of $n=2$ and $f(x_1,x_2)=x_1^3-x_2^2$.
In this situation, we have 
$\msupp(\psi_t^\RR(F))\subset T_{\{0\}}^\ast X\cup \overline{T_{Z_\reg}^\ast X}$.
Set $p\coloneq (0;\rmd x_1)\in T_{\{0\}}^\ast X\setminus \overline{T_{Z_\reg}^\ast X}$.
We shall calculate the ``microlocal types'' of the $\RR$-constructible sheaf $\psi_t^\RR(F)\in\BDrc(X)$ at $p$ and $-p$.
We use the same notation as in the proof of Theorem \ref{thm-limformula}.
Let us define a real analytic function 
$\varphi\colon X\longrightarrow\RR$ by $\varphi(x_1,x_2)\coloneq x_1$.
Then $\varphi$ satisfies the conditions (\ref{eq-conditions}).
Let $r_1>0$ and $\varepsilon_0>0$ be sufficiently small real numbers as in the proof of Theorem \ref{thm-limformula}.
Then for a sufficiently small $a>0$ one has 
\begin{align}
    H^j\rsect(X_a; (F_a)_{\overline{B(r_1)}\cap \{-\varepsilon_0<\varphi\leq\varepsilon_0\}})
    \simeq 
    \begin{cases}
        \CC & (j=0), \\
        0 & (\mathit{otherwise}).
    \end{cases}
\end{align}
This implies $\psi_t^\RR(F)\simeq \CC_{\{0\}}$ in $\BDC(X; p)$.
Similarly, it is straightforward to check that $\psi_t^\RR(F)\simeq \CC_{\{0\}}[-1]$ in $\BDC(X; -p)$.
\end{enumerate}
\end{example}

Now let $X$ be a real analytic manifold and $M\subset X$ its real analytic submanifold.
Recall that for the normal deformation $\tl{X}_M$ of $X$ along $M$ there exists a 
natural morphism $t\colon\tl{X}_M\longrightarrow\RR$ such that 
$t^{-1}(0)\simeq T_M X$ and a commutative diagram
\begin{equation}
\begin{tikzcd}
    T_M X \ar[r,"s",hook] \ar[d,"\tau"] & 
    \tl{X}_M \ar[d,"p"] &
    \Omega=t^{-1}(\RR_{>0}) \ar[l,"j",hook'] \ar[ld,"\tl{p}"] \\
    M \ar[r,"i",hook] & X. & 
\end{tikzcd}
\end{equation}
Then by Theorem \ref{thm-limformula} for $F\in\BDrc(X)$ we obtain the following 
formula of the characteristic cycle of its specialization $\nu_M(F)\in\BDrc(T_M X)$
along $M\subset X$.

\begin{theorem}\label{thm-CCspe}
    Let $F\in \BDrc(X)$ be an $\RR$-constructible sheaf on $X$.
    Then we have
    \begin{equation}
        \CCyc(\nu_M(F))=\lim_{a\to+0}\CCyc(\tl{F_a}),
    \end{equation}
    where for $a>0$ we set 
    $\tl{F_a}\coloneq (\tl{p}^{-1}F)|_{t^{-1}(a)}\simeq (p^{-1}F)|_{t^{-1}(a)}\in\BDrc\bigl( t^{-1}(a) \bigr)$.
\end{theorem}

\begin{proof}
    Note that we have isomorphisms
    \begin{equation}
        \nu_M(F)\simeq s^{-1}\rmR j_{\ast} j^{-1}(p^{-1}F)\simeq \psi_t^{\RR}(p^{-1}F).
    \end{equation}
    Then the assertion immediately follows from Theorem \ref{thm-limformula}.
\end{proof}

With Theorem \ref{thm-CCspe} at hand, for $F\in\BDrc(X)$ we obtain also a formula
of $\CCyc(\mu_M(F))$ as follows.
Recall that by \cite[Proposition 5.5.1]{KS90} there exists a natural morphism
\begin{equation}
    \Phi_{T_MX}\colon T^\ast(T_MX)\simto T^\ast(T_M^\ast X).
\end{equation}
Let $x=(y,z)$ be a local coordinate of $X$ such that $M=\{y=0\}$ and $(z,\eta)$
(resp. $(z,\eta^\ast)$) the coordinate of $T_M X$ (resp. $T_M^\ast X$) associated to it.
Then the morphism $\Phi_{T_M X}$ is explicitly written as
\begin{equation}
    \Phi_{T_M X}\colon (z,\eta;z^\ast,\eta^\ast)\longmapsto (z,\eta^\ast;z^\ast,-\eta).
\end{equation}
By $\Phi_{T_M X}\colon T^\ast(T_M X)\simto T^\ast(T_M^\ast X)$ identify 
$T^\ast(T_M X)$ with $T^\ast(T_M^\ast X)$.
Moreover, we identify $\pi_{T_M X}^{-1}\omega_{T_M X}$ with $\pi_{T_M^\ast X}^{-1}\omega_{T_M^\ast X}$ naturally.
Then by \cite[Theorem 5.5.5 and Exercise IX.7]{KS90} for any ``conic''
$\RR$-constructible sheaf $G\in\BDrc(T_M X)$ on $T_M X$ we have
\begin{equation}
    \msupp(G^\wedge)=\msupp(G), \quad
    \CCyc(G^\wedge)=\CCyc(G).
\end{equation}
Applying this result to $\nu_M(F)\in\BDrc(T_M X)$ in Theorem \ref{thm-CCspe},
we obtain a formula of $\CCyc(\mu_M(F))$.

\section{Characteristic cycles of complex nearby and vanishing cycle sheaves}\label{sec-cc}
In this section, we apply our results in Section \ref{sec-realnearby} to 
obtain formulas for the characteristic cycles of complex 
nearby and vanishing cycles.
Let $X$ be a complex manifold and 
$f\colon X\longrightarrow \CC$ a non-constant holomorphic function
on it.
Set $Y\coloneq f^{-1}(0) \subset X$ and let $i\colon Y \longhookrightarrow X$
and $j\colon X\setminus Y \longhookrightarrow X$ 
be the inclusion maps.
Let $\exp \colon \CC \longrightarrow \CC^\ast =\CC \setminus \{0\} \quad 
(\tau \longmapsto \exp (\tau))$
be the universal covering of $\CC^\ast$ and define a complex 
manifold $\widetilde{X\setminus Y}$ by the Cartesian diagram 
\begin{equation}\label{eq-uni-cov}
    \xymatrix{
        \widetilde{X\setminus Y} \ar@{}[rd]|{\Box} \ar[r]^{\widetilde{f}} \ar[d]_p & \CC \ar[d]^{\exp }\\
        X\setminus Y \ar[r]_-{f|_{X\setminus Y}} & \CC^\ast .
    }
\end{equation}
Then we obtain a commutative diagram
\begin{equation}\label{eq-near}
    \xymatrix{
    {} & {} & \widetilde{X\setminus Y} \ar[d]_p \ar[ld]_{\widetilde{p}} \\
    Y \ar@{^{(}->}[r]_i & X & X\setminus Y \ar@{_{(}->}[l]^j 
    }
\end{equation}
and the following definition of complex nearby cycle sheaves 
due to Deligne \cite{Del73}.

\begin{definition}
    For $F \in \mathbf{D}^{\mathrm{b}}(X)$ we set 
    \begin{equation}
        \psi_f (F) \coloneq i^{-1} \mathrm{R} \widetilde{p}_\ast \widetilde{p}^{-1} F 
        \simeq  i^{-1} \mathrm{R} j_\ast \mathrm{R} p_\ast p^{-1} j^{-1} F \quad \in \mathbf{D}^{\mathrm{b}}(Y)
    \end{equation}
    and call it the complex nearby cycle sheaf of $F$ along $f$.
\end{definition}
Let $\Dbc (X) \subset \Db (X)$ be the full subcategory of $\Db (X)$
consisting of objects whose cohomology sheaves are complex constructible.
Then we can show that the functor $\psi_f (\cdot) \colon \Db (X) \longrightarrow \Db (Y)$
preserves the constructibility and obtain a functor 

\begin{equation}
    \psi_f (\cdot) \colon \Dbc (X) \longrightarrow \Dbc(Y).
\end{equation}
For a point $x_0 \in Y$ of $Y=f^{-1}(0) \subset X$ let $M_{f,x_0}
\subset X\setminus Y$ be the Milnor fiber of $f$ at $x_0$.
Then we have the following basic result (see. e.g. \cite[Proposition 4.2.2]{Dim04},
\cite{Le77} and \cite[Theorem 2.6]{Tak25}).

\begin{lemma}\label{lemm-Miln} 
    Let $F \in \Dbc(X)$ be a complex constructible sheaf on $X$.
    Then for any point $x_0 \in Y$ of $Y=f^{-1}(0) \subset X$
    and any $j\in \mathbb{Z}$ there exists an isomorphism
    \begin{equation}
        H^j \psi_f (F)_{x_0} \simeq H^j (M_{f,x_0}; F).
    \end{equation}
\end{lemma}

Recall that by the construction of the functor $\psi_f (\cdot)$
for $F\in \Dbc(X)$ we obtain an automorphism $\Psi (F)\colon 
\psi_f(F) \simto \psi_f(F)$
that we call the monodromy automorphism of $\psi_f(F) \in \Dbc(Y)$.
From now, we will show that by forgetting the monodromy
automorphism $\Psi(F)$ we can simplify the construction of 
$\psi_f(F) \in \Dbc(Y)$ and define it by the real nearby cycle
functor introduced in Section \ref{sec-realnearby}.
For this purpose, assume that we have $\mathrm{d} f(x_0) \neq 0$
for any point $x_0 \in Y$ of $Y=f^{-1}(0) \subset X$ i.e. 
$f\colon X \longrightarrow \CC$ is smooth on a neighborhood of
$Y=f^{-1}(0) \subset X$ in $X$.
Shrinking $X$ if necessary, we may assume that $f\colon X \longrightarrow \CC$
itself is smooth.
In this situation, the complex hypersurface $Y=f^{-1}(0) \subset X$ 
is smooth and we define a smooth real analytic hypersurface $H\subset X$ 
of $X$ containing $Y\subset X$ by $H\coloneq \{ \mathrm{Im} f =0\} \subset X$.
Let $g\colon H \longrightarrow \RR$ be the restriction of $\mathrm{Re} \, f \colon X \longrightarrow \RR$
to $H=\{ \mathrm{Im} f=0\} \subset X$ so that we have $g^{-1}(0) =Y=f^{-1}(0)$.

\begin{lemma}\label{lem-nearby}
    Let $F\in \Dbc(X)$ be a complex constructible sheaf on $X$.
    Then there exists an isomorphim
    \begin{equation}
        \psi_f(F) \simto \psi_g^{\RR} (F|_H)
    \end{equation}
    in $\Db(Y)$.
\end{lemma}

\begin{proof}
    We shall use the notations in the commutative
    diagrams (\ref{eq-uni-cov}) and (\ref{eq-near}).
    Let $\RR \coloneq \{ \tau \in \CC \mid \mathrm{Im} \tau =0\} 
    \subset \CC$ be the real axis in $\CC$ and set $\ell \coloneq 
    \exp (\RR) \subset \CC^\ast$.
    Then the restriction $\RR \longrightarrow \ell$
    of $\exp \colon\CC \longrightarrow \CC^\ast$ to 
    $\RR\subset \CC$ is an isomorphism.
    Note that $L\coloneq \left( f|_{X\setminus Y}\right)^{-1} (\ell) \subset X\setminus Y$
    is nothing but the open half space $\{ g>0\} \subset H$
    of $H=\{ \mathrm{Im} f=0\} \subset X$
    and the restriction ${\widetilde{f}}^{-1}(\RR) \longrightarrow
    L =\left( f|_{X\setminus Y}\right)^{-1}(\ell)$
    of $p\colon \widetilde{X\setminus Y} \longrightarrow X\setminus Y$
    is an isomorphism.
    Let $\iota \colon \widetilde{f}^{-1} (\RR) \longhookrightarrow \widetilde{X\setminus Y}$
    and $j_{\RR} \colon L=\{ g>0\} \longhookrightarrow X$ be the inclusion maps.
    Then we obtain a chain of morphisms
    \begin{align}
        \psi_f(F) & = i^{-1}\mathrm{R}j_\ast \mathrm{R} p_\ast p^{-1} j^{-1} F\\
        & \longrightarrow i^{-1}\mathrm{R} j_\ast \mathrm{R} p_\ast 
\iota_\ast \iota^{-1} p^{-1} j^{-1} F \\
        & \simeq i^{-1} \mathrm{R} {j_{\RR}}_\ast j_{\RR}^{-1} F \simeq \psi_g^{\RR} (F|_H).
    \end{align}
    By the proof of \cite[Theorem 2.6]{Tak25} we can easily show that 
    it is an isomorphism.
\end{proof}
By Lemma \ref{lem-nearby}, as a special case of Theorem \ref{thm-limformula} we obtain 
a topological analogue of 
Ginsburg's theorem \cite[Theorem 5.5]{Gin86} as follows. 
As in Section \ref{sec-realnearby}, for $0\leq a \ll 1$ and $F\in \Dbc(X)$ 
we set $X_a \coloneq g^{-1}(a) =f^{-1}(a) \subset H \subset X$
and $F_a \coloneq F|_{X_a}$.
Note that for $a=0$ we have $X_0 =Y =f^{-1}(0)$.
Then we have the following result.

\begin{theorem}[{Ginsburg \cite[Theorem 5.5]{Gin86}}]\label{thm-Gins}
    Let $F\in \Dbc(X)$ be a complex constructible sheaf
    on $X$.
    Then we have 
    \begin{equation}
        \CCyc (\psi_f(F)) = \lim_{a \to +0} \CCyc (F_a)
    \end{equation}
    in $T^\ast X_0$, where this equality 
    holds true only over each relatively compact open subset 
    of $X_0$.
\end{theorem}

By the proof of Theorem \ref{thm-limformula}, for $F\in \Dbc(X)$ 
we obtain not only the characteristic cycle but also the 
microlocal types of $\psi_f(F)$ in many situations.
For example, let us show the well-known fact that 
for a perverse sheaf $F\in \Dbc(X)$ on $X$ the shifted 
nearby cycle sheaf $\psi_f(F)[-1] \in \Dbc(Y)$ 
is a perverse sheaf on $Y$ (here we adopt the convention 
that $\CC_X [ {\rm \dim}_\CC X] \in \Dbc (X)$ is perverse). 
Recall that this important result was first proved by 
Kashiwara in \cite{Kas83b} by constructing a regular 
holonomic $\mathcal{D}_Y$-module $\mathcal{N}$ on $Y$ 
which corresponds to $\psi_f(F)[-1] \in \Dbc(Y)$ 
via the Riemann-Hilbert correspondence. 
See also Goresky-MacPherson \cite{GM} for another 
approach to this problem. 
Later, some other proofs were given by 
Brylinski \cite[Theorem 1.2 and Corollary 1.7]{Bry86}, 
Kashiwara-Schapira \cite[Corollary 10.3.11]{KS90}
and Sch\"{u}rmann \cite[Chapter 6]{Sch03}.
Our proof below is microlocal, in the sense that it relies on 
the microlocal characterization of perversity proved by 
Kashiwara and Shcapira in \cite[Theorem 9.5.2]{KS85}.

\begin{theorem}\label{thm-perverse}
    Assume that $F\in \Dbc(X)$ is perverse.
    Then $\psi_f(F)[-1] \in \Db(Y)$ is also perverse.
\end{theorem}

\begin{proof}
By \cite[Proposition 4.2.11]{Dim04} and 
\cite[Exercise VI\hspace{-1.2pt}I\hspace{-1.2pt}I.15]{KS90}
and resolution of singularities, we can easily show that 
$\psi_f(F) \simeq \psi_g^{\RR} (F|_H) \in \Db(Y)$
is complex constructible i.e. $\psi_f(F) \in \Dbc(Y)$
(see also \cite[Proposition 8.6.3]{KS90}).
Then we have only to prove that 
it satisfies the microlocal characterization of perversity in 
\cite[Theorem 9.5.2]{KS85}.
As in the proof of Theorem \ref{thm-limformula}, after 
cutting the support of $F|_H$ to be compact in 
$H=X_0 \times ( - \varepsilon, \varepsilon )$ for $0< \delta < \varepsilon$ 
we set $G\coloneq \left( \rsect_{\{ g>0\}} (F|_H)\right)_{X_0\times 
[-\delta, \delta]} \in \Db_{\RR-\mathrm{c}}(H)$ 
and define a subanalytic Whitney stratification $\mathcal{S}$ of 
$\supp G \subset H$ satisfying the $a_g$-condition, closed conic subanalytic Lagrangian subsets $\Lambda \subset T^\ast H,
\Lambda_0 \subset T^\ast X_0$ and $\Lambda_a \subset T^\ast X_a 
\medspace (0< a \ll \delta), 
\Theta_{\alpha} \subset \Lambda_0 \medspace (\alpha \in A)$,
and $S_{\alpha} \in \mathcal{S} \medspace (\alpha \in A)$. 
For some technical reason, to construct the Whitney stratification 
$\mathcal{S}$, we first take a complex analytic one $\mathcal{T}$ of 
$X$ adapted to the original $F\in \Dbc(X)$ 
and $\psi_f(F) \in \Dbc(Y)$ and take the part of $\supp G \subset H 
\subset X$ of its subanalytic refinement. 
Then we define a subset $A^{\prime} \subset A$ of $A$ by 
\begin{equation}
A^{\prime} \coloneq \Set*{ \alpha \in A}{\mathrm{dim}_\RR\Theta_\alpha
= \mathrm{dim}_\RR X_0 \text{ and } S_{\alpha}  
\text{ is an open subset of some } T \in \mathcal{T}}.
\end{equation}
Then it suffices to show that for each $\alpha \in A^{\prime}$
there exists a finite dimensional $k$-vector space $V^{\alpha}$
such that for the object $V_{S_{\alpha}}^{\alpha} \coloneq k_{S_{\alpha}} 
\otimes_{k_{X_0}} V_{X_0}^{\alpha} \in \Db_{\RR- \mathrm{c}}(X_0)$
of $\Db_{\RR-\mathrm{c}}(X_0)$
we have isomorphisms
\begin{equation}
    \psi_f(F)[-1] \simeq \psi_g^{\RR}(F|_H)[-1] \simeq 
V_{S_{\alpha}}^{\alpha}[\mathrm{dim}_{\CC} S_{\alpha}]
\end{equation}
in the localized category $\Db(X_0; p)$ for generic points $p \in 
\Theta_{\alpha}$ of $\Theta_{\alpha}$.
Fix $\alpha \in A^{\prime}$ and choose a reference point $p_{\alpha} 
\in \Theta_{\alpha} \subset T_{S_{\alpha}}^\ast X_0$
of $\Theta_{\alpha}$.
For the projection $\pi_{X_0} \colon T^\ast X_0 \longrightarrow X_0$
we set $q_{\alpha} \coloneq \pi_{X_0} (p_{\alpha}) \in S_{\alpha}$.
Then we take a holomorphic function $\phi$ defined on a neighborhood $W$
of $q_{\alpha}$ in $X_0$ satisfying the following conditions:
\begin{equation}
    \begin{cases}
        (\textup{i})\: \phi (q_{\alpha}) =0, \quad \mathrm{d} \phi 
(q_{\alpha}) =p_{\alpha} \in \Theta_{\alpha} \subset T_{S_{\alpha}}^\ast X_0, \\
        (\textup{ii})\: \phi|_{S_{\alpha}} \colon S_{\alpha} \longrightarrow \CC \medspace 
        \text{has a complex}
        \\
        \text{Morse (non-degenerate) critical point at} \medspace q_{\alpha} \in S_{\alpha}.
    \end{cases}
\end{equation}
Then by \cite[Lemma 7.2.2]{KS85} the complex (not necessarily conic) 
Lagrangian submanifold 
\begin{equation}
    \Lambda^{\phi} \coloneq \mathrm{d} \phi (W) =\{ \mathrm{d} \phi (x) \: 
\mid \: x\in W\} \qquad \subset T^\ast X_0
\end{equation}
intersects $\Theta_{\alpha} \subset T^\ast X_0$ transversally.
Now we take a holomorphic coordinate $x= (x_1, \ldots , x_n)$  
$ (n\coloneq \mathrm{dim}_{\CC} X_0)$
of $X_0$ such that $q_{\alpha} = \{ x=0\}, \medspace S_{\alpha} =\{ x_1= 
\cdots =x_d=0 \} \medspace (n-d= \mathrm{dim}_{\CC} S_{\alpha}),
\medspace p_{\alpha} = (0, \mathrm{d} x_1) \in T_{S_{\alpha}}^\ast X_0$
and 
\begin{equation}
    \phi (x) = x_1+ (x_{d+1}^2 + \cdots + x_n^2).
\end{equation}
Then we can easily see that the real analytic function $\varphi \coloneq \mathrm{Re} \, \phi$
satisfies the following conditions:
\begin{equation}
    \begin{cases}
        (\textup{i})\: \varphi (q_{\alpha})=0, \quad 2 \cdot 
\mathrm{d} \varphi (q_{\alpha}) = p_{\alpha} \in \Theta_{\alpha} \subset T_{S_{\alpha}}^\ast X_0, \\
        (\textup{ii})\: \varphi|_{S_{\alpha}} \colon S_{\alpha} \longrightarrow \RR \medspace 
        \text{has a Morse (non-degenerate) critical point at} 
\\ \medspace q_{\alpha} \in S_{\alpha} \medspace \text{with 
Morse index} \medspace n-d= \mathrm{dim}_{\CC} S_{\alpha}.
    \end{cases}
\end{equation}
By \cite[Proposition 6.6.1 (ii)]{KS90} we take an object $M^{\alpha} \in \Db_{\RR-\mathrm{c}} (\pt)$
of $\Db_{\RR-\mathrm{c}} (\pt)$ such that for the object 
$M_{S_{\alpha}}^{\alpha} \coloneq k_{S_{\alpha}} 
\otimes_{k_{X_0}} M_{X_0}^{\alpha} \in \Db_{\RR-\mathrm{c}} (X_0)$
we have an isomorphism $\psi_g^{\RR} (F|_H) [-1] \simeq M_{S_{\alpha}}^{\alpha}$
in $\Db (X_0; p_{\alpha})$.
Then we obtain isomorphims
\begin{equation}
    \rsect_{\{ \varphi \geq 0\}} \left( \psi_g^{\RR} (F|_H) [-1]\right)_{q_{\alpha}}
    \simeq \rsect_{\{ \varphi \geq 0\}} \left( M_{S_\alpha}^{\alpha} \right)_{q_{\alpha}} 
    \simeq M^{\alpha} [-\mathrm{dim}_{\CC} S_{\alpha}].
\end{equation}
Hence it suffices to show that we have the concentration 
\begin{equation}
    H^j \rsect_{\{ \varphi \geq 0\}} \left( \psi_g^{\RR} (F|_H) [-1]\right) \simeq 0 \quad (j\neq 0).
\end{equation}
On the other hand, taking $r_1 >0, \varepsilon_0 >0$ and $\delta_0 >0$ 
and defining $\Lambda_a^{\prime} \subset T^\ast X_a$ $(0< a \ll \delta )$ 
etc. as in the proof of Theorem \ref{thm-limformula} 
we obtain isomorphisms 
\begin{equation}
    \rsect_{\{ \varphi \geq 0\}} (\psi_g^{\RR}(F|_H))_{q_{\alpha}} \simeq 
\rsect \left( X_a; (F_a)_{\overline{B(r_1)} \cap \{ -\varepsilon_0 < \varphi \leq \varepsilon\}}\right)
\end{equation}
for $0< a \leq \delta_0$.
Recall that by the proof Theorem \ref{thm-limformula} we have 
$\Lambda_a^{\prime} \rightsquigarrow \Lambda_0^{\prime}$
as $a \to +0$ (see the proof of Lemma \ref{ext-lem}) and for $0< a \ll \delta_0$ the Lagrangian 
subset $\Lambda_a^{\prime} \subset T^\ast X_a \simeq T^\ast X_0$
intersects 
$\mathrm{d} \varphi ( \overline{B(r_1)} \cap \{ |\varphi| 
\leq \varepsilon \} )$ $\subset T^\ast X_0$ 
only at finitely many points on a neighborhood of $p_{\alpha} 
\in \Theta_{\alpha} \subset T^\ast X_0$
in $T^\ast X_0$.
Fix such $0< a \ll \delta_0$ and set 
\begin{equation}
    \{ p_1, p_2, \ldots, p_m\} \coloneq \msupp ((F_a)_{\overline{B(r_1)}})
\cap \mathrm{d} \varphi (\overline{B(r_1)} \cap \{ |\varphi| \leq \varepsilon_0\}).
\end{equation}
Note that $F_a \in \Db(X_a)$ is complex constructible 
on $B(r_1) \subset Y=X_0 \simeq X_a$. 
Then by our construction of the subanalytic Whitney stratification 
$\mathcal{S}$, for any $1\leq i\leq m$ there exists a complex 
analytic sratum $T_i \in \mathcal{T}$ in $\mathcal{T}$ such that 
$T_i \subset X \setminus Y$ and $T_i \cap X_a$ is a complex submanifold of 
$X_a \simeq X_0=Y$, $S_{\alpha} \subset \overline{T_i}$ 
in $X$ and $p_i \in T_{T_i \cap X_a}^\ast (X_a)$. 
As $F_a [-1] =(F|_{X_a})[-1] \in \Dbc (X_a)$  $(0< a \ll \delta_0)$ being perverse,
for any $1\leq i\leq m$ there exists a finite dimensional $k$-vector space 
$V^i \not= 0$ such that for the object $V_{T_i \cap X_a}^i 
\coloneq k_{T_i \cap X_a} \otimes_{k_{X_a}} V_{X_a}^i \in \Db_{\RR-\mathrm{c}}(X_a)$
we have an isomorphism 
$F_a[-1] \simeq V_{T_i \cap X_a}^i [\mathrm{dim}_{\CC} (T_i\cap X_a)]$
in $\Db (T^\ast X_a ; p_i)$. 
Set $q_i \coloneq \pi_{X_0} (p_i) \in T_i \cap X_a \subset X_a \medspace (1\leq i \leq m)$
and recall that $T_{T_i \cap X_a}^\ast (X_a) \subset T^\ast X_a \simeq T^\ast X_0$ 
intersects $\mathrm{d} \varphi (B(r_1)) \subset T^\ast X_0$
transversally at the point $p_i$ for any $1\leq i \leq m$. 
Then by \cite[Lemma 7.3.2]{KS85} the holomorphic function 
$\phi|_{T_i \cap X_a} \colon T_i \cap X_a \longrightarrow \CC$
on $T_i \cap X_a$ has a complex Morse (non-degenerate) critical point at 
$q_i \in T_i \cap X_a$. This implies that the real analytic function 
$\mathrm{Re} (\phi|_{T_i\cap X_a}) =\varphi|_{T_i \cap X_a} 
\colon T_i \cap X_a \longrightarrow \RR$ on $T_i \cap X_a$ 
has a Morse (non-degenerate) critical point at $q_i \in T_i \cap X_a$ 
with Morse index 
$\mathrm{dim}_{\CC} (T_i\cap X_a)$. 
For any $1\leq i\leq m$ we thus obtain the concentration
\begin{equation}
    H^j \rsect _{\{ \varphi \leq \varphi (q_i)\}} (F_a [-1])_{q_i} \simeq 0 \quad (j\neq 0).
\end{equation}
Finally, by applying the microlocal Morse theory of Kashiwara and Schapira 
in \cite{KS90} to the Morse function $\varphi =\Re \phi$
and the $\RR$-constructible sheaf $(F_a[-1])_{\overline{B(r_1)}} \in \Db_{\RR-\mathrm{c}}(X_a)$
we obtain the desired concentration
\begin{equation}
    H^j \rsect (X_a ; (F_a[-1])_{\overline{B(r_1)} \cap \{ 
-\varepsilon_0 < \varphi \leq \varepsilon_0\}}) \quad (j\neq 0).
\end{equation}
This completes the proof.
\end{proof}

In the example below, we will show that our new proof of
Theorem \ref{thm-perverse} allows us to calculate the characteristic cycles of 
nearby cycle sheaves very explicitly.

\begin{example}\label{ex-CCisosing}
    Let $X=\CC_x^n$ be the $n$-dimensional complex vector space 
    and $f\colon X=\CC_x^n \longrightarrow \CC$ a non-constant holomorphic 
    function on it such that $f(0)=0$ and the complex hypersurface
    $Z\coloneq f^{-1}(0) \subset X=\CC_x^n$ has an isolated 
    singular point at the origin $0\in Z \subset X=\CC_x^n$. 
    Let $i_f \colon X \longhookrightarrow X\times \CC_t \medspace (x\longmapsto (x,f(x)))$
    be the graph embedding of $f$ and $\iota \colon Z=f^{-1}(0) \longhookrightarrow X=\CC_x^n$
    the inclusion map. Then by \cite[Proposition 4.2.11]{Dim04} 
    and \cite[Exercise V\hspace{-1.2pt}I\hspace{-1.2pt}I\hspace{-1.2pt}I.15]{KS90}
    we obtain an isomorphism
    \begin{equation}
        \psi_t ({i_f}_\ast (\CC_X [n])) [-1] \simeq \iota_\ast \psi_f (\CC_X[n])[-1].
    \end{equation}
    By applying Theorem \ref{thm-perverse} to the perverse 
sheaf ${i_f}_\ast (\CC_X [n]) \in \Dbc(X\times \CC)$
    on $X\times \CC$ and $t\colon X\times \CC \longrightarrow \CC$
    we thus see that $\iota_\ast \psi_f (\CC_X [n]) [-1] \in \Dbc (X)$
    is a perverse sheaf on $X=\CC_x^n$.
    In order to check the validity of Theorem \ref{thm-Gins}, we shall calculate 
    its characteristic cycle in two different ways.
    First, we calculate it by using Kashiwara's index theorem 
    in \cite{Kas83a}.
    Let $Z_{\textup{reg}} (= Z\setminus \{ 0\}) \subset Z$
    be the regular part of $Z=f^{-1}(0)$ and $H\subset X=\CC_x^n$
    a complex hyperplane passing through the origin $0\in X=\CC_x^n$
   such that we have 
    \begin{equation}
        \overline{T_{Z_{\textup{reg}}}^\ast X} \cap T_H^\ast X \subset T_X^\ast X
    \end{equation}
on a neighborhood of $0 \in X= \CC^n_x$. 
    For such $H\subset X=\CC^n$ the complex hypersurface $Z\cap H\subset H$
    in it has also an isolated singular point at the origin $0\in Z\cap H$
    and hence we denote its Milnor fiber (resp. Milnor number)
    at the origin $0\in Z\cap H$ by $M_{f|_H, 0} \medspace (\text{resp.} \medspace m_{Z\cap H})$.
    Then by Kashiwara's formula in \cite{Kas83a} the Euler obstruction 
    $\Eu_Z \colon X \longrightarrow \mathbb{Z}$ of $Z\subset X$
    is given by the formula 
    \begin{equation}
        \Eu_Z (q) =
        \begin{cases}
             1 \quad &(q\in Z_{\textup{reg}}), \\
             1+(-1)^n m_{Z\cap H} \quad &(q=0), \\
             0 \quad &(\textup{otherwise}).
        \end{cases}
    \end{equation}
   We denote the Milnor fiber (resp. Milnor number) 
    of the complex hypersurface $Z\subset X=\CC_x^n$ at the origin $0\in Z$
    by $M_{f,0}$ (resp. $m_Z$) and for a constructible sheaf 
    $F\in \Dbc(X)$ over $k= \CC$ on $X=\CC_x^n$ we define a $\mathbb{Z}$-valued 
    function $\chi(F) \colon X \longrightarrow \mathbb{Z}$
    on $X$ by 
    \begin{equation}
        \chi(F)(q) \coloneq \sum_{j\in \mathbb{Z}} (-1)^j 
\mathrm{dim}_{\CC} (H^jF)_q \quad (q\in X).
    \end{equation}
   Then by Lemma \ref{lemm-Miln} we obtain 
    \begin{equation}
        \chi \left( \iota_\ast \psi_f (\CC_X[n]) [-1]\right) (q) =
        \begin{cases}
            (-1)^{n-1} \quad & (q\in Z_{\textup{reg}}), \\
            (-1)^{n-1} + m_Z \quad & (q=0), \\
            0 \quad & (\textup{otherwise}).    
        \end{cases}
    \end{equation}
    From now, we shall use the theory of $\mathcal{D}$-modules and 
    for the standard terminology in it we refer to e.g. \cite{HTT08} and \cite{Kas03}. 
    Let $\mathcal{M}$ be the regular holonomic $\SD_X$-module 
    which corresponds to the perverse sheaf $\iota_\ast \psi_f (\CC_X[n])[-1] \in \Dbc(X)$
    via the Riemann-Hilbert correspondence:
    \begin{equation}
        \Sol_X (\mathcal{M})[n] \simeq \iota_\ast \psi_f (\CC_X[n])[-1].
    \end{equation}
    Then we have
    \begin{equation}
        \chi (\Sol_X(\mathcal{M}))(q) =
        \begin{cases}
            -1 \quad &(q\in Z_{\textup{reg}}), \\
            -1+(-1)^n m_Z \quad &(q=0), \\
            0 \quad & (\textup{otherwise})
        \end{cases}
    \end{equation}
    and the $\mathbb{Z}$-valued function $\chi(\Sol_X(\mathcal{M}))$ is 
decomposed as follows:
    \begin{equation}
        \chi(\Sol_X(\SM)) =-1\cdot \Eu_Z + (-1)^n (m_Z+m_{Z\cap H})\cdot \mathbf{1}_{\{ 0\}}.
    \end{equation}
    By Kashiwara's index theorem in \cite{Kas83a}, this implies that 
    the characteristic cycle $\CCyc (\SM)$ of $\SM$ is given by 
    the formula 
    \begin{equation}
        \CCyc (\SM) =1\cdot [\overline{T_{Z_{\textup{reg}}}^\ast X}] 
+ (m_Z+m_{Z\cap H}) \cdot [T_{\{ 0\}}^\ast X].
    \end{equation}
    We thus see that for any point $p$ of $T_{Z_{\textup{reg}}} X$
    (resp. $T_{\{0\}}^\ast X \setminus \overline{T_{Z_{\textup{reg}}}^\ast X}$)
    there exists an isomorphism
    \begin{equation}
       \Sol_X(\SM) [n] \simeq \CC_{Z_{\textup{reg}}} [n-1] \quad 
(\textup{resp.} \medspace \CC_{\{ 0\}}^{m_Z +m_{Z\cap H}})
    \end{equation}
    in the localized category $\Db(X; p)$.
    Consequently, we obtain a formula 
    \begin{align}\label{eq-char}
    \begin{split}
        \CCyc (\iota_\ast \psi_f (\CC_X [n])[-1]) 
        &=\CCyc (\Sol_X(\SM)[n]) \\
        &=(-1)^{n-1} [T_{Z_{\textup{reg}}}^\ast X] +(m_Z+m_{Z\cap H})\cdot [T_{\{ 0\}}^\ast X]
    \end{split}
    \end{align}
    of the characteristic cycle of the perverse sheaf $\iota_\ast \psi_f (\CC_X[n])[-1] \in \Dbc(X)$
    on $X$.
    Next, we shall calculate it by our new proof of Theorem \ref{thm-perverse}. 
    We have only to calculate its coefficient of the Lagrangian 
cycle $[T_{\{ 0\}}^\ast X] \in \sect(X;\mathcal{L}_X)$.
    Recall that by Theorem \ref{thm-Gins} we have 
    \begin{align}
        \CCyc (\iota_\ast \psi_f(\CC_X[n])[-1]) &= \CCyc (\psi_t({i_f}_\ast(\CC_X[n]))[-1]) \\
        &=\lim_{a\to +0} \CCyc ({i_f}_\ast (\CC_X[n-1])|_{t=a}) \\
        &=\lim_{a\to +0} \CCyc (\CC_{Z_a}[n-1]),
    \end{align}
    where for $a>0$ we set $Z_a \coloneq \{ x\in X=\CC^n \mid f(x)=a\} \subset X$.
    In order to calculate the coefficient of $[T_{\{ 0\}}^\ast X]$ in 
    this limit, we take a generic complex vector 
    \begin{equation}
        \begin{pmatrix}
            c_1\\
            c_2\\
            \vdots \\
            c_n
        \end{pmatrix}
        \in \CC^n \setminus \{ 0\}
    \end{equation}
    and define a complex linear form $\phi \colon X =\CC^n \longrightarrow \CC$
    by 
    \begin{equation}
        \phi(x) \coloneq \sum_{i=1}^n c_i x_i \quad (x=(x_1, \ldots ,x_n)\in X=\CC^n).
    \end{equation}
    For $0<a \ll 1$ let $N\geq 0$ be the number of the (complex Morse) 
    critical points of $\phi|_{Z_a} \colon Z_a \longrightarrow \CC$
    on a neighborhood of the origin $0\in Z_a \subset X=\CC^n$.
    Then by the proof of Theorem \ref{thm-perverse} we see that the coefficient of 
    $[T_{\{ 0\}}^\ast X]$ in the limit $\lim_{a\to +0} \CCyc (\CC_{Z_a}[n-1])$
    is equal to $N$.
    We thus obtain 
    \begin{equation}
        \CCyc \left( \iota_\ast \psi_f (\CC_X[n])[-1] \right) 
        = (-1)^{n-1}[T_{Z_{\textup{reg}}}^\ast X] + N\cdot [T_{\{ 0\}}^\ast X].
    \end{equation}
    Combining this with (\ref{eq-char}), we find a formula $N= m_Z+m_{Z\cap H}$,
    which can be explained by the classical results in L\^{e} \cite{Le73} as follows.
    We may assume that the complex hyperplane $H\coloneq \phi^{-1}(0) \subset X=\CC^n$
    satisfies the condition 
    \begin{equation}
        \overline{T_{Z_{\textup{reg}}}^\ast X} \cap T_H^\ast X \subset T_X^\ast X
    \end{equation}
on a neighborhood of $0 \in X= \CC^n_x$. 
   Then in \cite{Le73}, by using the Morse theory for the function 
    $|\phi| \colon Z_a \longrightarrow \RR$ on $Z_a$, L\^{e} proved that 
    $M_{f,0}$ is obtained from $M_{f|_H, 0}$ by attaching $(n-1)$-handles $N$ times.
    This implies that we have 
    \begin{equation}
        \chi(M_{f,0}) -\chi(M_{f|_H, 0}) =(-1)^{n-1}N,
    \end{equation}
    from which the formula $N=m_Z+ m_{Z\cap H}$ immediately follows.
\end{example}

From now, let us consider the characteristic cycles of vanishing cycle sheaves.
Let $X$ be a complex manifold and $f\colon X\longrightarrow\CC$ a non-constant
holomorphic function on it.
As before, set $Y\coloneq f^{-1}(0)\subset X$ and let $i\colon Y\longhookrightarrow X$
be the inclusion map.
Then in \cite{Del73} for $F\in \BDc(X)$ Deligne defined 
an object $\phi_f(F)\in\BDc(Y)$ which fits into the distinguished triangle
\begin{equation}
    i^{-1}F\longrightarrow\psi_f(F)\longrightarrow\phi_f(F)
    \overset{+1}{\longrightarrow}
\end{equation}
in $\BDc(Y)$ (see e.g. \cite[Section 8.6]{KS90}) for the precise definition).
We call it the vanishing cycle sheaf of $F$ along $f$.
In order to study its characteristic cycle, we assume that 
$f\colon X\longrightarrow\CC$ is smooth and hence $Y=f^{-1}(0)\subset X$ is a 
smooth complex hypersurface.
In this case, by \cite[Proposition 8.6.3]{KS90} for the section
$\sigma\colon Y\longhookrightarrow T_Y^\ast X$ ($x\longmapsto \rmd f(x)$) of
the conormal bundle $T_Y^\ast X\longrightarrow Y$ there exists an isomorphism
\begin{equation}
    \phi_f(F)\simeq \sigma^{-1}\mu_Y(F)[1].
\end{equation}
Note that our definition of $\phi_f(F)$ is different from the one in 
\cite[Definition 8.6.2]{KS90} by shift $[1]$.
Recall that at the end of Section \ref{sec-realnearby} we obtained the formulas of 
$\CCyc(\nu_Y(F))$ and $\CCyc(\mu_Y(F))$.
Moreover, after the works by Brylinski \cite{Bry86} and 
Kashiwara-Schapira \cite{KS90} etc., it is well-known that $\nu_Y(F)\in\BDc(T_Y X)$ and 
$\mu_Y(F)\in\BDc(T_Y^\ast X)$ are not only complex constructible but also
monodromic in the sense of Verdier \cite{Ver83} with respect to the standard 
actions of $\CC^\ast$ on $T_Y X$ and $T_Y^\ast X$ respectively.
This implies that $\msupp(\nu_Y(F))\subset T^\ast(T_Y X)$ 
(resp. $\msupp(\mu_Y(F))\subset T^\ast (T_Y^\ast X)$) is biconic i.e. invariant under
the two standard actions of $\CC^\ast$ on $T^\ast (T_Y X)$ 
(resp. $T^\ast(T_Y^\ast X)$) as in \cite[(5.5.8)]{KS90}.
We thus see that the morphism $\sigma\colon Y\longhookrightarrow T_Y^\ast X$
is non-characteristic for $\mu_Y(F)\in\BDc(T_Y^\ast X)$ and obtain a formula for 
$\CCyc(\phi_f(F))$.
In order to see it more explicitly, let 
$(x,\tau)=(x_1,\dots,x_n,\tau)$ be a local coordinate of $X$ such that
$f(x,\tau)=\tau$ and $Y=\{\tau=0\}$ and $(x,\tau)$ (resp. $(x,\tau^\ast)$)
the coordinate of $T_Y X$ (resp. $T_Y^\ast X$) associated to it.
Then by \cite[Proposition 5.5.1]{KS90} we obtain an identification
\begin{equation}
\begin{tikzcd}[row sep=0.1cm]
    \Phi_{T_Y X} \colon T^\ast(T_Y X) \ar[r,"\sim"] &
    T^\ast(T_Y^\ast X) \\[-8pt]
    \: \qquad \raisebox{-2pt}{\rotatebox{90}{$\in$}} & 
\raisebox{-2pt}{\rotatebox{90}{$\in$}} \\[-8pt]
    \quad \qquad (x,\tau;x^\ast,\tau^\ast) \ar[r,mapsto] &
    (x,\tau^\ast; x^\ast,-\tau)
\end{tikzcd}
\end{equation}
by which we have $\CCyc(\mu_Y(F))=\CCyc(\nu_Y(F))$ (see Section \ref{sec-realnearby}).
For the local coordinate $(x,\tau;x^\ast \tau^\ast)$ of $T^\ast(T_Y X)$ set 
$Z\coloneq\{\tau^\ast=1\}$ ($\simeq T^\ast Y\times\CC_\tau$) $\subset T^\ast(T_Y X)$
and let $h\colon Z\longhookrightarrow T^\ast (T_Y X)$ be the inclusion map.
Note that on a neighborhood of the complex hypersurface 
$Z=\{\tau^\ast=1\}\subset T^\ast(T_Y X)$ in $T^\ast(T_Y X)$ the support of the 
Lagrangian cycle $\CCyc(\nu_Y(F))$ is contained in 
$\{\tau=0\}\subset T^\ast(T_Y X)$.
This implies that its pull-back $h^\ast\CCyc(\nu_Y(F))$ by 
$h\colon Z\longhookrightarrow T^\ast(T_Y X)$ is supported on $T^\ast Y\times\{0\}$
($\simeq T^\ast Y$) $\subset Z$.
We thus can regard it as a Lagrangian cycle in $T^\ast Y$ and obtain
a formula $\CCyc(\phi_f(F))=-h^\ast\CCyc(\nu_Y(F))$.
By the coordinate $(x,\tau)$ of $X$, for $a>0$ consider the morphism
$\Phi_a\colon X\longrightarrow X$ ($(x,\tau)\longmapsto(x,a\tau)$)
and recall the commutative diagram 
\begin{equation}
\begin{tikzcd}
    T_Y X \ar[r,hook] \ar[d] & 
    \tl{X}_Y \ar[d,"p"] &
    \Omega=t^{-1}(\RR_{>0}) \ar[l,hook'] \ar[ld,"\tl{p}"] \\
    Y \ar[r,"i",hook] & X & 
\end{tikzcd}
\end{equation}
associated to the normal deformation $\tl{X}_Y$ of $X$ along $Y$. 
Then for the coordinate $(x,\tau,t)$ of $\tl{X}_Y$ and $a>0$ we have an isomorphism
\begin{equation}
    (\tl{p}^{-1}F)\vbar_{t^{-1}(a)}\simeq \Phi_a^{-1}(F)
\end{equation}
and obtain the following result by Lemma \ref{lem-pullPhi} and Theorem \ref{thm-CCspe}. 

\begin{theorem}\label{thm-CCvanish}
    For $F\in\BDc(X)$ we have 
    \begin{equation}
        \CCyc(\phi_f(F))
= - h^\ast \Bigl( \lim_{a\to+0} \CCyc(\Phi_a^{-1}(F)) \Bigr) 
=-\lim_{a\to+0}h^\ast\CCyc(\Phi_a^{-1}(F)),
    \end{equation}
    where $h\colon T^\ast Y\times \CC\longhookrightarrow T^\ast X$
    $\bigl(((x;x^\ast),\tau)\longmapsto(x,\tau;x^\ast,1)\bigr)$ is the inclusion map.
\end{theorem}

\begin{example}\label{ex-1a} 
Let us consider the situation in Example \ref{ex-CCisosing} and use the notations in it.
Then by \cite[Proposition 4.2.11]{Dim04} and 
\cite[Exercise V\hspace{-1.2pt}I\hspace{-1.2pt}I\hspace{-1.2pt}I.15]{KS90}
we obtain an isomorphism
\begin{equation}
    \phi_t(i_{f\ast}(\CC_X)[n])[-1]
    \simeq \iota_\ast \phi_f(\CC_X[n])[-1].
\end{equation}
Now our objective here is to calculate the characteristic cycle of 
the perverse sheaf $\iota_\ast \phi_f(\CC_X[n])[-1] \in \BDc(X)$.
For $a>0$ consider the mosphisms 
$\Phi_a\colon X\times\CC\longrightarrow X\times\CC$ ($(x,t)\longmapsto (x,at)$)
and $i_{f/a}\colon X\longhookrightarrow X\times\CC$
($x\longmapsto(x,\frac{1}{a}f(x))$).
Then for $a>0$ we have an isomorphism
\begin{equation}
    \Phi_a^{-1}\bigl( i_{f\ast}(\CC_X[n])[-1] \bigr)
    \simeq i_{f/a \ast}(\CC_X[n])[-1].
\end{equation}
Moreover, we have
\begin{equation}
    \CCyc\bigl( i_{f/a\ast}(\CC_X[n])[-1]\bigr)
    =(-1)^{n-1}[T_{\Gamma_a}^\ast(X\times\CC)],
\end{equation}
where we set $\Gamma_a\coloneq i_{f/a}(X)$ ($\simeq X$) $\subset X\times\CC$.
Since we assume here that $f\colon X=\CC^n\longrightarrow\CC$ has an isolated singular
point at the origin $0\in Z=f^{-1}(0)$,
$\Gamma_a$ intersects $X\times\{0\}$ transversally on 
$(X\setminus\{0\}) \times\{0\}$ 
and Theorem \ref{thm-CCvanish} implies that the support of the characteristic cycle of 
$\iota_\ast\phi_f(\CC_X[n])[-1]$ is contained in $T_{\{0\}}^\ast X\subset T^\ast X$.
Let us show that its coefficient of $[T_{\{0\}}^\ast X]$ is equal to $m_Z$.
For this purpose, we take a generic complex vector 
\begin{equation}
    \begin{pmatrix}
        c_1\\
        c_2\\
        \vdots \\
        c_n
    \end{pmatrix}
    \in \CC^n \setminus \{ 0\}
\end{equation}
and define a complex linear form 
$\tl{\phi}\colon X\times\CC=\CC^{n+1}\longrightarrow\CC$ by
\begin{equation}
    \tl{\phi}(x,t)\coloneq t+\sum_{i=1}^n c_i x_i
    \quad ((x,t)\in X\times\CC).
\end{equation}
Then for $a>0$ the restriction 
$\tl{\phi}\vbar_{\Gamma_a}\colon \Gamma_a\longrightarrow\CC$ of $\tl{\phi}$
to $\Gamma_a\simeq X=\CC^n$ is equal to
\begin{equation}
    \Gamma_a\simeq X=\CC_x^n
    \longrightarrow \CC
    \quad \Bigl( x\longmapsto \frac{1}{a}\{f(x)+a\sum_{i=1}^n c_i x_i\}\Bigr).
\end{equation}
Note that for $0<a\ll1$ the function $f(x)+a\sum_{i=1}^n c_i x_i$ is a small
linear perturbation i.e. a Morsification of $f(x)$. 
Then by \cite[Proposition 2.2]{Bro88} 
we see that for $0<a<\ll1$ all of its critical points are 
complex Morse (non-degenerate) 
and their number is equal to $m_Z$. 
Hence, by the proof of Theorem \ref{thm-perverse} we see that the coefficient 
of $[T_{\{ 0\} \times \{ 0\}} ^\ast (X\times \CC)]$
    in the limit 
    \begin{equation}
        \lim_{a\to +0} \CCyc \left( \Phi_a^{-1} ({i_f}_\ast (\CC_X[n])[-1])\right)
    \end{equation}
    is equal to $-m_Z$. 
    By Theorem \ref{thm-CCvanish}, this implies that we have 
\begin{equation}
    \CCyc\bigl( \iota_\ast\phi_f(\CC_X[n])[-1] \bigr)
    =m_Z\cdot[T_{\{0\}}^\ast X]
\end{equation}
as expected.
\end{example}

\begin{example}
    Let $X=\CC^3$ be the $3$-dimensional complex vector space 
    and $(x,y,z)$ its linear coordinate and define a 
    holomorphic function $f\colon X=\CC^3 \longrightarrow \CC$
    by $f(x,y,z)\coloneq xyz \: ((x,y,z)\in X=\CC^3)$.
    Set $Z\coloneq f^{-1}(0) \subset X=\CC^3$
    and let $\iota \colon Z \longhookrightarrow X=\CC^3$
    be the inclusion map. In this situation, we shall calculate the 
    characteristic cycle of the perverse sheaf 
    $\iota_\ast \phi_f(\CC_X [3])[-1] \in \Dbc(X)$ on $X$ in two 
    different ways.
    First, we calculate it by using Kashiwara's index theorem.
    Set $H_1\coloneq \{ x=0\}, H_2\coloneq \{ y=0\}, H_3\coloneq\{ z=0\},
    L_1\coloneq \{ y=z=0\}, L_2\coloneq \{ x=z=0\}, L_3\coloneq \{ x=y=0\}$.
    Then we have $Z=H_1\cup H_2 \cup H_3$ and the singular set 
    $Z_{\textup{sing}}$ of $Z$ is $L_1\cup L_2 \cup L_3$.
    Recall that the support of the vanishing cycle sheaf 
    $\phi_f(\CC_X) \in \Dbc(Z)$ is contained in $Z_{\textup{sing}}
    =L_1\cup L_2\cup L_3 \subset Z$.
    For a point $q\in Z$ of $Z=f^{-1}(0)$ let $M_{f,q}\subset X\setminus Z$
    be the Milnor fiber of $f$ at it.
    Then it is well-known that $M_{f,q}$
    is homotopic to $\pt$ (resp. $S^1, S^1\times S^1$) if 
    $q\in Z\setminus Z_{\textup{sing}}$ (resp. $q\in Z_{\textup{sing}}\setminus \{ 0\},\; q=0$) 
    (see Oka \cite[Example (3.7)]{Oka} etc.). 
    This implies that we have 
    \begin{equation}
        \chi \left( \iota_\ast \phi_f (\CC_X[3])[-1]\right) (q) =
        \begin{cases}
            0 \quad & (q\in Z\setminus Z_{\textup{sing}}), \\
            -1 \quad & (q\in Z_{\textup{sing}}), \\
            0 \quad & (\text{otherwise})
        \end{cases}
    \end{equation}
    and hence 
    \begin{equation}
        \chi \left( \iota_\ast \phi_f (\CC_X[3])[-1]\right) = 
-\textbf{1}_{L_1} -\textbf{1}_{L_2} -\textbf{1}_{L_3} +2\cdot \textbf{1}_{\{ 0\}}.
    \end{equation}
    Then for the regular holonomic $\SD_X$-module $\SN$ such that 
    \begin{equation}
        \Sol_X(\SN)[3] \simeq \iota_\ast \phi_f(\CC_X[3])[-1],
    \end{equation}
    by applying Kashiwara's index theorem to the equality 
    \begin{equation}
        \chi (\Sol_X(\SN)) =\textbf{1}_{L_1} + \textbf{1}_{L_2} + \textbf{1}_{L_3} 
- 2\cdot \textbf{1}_{\{ 0\}}
    \end{equation}
    we obtain 
    \begin{equation}
        \CCyc(\SN) = \sum_{i=1}^3 \ [T_{L_i}^\ast X] + 2\cdot [T_{\{0\}}^\ast X].
    \end{equation}
This implies that for the perverse sheaf 
$\iota_\ast \phi_f (\CC_X[3])[-1] \in \Dbc (X)$ on $X= \CC^3$ 
we have 
    \begin{equation}
        \CCyc( \iota_\ast \phi_f (\CC_X[3])[-1] ) = 
- \sum_{i=1}^3 \ [T_{L_i}^\ast X] + 2\cdot [T_{\{0\}}^\ast X].
    \end{equation}
Next we shall calculate it by using our (new) proof Theorem \ref{thm-perverse} 
and Theorem \ref{thm-CCvanish}. 
    Let $i_f\colon X \longhookrightarrow X\times \CC_t \: (x\longmapsto (x,f(x)))$
    be the graph embedding of $f$.
    Then by \cite[Proposition 4.2.11]{Dim04} and
 \cite[Exercise V\hspace{-1.2pt}I\hspace{-1.2pt}I\hspace{-1.2pt}I.15]{KS90}
    there exists an isomorphism 
    \begin{equation}
        \phi_t \left( {i_f}_\ast (\CC_X[3])\right) [-1] \simeq \iota_\ast \phi_f (\CC_X[3])[-1].
    \end{equation}
    It suffices to show that the coefficient of $[T_{\{ 0\}}^\ast X]$
    in the characteristic cycle of the perverse sheaf $\phi_t
 \left( {i_f}_\ast (\CC_X[3])\right) [-1] \in \Dbc(X)$
    is equal to $2$.
    For $a>0$ consider the morphisms $\Phi_a \colon X\times 
\CC \longrightarrow X\times \CC \: ((x,t)\longmapsto (x,at))$
    and $i_{f/a}\colon X\longhookrightarrow X\times \CC \: 
(x\longmapsto (x,\frac{1}{a} f(x)))$
    as in Example \ref{ex-1a}. 
    Then for $a>0$ we have an isomorphism 
    \begin{equation}
        \Phi_a^{-1} \left( {i_f}_\ast (\CC_X[3])[-1]\right) = [T_{\Gamma_a}^\ast (X\times \CC)],
    \end{equation}
    where we set $\Gamma_a \coloneq i_{f/a} (X) (\simeq X) \subset X\times \CC$.
    Now we define a holomorphic function $\widetilde{\phi} \colon 
X\times \CC =\CC^4 \longrightarrow \CC$
    by
    \begin{equation}
        \widetilde{\phi} (x,y,z,t)=t-x-y-z \quad ((x,y,z,t)\in X\times \CC).
    \end{equation}
    Then for $a>0$ the restriction $\widetilde{\phi}|_{\Gamma_a} 
\colon \Gamma_a \longrightarrow \CC$ of $\widetilde{\phi}$
    to $\Gamma_a \simeq X =\CC^3$ is equal to 
    \begin{equation}
        \Gamma_a \simeq X=\CC^3 \longrightarrow \CC \quad ((x,y,z)
\longmapsto \frac{1}{a} (xyz-ax-ay-az)).
    \end{equation}
    For $a>0$ let us show that $\widetilde{\phi}|_{\Gamma_a}$
    has two (complex Morse) critical points. For this purpose,
    define a holonomic function $g\colon X=\CC^3 \longrightarrow \CC$
    by $g(x,y,z)\coloneq xyz-ax-ay-az \: ((x,y,z)\in X=\CC^3)$.
    Then we have 
    \begin{align}
        \mathrm{d} g(x,y,z) =0 &\iff 
        \begin{cases}
            yz=a \\
            xz=a \\
            zx=a \\
        \end{cases} \\
        &\Longrightarrow \: (xyz)^2 =(yz)(xz)(zx)=a^3 \: \Longrightarrow
        \: xyz=\pm a^{\frac{3}{2}}.
    \end{align}
    We thus obtain 
    \begin{equation}
        (x,y,z)=\left( \frac{xyz}{yz}, \frac{xyz}{xz}, 
\frac{xyz}{xy} \right) =\pm (\sqrt{a}, \sqrt{a}, \sqrt{a})
    \end{equation}
    as we desired.
    Now, by the proof of Theorem \ref{thm-perverse} we see that the coefficient 
of $[T_{\{ 0\} \times \{ 0\}} ^\ast (X\times \CC)]$
    in the limit 
    \begin{equation}
        \lim_{a\to +0} \CCyc \left( \Phi_a^{-1} ({i_f}_\ast (\CC_X[3])[-1])\right)
    \end{equation}
    is equal to $-2$.
    By Theorem \ref{thm-CCvanish}, this implies that the coefficient of $[T_{\{ 0\}}^\ast X]$
    in the characteristic cycle of $\phi_t ({i_f}_\ast (\CC_X[3]))[-1]$
    is equal to $2$ as expected.
\end{example}

\appendix

\section{Some results on commutative diagrams for the six operations}
In this appendix, we prepare some auxiliary results on commutative diagrams.
Note that the following Cartesian diagram of topological spaces
\begin{equation}
\begin{tikzcd}
      A \ar[r,"f"] \ar[d,"p"] \ar[dr,"\square",phantom]  &  B \ar[d,"q"'] \\
      P \ar[r,"h"'] & Q .
\end{tikzcd}
\end{equation}
induces a natural morphism
\begin{equation}
      h^{-1} \mathrm{R}q_\ast \to \mathrm{R} p_\ast p^{-1} h^{-1} \mathrm{R} q_\ast \simto \mathrm{R} p_\ast f^{-1} q^{-1} \mathrm{R} q_\ast
   \to \mathrm{R} p_\ast f^{-1}.
\end{equation}

\begin{lemma}\label{lem-bs}
Assume that we are given the following Cartesian diagram of topological spaces:

\begin{equation}
\begin{tikzcd}
    A \ar[r,"f"] \ar[d,"p"] \ar[dr,"\square",phantom] &  
      B \ar[r,"g"] \ar[d,"q"] \ar[dr,"\square",phantom,pos=0.55] 
      & C \ar[d,"r"] \\
      P \ar[r,"h"'] & Q \ar[r,"i"'] & R .
\end{tikzcd}
\end{equation}
Then the following diagram commutes:
\begin{equation}
\begin{tikzcd}
      (h^{-1} i^{-1}) \mathrm{R} r_\ast \ar[r] \ar[d] &
      h^{-1} \mathrm{R} q_\ast g^{-1} \ar[d] \\
      \mathrm{R} p_\ast(f^{-1} g^{-1}) \ar[r, equal] & \mathrm{R} p_\ast f^{-1} g^{-1} .
\end{tikzcd}
\end{equation}

\end{lemma}
\begin{proof}
   It is enough to prove the commutativity of the diagram below:
    \begin{equation}
    \begin{tikzcd}    
         h^{-1} i^{-1} \mathrm{R} r_\ast \ar[d] \ar[r] & 
         h^{-1} \mathrm{R} q_\ast g^{-1} r^{-1} \mathrm{R} r_\ast \ar[d] \ar[r] &
         h^{-1} \mathrm{R} q_\ast g^{-1} \ar[d] \\
         \mathrm{R} p_\ast f^{-1} q^{-1} i^{-1} \mathrm{R} r_\ast \ar[rd,equal, "(\#)"] \ar[r] &
         \mathrm{R} p_\ast f^{-1} q^{-1} \mathrm{R} q_\ast q^{-1} i^{-1} \mathrm{R} r_\ast \ar[d] \ar[r] &
         \mathrm{R} p_\ast f^{-1} q^{-1} \mathrm{R} q_\ast g^{-1} \ar[d] \\
         {}& \mathrm{R} p_\ast f^{-1} q^{-1} i^{-1} \mathrm{R} r_\ast \ar[r] &
         \mathrm{R} p_\ast f^{-1} g^{-1} ,
    \end{tikzcd}
    \end{equation}
where all the arrows are units and counits. Therefore the statement follows from the
functoriality of them and the triangle identity applied to $(\#)$.
\end{proof}

\begin{lemma}\label{lem-intunit}
We consider the Cartesian diagram (\ref{eq-app2}) of manifolds, where $f$ is smooth and $g$ is a closed embedding.
\begin{equation}\label{eq-app2}
    \begin{tikzcd}
        Y^\prime \ar[r,hook,"g^\prime"] \ar[d,"f^\prime"] \ar[dr,phantom,"\square",pos=0.55] &
        Y \ar[d,"f"] \\
        X^\prime \ar[r,hook,"g"'] & 
        X,
\end{tikzcd}
\end{equation}
Then the following diagram commutes:
\begin{equation}
\begin{tikzcd}
    \rmR f_!f^! \ar[r] \ar[d] & 
    \rmR g_\ast g^{-1} \rmR f_!f^! \ar[r] \ar[d,"\sim",sloped] & 
    \rmR g_\ast g^{-1} \\
    \rmR f_! \rmR g_\ast^\prime g^{\prime-1} f^! \ar[r,dash,"\sim"] &
    \rmR g_\ast \rmR f_!^\prime g^{\prime-1} f^! \ar[r,"\sim"]  &
    \rmR g_\ast \rmR f_!^\prime f^{\prime!} g^{-1}. \ar[u] 
\end{tikzcd}
\end{equation}
\end{lemma}

\begin{proof}
    Since the commutativity of the left square is obvious, we shall prove that the diagram (\ref{eq-app3}) commutes.
    \begin{equation}\label{eq-app3}
    \begin{tikzcd}
        g^{-1} \rmR f_!f^! \ar[r] \ar[d,"\sim",sloped] &
        g^{-1} \\
        \rmR f_!^\prime g^{\prime-1} f^! \ar[r,"\sim"]  &
        \rmR f_!^\prime f^{\prime!} g^{-1}. \ar[u]
    \end{tikzcd}
    \end{equation}
    Note that the isomorphism $g^{\prime-1} f^!\simto f^{\prime!} g^{-1}$ is induced by
    \begin{equation}
        \rmR f_!^\prime g^{\prime-1} f^! \simot 
        g^{-1} \rmR f_!f^! \longrightarrow g^{-1}.
    \end{equation}
    As the adjunction isomorphism is determined by the unit, the bottom arrow of (\ref{eq-app3}) is 
    \begin{equation}
        \rmR f_!^{\prime}g^{\prime-1}f^! \longrightarrow 
        \rmR f_!^{\prime} f^{\prime !} \rmR f_!^{\prime}g^{\prime-1}f^! \simot
        \rmR f_!^{\prime} f^{\prime !} g^{-1} \rmR f_! f^!  \longrightarrow
        \rmR f_!^{\prime} f^{\prime !} g^{-1}.
    \end{equation}
    Then the commutativity of (\ref{eq-app3}) follows from the commutative diagram below.
    \begin{equation}
    \begin{tikzcd}
        &
        \rmR f_!^{\prime}g^{\prime-1}f^! &
        g^{-1} \rmR f_! f^! \ar[l,"\sim"'] \ar[r] &
        g^{-1} \\
        \rmR f_!^{\prime}g^{\prime-1}f^! \ar[r] \ar[ur,equal] &
        \rmR f_!^{\prime} f^{\prime !} \rmR f_!^{\prime}g^{\prime-1}f^! \ar[u] &
        \rmR f_!^{\prime} f^{\prime !} g^{-1} \rmR f_! f^! \ar[l,"\sim"'] \ar[r] \ar[u] &
        \rmR f_!^{\prime} f^{\prime !} g^{-1}. \ar[u] 
    \end{tikzcd}
    \end{equation}
\end{proof}

\section{Proof of Lemma \ref{ext-lem}}\label{app-proof}
In this appendix, we give a proof of Lemma \ref{ext-lem}. 
We shall use the notations in Lemma \ref{ex-lem}. 
First, we prepare the following result. 

\begin{lemma}\label{lem-tangent}
For a subanalytic submanifold $S \subset X$ of the real analytic manifold $X$,
its tangent bundle $TS \subset TX$ is subanalytic in $TX$. 
\end{lemma}

\begin{proof}
The problem being local, we can choose a real analytic metric of 
the vector bundle $T^\ast X$ 
on $X$ and define a sphere bundle $G^\ast (X) \subset T^\ast X$ in it. 
We set $G^\ast(S) \coloneq T_S^\ast X \cap G^\ast(X) \subset T^\ast X$. 
Then by \cite[Proposition 8.3.1]{KS90}, 
$G^\ast(S)$ is subanalytic in $T^\ast X$. 
Moreover, we define a subanalytic subset 
$TX|_S$ of $TX$ by $TX|_S \coloneq S\times_X TX \subset TX$. Let 
$h \colon TX \times_X T^\ast X \longrightarrow \mathbb{R}$ be the 
canonical pairing. 
Then the following subset of $TX|_S \times_X G^\ast(S)$
\begin{equation}
K_S \coloneq h^{-1} (\mathbb{R}\setminus \{ 0\}) \cap (TX|_S \times_X G^\ast(S))
\quad \subset TX|_S \times_X G^\ast(S)
\end{equation}
is subanalytic in $TX\times_X T^\ast X$.
Indeed, $TX|_S \times_X G^\ast(S)$ is the inverse image of the subanalytic 
subset $TX|_S \times G^\ast(S) \subset TX \times T^\ast X$ by the closed 
embedding $TX\times_X T^\ast X \longhookrightarrow TX\times T^\ast X$. 
Moreover, since the projection $G^\ast(S) \longrightarrow X$ is proper, 
the morphism $q \colon TX|_S \times_X G^\ast(S) \longrightarrow 
TX$ induced by it is also proper and hence 
$q(K_S) \subset TX$ is subanalytic in $TX$. 
On the other hand, we can easily see that 
\begin{equation}
TS=TX|_S \setminus q(K_S) \quad \subset TX. 
\end{equation}
This implies that $TS$ is subanalytic in $TX$.
\end{proof}

\begin{proof}
Let us start the proof of Lemma \ref{ext-lem}. 
For $S\in \mathcal{S}$ such that $S \subset \{f\geq 0\} = \{x\in X \mid f(x)\geq 0\}$,
we set $T_S^\ast (X/J) \coloneq \{ ((x;\xi), a) \in T^\ast X_0 \times J \mid  
(x,a) \in S \textrm{ and }\xi (v)=0 \textrm{ for all } v \in T_x (S\cap X_a)\} 
\subset T^\ast(X/(-\varepsilon, \varepsilon))$. 
Then we can easily see that 
\begin{equation}\label{eq-Xi}
    \Xi = \bigcup_{\substack{S\in \mathcal{S}\\ S\subset \{f\geq 0\}}} T_S^\ast (X/J).
\end{equation} 
Since $\mathcal{S}$ is locally finite, the right hand side of (\ref{eq-Xi}) 
is a locally finite union. 
Hence it suffices to show that 
$T_S^\ast (X/J) \subset T^\ast(X/(-\varepsilon, \varepsilon))$
is subanalytic in $T^\ast(X/(-\varepsilon, \varepsilon))$ 
for any stratum $S \in \mathcal{S}$ contained in $\{ f\geq0\}$.
For $S\in \mathcal{S}$ such that $S \subset \{f= 0\} = X_0$ 
this follows from \cite[Proposition 8.3.1]{KS90}. 
So we treat only $S\in \mathcal{S}$ such that $S \subset \{f > 0\}$. 
Fix such $S\in \mathcal{S}$ and 
consider the following two subsets of $T^\ast (X/(-\varepsilon, \varepsilon))$:
\begin{align}
    T^\ast(X/(-\varepsilon, \varepsilon))|_S \coloneq 
    S\times_X T^\ast (X/(-\varepsilon, \varepsilon)) \quad &\subset 
T^\ast (X/(-\varepsilon, \varepsilon)), \\
    Z_S \coloneq T^\ast(X/(-\varepsilon, \varepsilon))|_S \setminus 
T_S^\ast (X/J) \quad &\subset T^\ast (X/(-\varepsilon, \varepsilon)).
\end{align}
Then it is easy to see that 
$T^\ast(X/(-\varepsilon, \varepsilon))|_S$ is subanalytic in 
$T^\ast (X/(-\varepsilon, \varepsilon))$ and 
\begin{equation}
    T_S^\ast (X/J) =( T^\ast(X/(-\varepsilon, \varepsilon))|_S) \setminus Z_S.
\end{equation}
Thus we have only to show that 
$Z_S \subset T^\ast(X/(-\varepsilon, \varepsilon))$ is 
subanalytic in $T^\ast(X/(-\varepsilon, \varepsilon))$.
For this purpose,
first we define a real analytic subbundle 
$T(X/(-\varepsilon, \varepsilon)) \subset TX$
of $TX$ by
\begin{equation}
    T(X/(-\varepsilon, \varepsilon)) \coloneq \mathrm{Ker} (df) \quad \subset TX.
\end{equation}
The problem being local, as in the proof of Lemma \ref{lem-tangent},
we can take a sphere bundle $G(T(X/(-\varepsilon, \varepsilon))) 
\subset T(X/(-\varepsilon, \varepsilon)) \subset TX$ in it. 
Then it follows from Lemma \ref{lem-tangent} that 
$G(S)\coloneq TS \cap G(T(X/(-\varepsilon, \varepsilon))) 
\subset T(X/(-\varepsilon, \varepsilon))$
is subanalytic in $T(X/(-\varepsilon, \varepsilon))$. 
Let us consider the fiber product 
$T(X/(-\varepsilon, \varepsilon)) \times_X T^\ast (X/(-\varepsilon, \varepsilon))$ 
and its canonical pairing 
$g \colon T(X/(-\varepsilon, \varepsilon)) \times_X T^\ast (X/(-\varepsilon, 
\varepsilon)) \longrightarrow \mathbb{R}$ induced by the one  
$TX \times_X T^\ast X \longrightarrow \mathbb{R}$. 
Note that $g$ is well-defined, non-degenerate and real analytic.
Thus the following subset of $G(S)\times_X T^\ast(X/(-\varepsilon, \varepsilon))$
\begin{align}
    W_S &\coloneq g^{-1}(\mathbb{R} \setminus \{ 0\}) \cap (G(S)\times_X 
T^\ast(X/(-\varepsilon, \varepsilon)))\\
    & =\{ (x; v,\xi) \in G(S)\times_X T^\ast(X/(-\varepsilon, \varepsilon)) 
\mid g(x;v,\xi)\neq 0\} \subset G(S)\times_X T^\ast(X/(-\varepsilon, \varepsilon))
\end{align}
is subanalytic in 
$T(X/(-\varepsilon, \varepsilon)) \times_X T^\ast (X/(-\varepsilon, \varepsilon))$.
Indeed, $G(S)\times_X T^\ast(X/(-\varepsilon, \varepsilon))$
is the inverse image of the subanalytic subset $G(S)\times T^\ast(X/(
-\varepsilon, \varepsilon)) \subset T(X/(-\varepsilon, \varepsilon)) 
\times T^\ast (X/(-\varepsilon, \varepsilon))$
by the closed embedding $T(X/(-\varepsilon, \varepsilon)) \times_X 
T^\ast (X/(-\varepsilon, \varepsilon)) \longhookrightarrow T(X/(-\varepsilon, 
\varepsilon)) \times T^\ast (X/(-\varepsilon, \varepsilon))$. 
Moreover, since the projection $G(T(X/(-\varepsilon, \varepsilon))) 
\longrightarrow X$ is proper,
the morphism $p \colon G(T(X/(-\varepsilon, \varepsilon))) 
\times_X T^\ast(X/(-\varepsilon, \varepsilon)) \longrightarrow 
T^\ast (X/(-\varepsilon, \varepsilon))$ 
induced by it is also proper. This implies that $p(W_S) \subset 
T^\ast (X/(-\varepsilon, \varepsilon))$ is subanalytic in 
$T^\ast (X/(-\varepsilon, \varepsilon))$. We can easily see that 
\begin{equation}
p(W_S)=Z_S \quad \subset T^\ast (X/(-\varepsilon, \varepsilon)).
\end{equation}
Therefore, $Z_S$ is subanalytic in $T^\ast(X/(-\varepsilon, \varepsilon))$. 
This completes the proof.
\end{proof}

\end{document}